\documentclass[a4paper,11pt]{amsart}
\usepackage[colorlinks, linkcolor=blue,anchorcolor=Periwinkle,
    citecolor=blue,urlcolor=Emerald]{hyperref}
\usepackage[all]{xy}
\SelectTips{cm}{}

\usepackage{graphicx}
\usepackage{psfrag}

\usepackage{mathtools}
\usepackage{tikz}
\usepackage{tikz-cd}
\usetikzlibrary{calc}
\usepackage{adjustbox}
\usepackage{extarrows}

\usepackage{amssymb}
\usetikzlibrary{decorations.pathreplacing}
\usetikzlibrary{matrix,arrows}

\usepackage{enumitem}

\usepackage{stmaryrd}

\usepackage{mathdots}

\newcommand{\cf}{{cf.}\ }
\newcommand{\eg}{{e.g.}\ }
\newcommand{\ko}{\: , \;}
\newcommand{\ul}[1]{\underline{#1}}
\newcommand{\ol}[1]{\overline{#1}}

\numberwithin{equation}{section}

\newtheorem{classification-theorem}[subsection]{Classification Theorem}
\newtheorem{decomposition-theorem}[subsection]{Decomposition Theorem}
\newtheorem{proposition-definition}[subsection]{Proposition-Definition}
\newtheorem{periodicity-conjecture}[subsection]{Periodicity Conjecture}

\newtheorem{theorem}{Theorem}
\numberwithin{theorem}{section}
\newtheorem{thmx}{Theorem}

\newtheorem{lemma}[theorem]{Lemma}
\newtheorem{proposition}[theorem]{Proposition}
\newtheorem{propx}[thmx]{Proposition}
\newtheorem{corollary}[theorem]{Corollary}
\newtheorem{corx}[thmx]{Corollary}

\theoremstyle{definition}
\newtheorem{definition}[theorem]{Definition}
\theoremstyle{plain}

\theoremstyle{definition}
\newtheorem{example}[theorem]{Example}
\theoremstyle{plain}
\newtheorem{remark}[theorem]{Remark}
\newtheorem{assumption}[theorem]{Assumption}
\newtheorem{notation}[theorem]{Notation}

\newcommand{\reminder}[1]{}

\renewcommand{\mod}{\mathrm{mod}\,}

\newcommand{\CM}{\mathrm{CM}\,}

\newcommand{\per}{\mathrm{per}\,}

\newcommand{\add}{\mathrm{add}\,}

\newcommand{\op}{^{op}}

\newcommand{\Z}{\mathbb{Z}}
\newcommand{\N}{\mathbb{N}}
\newcommand{\Q}{\mathbb{Q}}

\newcommand{\iso}{\xrightarrow{_\sim}}

\newcommand{\id}{\mathbf{1}}

\newcommand{\Hom}{\mathrm{Hom}}

\newcommand{\RHom}{\mathrm{RHom}}

\newcommand{\Ext}{\mathrm{Ext}}

\newcommand{\End}{\mathrm{End}}
\newcommand{\rad}{\mathrm{rad}}

\newcommand{\ten}{\otimes}
\newcommand{\lten}{\overset{\mathrm{L}}{\ten}}

\newcommand{\cd}{{\mathcal D}}

\newcommand{\ct}{{\mathcal T}}

\newcommand{\Ga}{\Gamma}

\newcommand{\Si}{\Sigma}
\newcommand{\si}{\sigma}

\renewcommand{\phi}{\varphi}

\newcommand{\sequal}{\!\!\!=\!\!\!}
\newcommand{\sminus}{\!\!\!-\!\!\!}
\newcommand{\bigbullet}{\mathrel{\raisebox{-0.3ex}{\scalebox{1.2}{$\bullet$}}}}
\renewcommand{\bigcirc}{\mathrel{\raisebox{-0.3ex}{\scalebox{1.2}{$\circ$}}}}
\newcommand{\smallodot}{\mathrel{{\scalebox{0.7}{$\odot$}}}}
\newcommand{\smallcircledast}{\mathrel{{\scalebox{0.7}{$\circledast$}}}}

\renewcommand{\tilde}[1]{\widetilde{#1}}

\begin{document}

\date{\today}

\title[Tilting theory for hypersurface singularities of dimension one]{Tilting theory for\\[0.15cm]
hypersurface singularities of dimension one}

\author{Osamu Iyama}
\address{Graduate School of Mathematical Sciences, The University of Tokyo, 3-8-1 Komaba \linebreak Meguro-ku Tokyo 153-8914, Japan}
\email{iyama@ms.u-tokyo.ac.jp}
\urladdr{https://www.ms.u-tokyo.ac.jp/~iyama/}

\author{Junyang Liu}
\address{School of Mathematical Sciences, University of Science and Technology of China, Hefei 230026, China}
\address{Graduate School of Mathematical Sciences, The University of Tokyo, 3-8-1 Komaba \linebreak Meguro-ku Tokyo 153-8914, Japan}
\email{liuj@imj-prg.fr}
\email{liu@ms.u-tokyo.ac.jp}
\urladdr{https://webusers.imj-prg.fr/~junyang.liu}

\begin{abstract}
Any $\N$-graded commutative Gorenstein ring $R$ of Krull dimension one with $R_0$ a field admits a standard silting object $V$ in the stable category $\ul{\CM \!}_0^\Z R$, and the object $V$ is tilting if and only if the $a$-invariant $a$ is non-negative, as shown by Buchweitz, the first author, and Yamaura. In this article, under the additional assumption that $R$ is a hypersurface singularity, we prove that endomorphism algebra of $V$ is Iwanaga--Gorenstein of self-injective dimension at most $2$, and we give its explicit presentation in terms of a quiver with relation. In the case of where $a$ is negative, we prove that the dg endomorphism algebra of $V$ is Gorenstein, and we give its explicit presentation in terms of a dg path algebra. We explain our results by several examples including numerical semigroup algebras generated by two elements. Moreover, for each finite and countable Cohen--Macaulay representation type, we include the Auslander--Reiten quiver of the category $\CM\!_0^\Z R$ with the position of the standard silting object. As a step of the proof of our results, we give a characterization of Gorensteinness of homologically finite dg algebras in terms of Serre functors.
\end{abstract}




\maketitle

\vspace*{-1cm}
\tableofcontents

\section{Introduction}

The study of maximal Cohen--Macaulay modules is one of the central subjects in commutative algebra and representation theory, \cf \cite{Auslander78, CurtisReiner81, Yoshino90, Simson92, LeuschkeWiegand12}. It has a strong connection with algebraic geometry, singularity theory, mathematical physics and a number of important results are obtained by applying tilting theory for singularity categories of commutative and non-commutative Gorenstein rings, \cf~\cite{KajiuraSaitoTakahashi09, FutakiUeda11, KussinLenzingMeltzer13, HerschendIyamaMinamotoOppermann23, MoriUeyama16, LuZhu21, BuchweitzIyamaYamaura20, Iyama18, HiranoOuchi23, HaniharaIyama22, KimuraMinamotoYamaura25, IyamaKimuraUeyama24}. For example, a milestone of Cohen--Macaulay representation theory is Buchweitz--Greuel--Schreyer's classification of commutative complete local Gorenstein rings of finite Cohen--Macaulay representation type as $ADE$ singularities, see \cite{Yoshino90, LeuschkeWiegand12}. This result gives a commutative counter part of famous Gabriel's classification of quivers of finite representation type as Dynkin quivers, and the connection of these two classifications is given by tilting theory, see~\eg \cite{GeigleLenzing91, KajiuraSaitoTakahashi07, BuchweitzIyamaYamaura20} and Section~\ref{section:Auslander--Reiten quivers} below.

For a commutative Gorenstein ring $R$, the category $\CM R$ of maximal Cohen--Macaulay modules is a Frobenius category. Its stable category is triangle equivalent to the singularity category associated with $R$ introduced by Buchweitz \cite{Buchweitz21} and Orlov \cite{Orlov04}. The category $\CM R$ has a subcategory $\CM\!_0 R$ which behaves much nicer than $\CM R$ since it enjoys Auslander--Reiten--Serre duality. For a $\Z$-graded commutative Gorenstein ring $R$, we consider the $\Z$-graded version $\CM\!_0 ^{\Z} R$. It is also a Frobenius category. When $R$ is $\N$-graded and of Krull dimension one such that $R_0$ is a field, Buchweitz--Iyama--Yamaura proved \cite{BuchweitzIyamaYamaura20} that the stable category $\ul{\CM \!}_0^\Z R$ admits a silting object (which we call the {\it standard silting object}) and it is a tilting object if and only if either $R$ is regular or the $a$-invariant of $R$ is non-negative.

It is important to study the dg(=differential graded) endomorphism algebra of the standard silting object because it controls the triangle structure of the stable category $\ul{\CM \!}_0^\Z R$. Hypersurface singularities $R=k[x, y]/(f)$ of Krull dimension one provide an important class of $\N$-graded commutative Gorenstein rings above.

\begin{assumption} \label{ass:setting}
Let $k$ be a field and $k[x, y]$ a $\Z$-graded polynomial algebra such that $x$ and $y$ are of positive degrees $m$ respectively $n$. Let $f$ be a homogeneous element of $k[x, y]$ and denote $R=k[x, y]/(f)$.
\end{assumption}

In this case, we prove the following homological property of the endomorphism algebra and describe the algebra in terms of an explicit quiver with relation.

\begin{thmx}[see Theorem~\ref{thm:non-negative} for details] \label{thm:A}
Under Assumption~\ref{ass:setting}, suppose that we have $\deg f\geq m+n$. Denote the endomorphism algebra of the standard tilting object $V$ in $\ul{\CM \!}_0^\Z R$ by $\Ga$.
\begin{itemize}
\item[a)] The algebra $\Ga$ is finite-dimensional and Iwanaga--Gorenstein of self-injective dimension less than or equal to $2$. Moreover, its global dimension is finite if and only if it is less than or equal to $2$ if and only if $R$ is reduced.
\item[b)] The algebra $\Ga$ admits a presentation by a quiver with relation given in Definition~\ref{def:quivers with relations}.
\end{itemize}
\end{thmx}

This is a significant generalization of a result in \cite{BuchweitzIyamaYamaura20} for the case of standard grading $\deg x=\deg y=1$. Explicitly, we simultaneously generalize this setup in two directions: From the standard grading to arbitrary positive grading and from algebraically closed fields to arbitrary fields. The generality of Theorem~\ref{thm:A} for arbitrary grading is important because it includes many notable examples, \eg the $ADE$ singularities, the $T_{pq}$ singularities and the numerical semigroup algebras generated by two elements as follows.

\begin{corx}[=Corollary~\ref{cor:numerical semigroup algebra}] \label{cor:B}
Let $S\subset \N$ be a cofinite proper submonoid which is generated by two elements. Let $k[S]\subset k[x]$ be the numerical semigroup algebra. Then the global dimension of the endomorphism algebra $\Ga_S$ of the standard tilting object $V_S$ in $\ul{\CM \!}^\Z k[S]$ is less than or equal to $2$.
\end{corx}

The condition $\deg f\geq m+n$ in Theorem~\ref{thm:A} is satisfied except for the case where $R$ is regular or $f$ equals $x^{n_x}$ with $n_x>1$, \cf Lemma~\ref{lem:negative}. For the latter case, the dg endomorphism algebra of the standard silting object $V$ is described in Theorem~\ref{thm:C}.

\begin{thmx}[see Theorem~\ref{thm:negative} for details] \label{thm:C}
Under Assumption~\ref{ass:setting}, suppose that we have $0<\deg f<m+n$ and that $R$ is not regular. Denote the dg endomorphism algebra of the standard silting object $V$ in $\ul{\CM \!}_0^\Z R$ by $\Ga$.
\begin{itemize}
\item[a)] The dg algebra $\Ga$ is proper and Gorenstein (see Section~\ref{section:homologically finite dg algebras} for the definitions).
\item[b)] The dg algebra $\Ga$ admits a presentation by a dg path algebra given in Definition~\ref{def:dg path algebras}.
\end{itemize}
\end{thmx}

The article is organized as follows: In Section~\ref{section:preliminaries}, we recall $\Z$-graded commutative Gorenstein rings and the category $\CM \!_0^\Z R$.

In Section~\ref{section:homologically finite dg algebras}, we recall homologically finite and Gorenstein dg algebras and prove the following results.

\begin{propx} [=Proposition~\ref{prop:Serre duality} and Corollary~\ref{cor:Gorensteinness}] \label{prop:D}
Let $A$ be a homologically finite dg algebra (see Section~\ref{section:homologically finite dg algebras} for the definition).
\begin{itemize}
\item[a)] We have a canonical isomorphism
\[
\begin{tikzcd}
D\RHom_A(X, Y)\arrow[no head]{r}{\sim} & \RHom_A(Y, X\lten_A DA)
\end{tikzcd}
\]
bifunctorial in $X\in \per A$ and $Y\in \cd(A)$.
\item[b)] The dg algebra $A$ is Gorenstein if and only if the triangulated category $\per A$ has a Serre functor.
\end{itemize}
\end{propx}

In Section~\ref{section:our results}, we consider hypersurface singularities with arbitrary positive grading. In Section~\ref{ss:structure theorem for non-negative a-invariants}, we introduce a quiver with relation and state the main results in the case where the $a$-invariant of $R$ is non-negative. In Section~\ref{ss:numerical semigroup algebras}, we study the case of numerical semigroup algebras and state the results. Then we give examples in the generality of non-standard grading, non-reduced ring and non-algebraically closed field to illustrate the construction of the quivers with relations, \cf Section~\ref{ss:examples of non-negative}. In Section~\ref{ss:structure theorem for negative a-invariants}, we introduce a dg path algebra and state the main results in the case where the $a$-invariant of $R$ is negative. Then we give an example to illustrate the construction of the dg path algebras, \cf Section~\ref{ss:an example of negative}.

In Section~\ref{section:proof of our results}, we prove our results. In Section~\ref{ss:structures of the morphism spaces}, we discuss the structures of the morphism spaces, radical morphism spaces and the square of radical morphism spaces between direct summands of $V$. In Section~\ref{ss:proof of non-negative}, we use them to prove Theorem~\ref{thm:A}. In Section~\ref{ss:proof of negative}, we prove Theorem~\ref{thm:C}.

In Section~\ref{section:Auslander--Reiten quivers}, we draw the Auslander--Reiten quivers of the categories $\CM\!_0 ^{\Z} R$ for finite and countable Cohen--Macaulay representation types in the $\Z$-graded case. Our Auslander--Reiten quivers of $\CM\!_0^\Z R$ are displayed in a standard way following the theory of translation quivers due to Gabriel and Riedtmann, \eg \cite{Riedtmann80, AssemSimsonSkowronski06}. Then we illustrate the explicit positions of the standard silting objects $V$ in the Auslander--Reiten quivers.

In Appendix~\ref{appendix:dg resolutions of self-injective Nakayama algebras}, we give cofibrant dg replacements of self-injective Nakayama algebras. In Section~\ref{ss:Ainfty-algebras}, we recall $A_{\infty}$-algebras. In Section~\ref{ss:cofibrant dg replacements of self-injective Nakayama algebras}, we prove our result.

\begin{notation}
The following notation is used throughout the article: We let $k$ be a field. For a $k$-vector space $V$, we denote its $k$-dual space $\Hom_k(V,k)$ by $DV$. Algebras have units and morphisms of algebras preserve the units. Modules are unital right modules. For a ring $R$, we denote the category of finitely generated $R$-modules by $\mod R$. The degree of a homogeneous element $a$ in a $\Z$-graded vector space is denoted by $\deg a$. We denote the shift functor of $\Z$-graded vector spaces by $(1)$ so that we have $V(1)_n =V_{n+1}$. We denote the truncation of a $\Z$-graded vector space $V$ in non-negative degrees by $V_{\geq 0}$.
\end{notation}

\subsection*{Acknowledgments}
The first-named author is supported by JSPS Grant-in-Aid for Scientific Research (B) 23K22384.

\section{Preliminaries} \label{section:preliminaries}

A commutative noetherian ring of finite Krull dimension is {\it Gorenstein} if it has finite injective dimension as a module over itself. For a Gorenstein ring $R$, a finitely generated $R$-module $M$ is {\it maximal Cohen--Macaulay} if $\Ext_R ^i(M, R)$ vanishes for all positive integers $i$. For a $\Z$-graded Gorenstein ring $R$, we write $\CM \!^{\Z} R$ for the category whose objects are $\Z$-graded $R$-modules which are maximal Cohen--Macaulay and morphisms are $R$-homomorphisms which are homogeneous of degree $0$. The category $\CM \!_0^\Z R$ is defined to be the full subcategory of $\CM \!^{\Z} R$ whose objects are the $\Z$-graded $R$-modules $M$ such that for each non-maximal prime ideal $\mathfrak{p}$ of $R$, the $R_{\frak{p}}$-modules $M_{\frak{p}}$ are projective. Both categories $\CM \!^{\Z} R$ and $\CM \!_0^\Z R$ are Frobenius. They coincide if and only if $R$ has at worst isolated singularities. If $R$ is an $\N$-graded Gorenstein ring of Krull dimension one with $R_0=k$ a field, then there exists a unique integer $a$ such that the $\Z$-graded $R$-module $\Ext_R ^1(k, R(a))$ is isomorphic to $k$. We call $a$ the {\it $a$-invariant} of $R$.

\section{Homologically finite dg algebras and Serre functors} \label{section:homologically finite dg algebras}

Let $A$ be a dg algebra. We write $\cd(A)$ for the (unbounded) derived category of $A$. Its thick subcategory generated by the free dg $A$-module of rank one is the {\em perfect derived category} $\per A$. It consists of compact objects in $\cd(A)$. Recall that $A$ is {\em proper} if its underlying complex is perfect over $k$. Equivalently, its homology $H^p(A)$ is finite-dimensional for all integers $p$ and vanishes for all $|p|\gg 0$. The dg algebra $A$ is {\em homologically finite} if its homology $H^p(A)$ is finite-dimensional for all integers $p$. Following Section~0 of \cite{Jin20}, a dg algebra $A$ is {\em Gorenstein} if the thick subcategories of $\cd(A)$ generated by $A$ respectively $DA$ coincide. The following proposition is stated in Section~I.4.6 of \cite{Happel88} and Section~10.4 of \cite{Keller94}, and the Serre duality in Lemma~4.1 of \cite{Keller08d} is its dual version. We include our proof for the convenience of the reader.

\begin{proposition} \label{prop:Serre duality}
Let $A$ be a homologically finite dg algebra. Then we have a canonical isomorphism
\[
\begin{tikzcd}
D\RHom_A(X, Y)\arrow[no head]{r}{\sim} & \RHom_A(Y, X\lten_A DA)
\end{tikzcd}
\]
bifunctorial in $X\in \per A$ and $Y\in \cd(A)$.
\end{proposition}

\begin{proof}
Since the dg $A$-module $X$ is perfect, we have the isomorphism
\[
\begin{tikzcd}
Y\lten_A\RHom_A(X, A)\arrow{r}{\sim} & \RHom_A(X, Y) \: .
\end{tikzcd}
\]
It follows that we have the composed isomorphism
\begin{equation} \label{eq:isomorphism 1}
\begin{tikzcd}
D\RHom_A(X, Y) \arrow{r}{\sim} & D(Y\lten_A \RHom_A(X, A)) \arrow[no head]{r}{\sim} & \RHom_A(Y, D\RHom_A(X, A)) \: .
\end{tikzcd}
\end{equation}
Since the dg algebra $A$ is homologically finite, we have the isomorphism $A\iso DDA$ in $\cd(A)$. It follows that we have the composed isomorphism
\[
\begin{tikzcd}
D(X\lten_A DA) \arrow[no head]{r}{\sim} & \RHom_A(X, DDA) & \RHom_A(X, A) \arrow[swap]{l}{\sim} \: .
\end{tikzcd}
\]
If we apply the functor $D$ to it, we obtain the isomorphism
\begin{equation} \label{eq:isomorphism 2}
X\lten_A DA\xlongrightarrow{_\sim} D\RHom_A(X, A) \: .
\end{equation}
Combining the isomorphisms~(\ref{eq:isomorphism 1}) and (\ref{eq:isomorphism 2}) we obtain the desired isomorphism. Its bifunctoriality follows from those of the above isomorphisms.
\end{proof}

\begin{corollary} \label{cor:Gorensteinness}
Let $A$ be a homologically finite dg algebra. Then it is Gorenstein if and only if the triangulated category $\per A$ has a Serre functor.
\end{corollary}

\begin{proof}
For the necessity, by Propsition~\ref{prop:Serre duality}, the functor $?\lten_A DA$ is a Serre functor of the triangulated category $\per A$. We now prove the sufficiency. By Proposition~\ref{prop:Serre duality}, we have the natural isomorphism
\[
\begin{tikzcd}
\Hom_{\per A}(A, ?) \arrow[no head]{r}{\sim} & D\Hom_{\per A}(?, DA) \: .
\end{tikzcd}
\]
On the other hand, since the triangulated category $\per A$ has a Serre functor, we denote it by $\mathbb{S}$ and so that we have the natural isomorphism
\[
\begin{tikzcd}
\Hom_{\per A}(A, ?) \arrow{r}{\sim} & D\Hom_{\per A}(?, \mathbb{S}A) \: .
\end{tikzcd}
\]
By Yoneda Lemma, we deduce that $DA\iso \mathbb{S}A$ lies $\per A$ and hence the thick subcategory of $\cd(A)$ generated by it coincides with $\per A$.
\end{proof}

\section{Our results for hypersurface singularities} \label{section:our results}

We use the following setting throughout this section. Let $k$ be a field and $k[x, y]$ a $\Z$-graded polynomial algebra such that $x$ and $y$ are of positive degrees $m$ respectively $n$. Without loss of generality, we may and will assume that $m$ and $n$ are coprime and that we have $m\leq n$. Let $f$ be a homogeneous element of $k[x, y]$ of positive degree. Without loss of generality, we may and will assume that it is monic. Then we decompose $f$ as
\[
f=x^{n_x} y^{n_y} \prod_{j=1}^q f_j ^{n_j} \: ,
\]
where $n_x$, $n_y$ and $n_j$ are non-negative integers and $f_j$ are distinct monic irreducible polynomials which are distinct from $x$ and $y$. The $\N$-graded commutative ring
\[
R=k[x, y]/(f)
\]
is Gorenstein with $a$-invariant $a=\deg f-m-n$. Denote the stable category $\ul{\CM \!}_0^\Z R$ by $\ct$.

\begin{proposition} \label{prop:properness and Gorensteinness}
Let $X$ be an object in $\ct$ such that the thick subcategory generated by it coincides with $\ct$. Then the dg algebra $\RHom_{\ct}(X, X)$ is a proper Gorenstein dg algebra.
\end{proposition}

\begin{proof}
The Gorensteinness of the dg algebra $A=\RHom_{\ct}(X, X)$ follows from Proposition~4.6 of \cite{BuchweitzIyamaYamaura20} and the sufficiency in Corollary~\ref{cor:Gorensteinness}. We now prove that the dg algebra $A$ is proper. Since each homogeneous component of $R$ is finite-dimensional, the category of finitely generated graded $R$-modules is $\Hom$-finite and hence so is $\ct$. It follows that $A$ is homologically finite. So it suffices to show its homology is concentrated in bounded degrees. Denote shift functor of $\ct$ by $\Si$ and the object $X\oplus \Si X$ by $Y$. Since the graded $R$-module $Y$ is finitely generated, the morphism space $\Hom_R(Y, Y(p))_0$ vanishes for all $p\ll 0$ and hence so does $\Hom_{\ct}(Y, Y(p))$. By Proposition~4.6 of \cite{BuchweitzIyamaYamaura20}, the morphism space $\Hom_{\ct}(Y, Y(p))$ also vanishes for all $p\gg 0$. Since $R$ is a hypersurface, by the graded version of the sufficiency in part~(ii) of Theorem~6.1 of \cite{Eisenbud80}, we have $\Si^2 \simeq (\deg f)$. Therefore, the morphism space $\Hom_{\ct}(Y, \Si^{2p}Y)$ vanishes for all $|p|\gg 0$. Then the assertion follows.
\end{proof}

\subsection{The structure theorem for non-negative $a$-invariants} \label{ss:structure theorem for non-negative a-invariants}

Suppose that $a$ is non-negative. For any positive integer $i$, we define graded $R$-modules
\[
V^i=R(i)_{\geq 0} \ko V^{[1, a]}=\bigoplus_{i=1}^a V^i \: .
\]
Let $K$ be the graded total quotient algebra of $R$. It decomposes as $K=K^x \times K^y \times \prod_{j=1}^q K^j$, where we denote $K^x=k[x, y^{\pm 1}]/(x^{n_x})$, $K^y=k[x^{\pm 1}, y]/(y^{n_y})$, and $K^j=k[x^{\pm 1}, y^{\pm 1}]/(f_j ^{n_j})$. Denote by $1_x$, $1_y$ and $1_j$ the corresponding central idempotents of $K$. If we have $i\geq a+1$, then $V^i$ is a graded $K$-module and we define graded $R$-modules
\[
V^{i, x}=1_x \cdot V^i \iso K^x (i)_{\geq 0} \ko V^x=\bigoplus_{i=a+1}^{a+n}V^{i, x} \: ,
\]
\[
V^{i, y}=1_y \cdot V^i \iso K^y (i)_{\geq 0} \ko V^y=\bigoplus_{i=a+1}^{a+m}V^{i, y} \: ,
\]
\[
V^{i, f_j}=1_j \cdot V^i \iso K^j (i)_{\geq 0} \ko V^{f_j}=V^{a+1, f_j} \mbox{ for }1\leq j\leq q \: .
\]
By Theorem~1.4 of \cite{BuchweitzIyamaYamaura20}, the stable category $\ct =\ul{\CM \!}_0^\Z R$ contains a standard tilting object $V=V^{[1, a]}\oplus V^x \oplus V^y \oplus \bigoplus_{j=1}^q V^{f_j}$.

\begin{lemma}
We have an isomorphism
\[
\End_R(V)_0 \simeq
\begin{bmatrix}
_R(V^{[1, a]})_0 & 0 & 0 & 0 & \cdots & 0 \\
_R(V^{[1, a]}, V^x)_0 & _R(V^x)_0 & 0 & 0 & \cdots & 0 \\
_R(V^{[1, a]}, V^y)_0 & 0 & _R(V^y)_0 & 0 & \cdots & 0 \\
_R(V^{[1, a]}, V^{f_1})_0 & 0 & 0 & _R(V^{f_1})_0 & \cdots & 0 \\
\vdots & \vdots & \vdots & \vdots & \ddots & \vdots \\
_R(V^{[1, a]}, V^{f_q})_0 & 0 & 0 & 0 & \cdots & _R(V^{f_q})_0
\end{bmatrix}
\]
of algebras, where we write $_R(?, -)_0$ (respectively, $_R(?)$) for $\Hom_R(?, -)_0$ (respectively, $\End_R(?)_0$). The entries are described via the bijections
\begin{align*}
\Hom_R(V^i, V^{i'})_0 \iso & R_{i'-i}\quad \mbox{for all }1\leq i, i'\leq a \: , \\
\Hom_R(V^{i, x}, V^{i', x})_0 \iso & K^x_{i'-i}\quad \mbox{for all }a+1\leq i, i'\leq a+n \: , \\
\Hom_R(V^{i, y}, V^{i', y})_0 \iso & K^y_{i'-i}\quad \mbox{for all }a+1\leq i, i'\leq a+m \: , \\
\End_R(V^{a+1, f_j})_0 \iso & K^j_{0}\quad \mbox{for all }1\leq j\leq q \: , \\
\Hom_R(V^i, V^{i', x})_0 \iso & K^x_{i'-i}\quad \mbox{for all }1\leq i\leq a \ko a+1\leq i'\leq a+n \: , \\
\Hom_R(V^i, V^{i', y})_0 \iso & K^y_{i'-i}\quad \mbox{for all }1\leq i\leq a \ko a+1\leq i'\leq a+m \: , \\
\Hom_R(V^i, V^{a+1, f_j})_0 \iso & K^j_{a+1-i}\quad \mbox{for all }1\leq i\leq a \ko 1\leq j\leq q \: .
\end{align*}
\end{lemma}

\begin{proof}
The assertions follows from the proof of Lemma~\ref{lem:morphism}.
\end{proof}

To describe the endomorphism algebra $\End_{\ct}(V)$, we introduce the following quiver with relation $(Q, I_{r, s})$.

\begin{definition} \label{def:quivers with relations}
The set $Q_0$ of the vertices of $Q$ is the union
\[
\{i\mid 1\leq i\leq a\} \cup \{(i, x)\mid 0\leq i<n \} \cup \{(i, y)\mid 0\leq i<m\} \cup \{(0, j)\mid 1\leq j\leq q\} \: .
\]
Notice that the power of $x$ in each term of $f_j$ must be a multiple of $n$. Put $g_j (t_j)=f_j (t_j ^{\frac{1}{n}}, 1)$.

The set $Q_1$ of the arrows of $Q$ is the union
\begin{align*}
 & \{x=x_i\colon i \to i+m \mid 1\leq i\leq a-m\} \\
\cup & \{y=y_i\colon i \to i+n \mid 1\leq i\leq a-n\} \\
\cup & \{a_x=a_{i, x}\colon a-n+1+i \to (i, x)\mid \max \{n-a, 0\}\leq i<n\} \\
\cup & \{a_y=a_{i, y}\colon a-m+1+i \to (i, y)\mid \max\{m-a, 0\}\leq i<m\} \\
\cup & \{a_j=a_{i, j}\colon a-m+1+i \to (0, j)\mid  \max\{m-a, 0\}\leq i<m \ko 1\leq j\leq q\} \\
\cup & \{b_x=b_{i, x}\colon (i, x) \to (i+m, x)\mid 0\leq i<n \} \\
\cup & \{b_y=b_{i, y}\colon (i, y) \to (i+n, y)\mid 0\leq i<m \} \\
\cup & \{b_j=b_{0, j}\colon (0, j) \to (0, j)\mid 1\leq j\leq q\} \: ,
\end{align*}
where $(i, x)$ and $(i, y)$ denote $(i-\lfloor \frac{i}{n} \rfloor n, x)$ respectively $(i-\lfloor \frac{i}{m} \rfloor m, y)$ for all integers $i$.
We define $I$ to be the ideal of $kQ$ generated by
\[
b_x ^{n_x} \ko b_y ^{n_y} \ko g_j(b_j)^{n_j} \ko xy-yx \ko b_{i, x}a_{i, x}y^{\lfloor \frac{i+m}{n}\rfloor}-a_{i+m, x}x \ko b_{i, y}a_{i, y}x^{\lfloor \frac{i+n}{m}\rfloor}-a_{i+n, y}y
\]
To define other ideals, let us introduce the following notation. We choose integers $r$ and $s$ satisfying $rm+sn=1$. For any $1\leq j\leq q$, denote $N^j=k[b_{0, j}]/(g_j(b_{0, j})^{n_j})$. Let $F^j$ be the free $k[x, y]\ten_k N^j$-module with the basis $(a_{i, j})_{\max\{m-a, 0\}\leq i<m}$, where $N^j$ acts on the left and $k[x, y]$ acts on the right. Let $F^j_{xy}$ be its localization with respect to the multiplicative set $\{(xy)^i\mid i\in \N \} \subset k[x, y]$. Let $(F^j_{xy})_{r, s}$ be the $k[x^{\pm 1}, y^{\pm 1}]\ten_k N^j$-submodule of $F^j_{xy}$ generated by
\[
b_{0, j}a_{m-1, j}-a_{m-1, j}\frac{x^n}{y^m}\quad \mbox{and}\quad a_{i, j}-a_{i+1, j}x^r y^s \: .
\]
Denote the intersection of $F^j$ and $(F^j_{xy})_{r, s}$ by $F^j_{r, s}$. Let $\tilde{Q}$ be the infinite quiver obtain from $Q$ by adding vertices $i$ for $i\leq 0$ and arrows $x=x_i \colon i\to i+m$, $y=y_i \colon i\to i+n$ for $i\leq 0$. The vector space
\[
\bigoplus_{i\leq a}e_{(0, j)}(k\tilde{Q}/(xy-yx, g_j(b_{0, j})^{n_j}))e_i
\]
admits a natural $k[x, y]\ten_k N^j$-module structure such that there is a canonical isomorphism $\psi^j_{r, s}$ from $F^j$ to it which maps $a_{i, j}$ to $a_{i, j}$. Denote $\ol{I}^j_{r, s}=\psi^j_{r, s}(F^j_{r, s})\sum_{i=1}^a e_i$. We define $I^j_{r, s}$ to be the preimage of $\ol{I}^j_{r, s}$ under the canonical surjection map
\[
\begin{tikzcd}
[ampersand replacement=\&]
\bigoplus_{i=1}^a e_{(0, j)}(kQ)e_i \arrow{r} \& \bigoplus_{i=1}^a e_{(0, j)}(kQ/(xy-yx, g_j(b_{0, j})^{n_j}))e_i \: .
\end{tikzcd}
\]
The relation $I_{r, s}$ is the ideal of $kQ$ generated by $I$ and $I^j_{r, s}$, $1\leq j\leq q$.
\end{definition}

Recall that a (not necessarily commutative) noetherian ring {\it Iwanaga--Gorenstein} if it has finite injective dimension as both a left and a right module over itself.

\begin{theorem} \label{thm:non-negative}
Let $R$ and $(Q, I_{r, s})$ be as above. Denote the endomorphism algebra of the standard tilting object $V$ in $\ul{\CM \!}_0^\Z R$ by $\Ga$.
\begin{itemize}
\item[a)] The algebra $\Ga$ is isomorphic to $kQ/I_{r, s}$. In particular, the stable category $\ul{\CM \!}_0^\Z R$ is triangle equivalent to $\per kQ/I_{r, s}$.
\item[b)] The algebra $\Ga$ is finite-dimensional and Iwanaga--Gorenstein of self-injective dimension less than or equal to $2$. Moreover, its global dimension is finite if and only if it is less than or equal to $2$ if and only if $R$ is reduced.
\end{itemize}
\end{theorem}

The isomorphism in part~b) of Theorem~\ref{thm:non-negative} is induced from the surjective morphism $\phi_{r, s}\colon kQ\to \Ga$ of algebras given as follows. The images of the idempotents corresponding to the vertices $i$, $(i, x)$, $(i, y)$, and $(0, j)$ under $\phi_{r, s}$ are the idempotents of $\Ga$ corresponding to the indecomposable direct summands $V^i$, $V^{a+1+i, x}$, $V^{a+1+i, y}$, and $V^{a+1, j}$, respectively. The images of the elements $x_i$, $y_i$, $a_{i, x}$, $a_{i, y}$, $a_{i, j}$, $b_{i, x}$, $b_{i, y}$, and $b_{0, j}$ under $\phi_{r, s}$ are given as follows. The elements $\phi_{r, s}(x_i)$ are the morphisms
\[
R(i)_{\geq 0} \longrightarrow R(i+m)_{\geq 0}
\]
given by multiplying $x$. The elements $\phi_{r, s}(y_i)$ are the morphisms
\[
R(i)_{\geq 0} \longrightarrow R(i+n)_{\geq 0}
\]
given by multiplying $y$. The elements $\phi_{r, s}(a_{i, x})$ are the morphisms
\[
\begin{tikzcd}
R(a+1+i-n)_{\geq 0} \arrow{r} & K^x (a+1+i)_{\geq 0}
\end{tikzcd}
\]
given by multiplying $y$. The elements $\phi_{r, s}(a_{i, y})$ are the morphisms
\[
\begin{tikzcd}
R(a+1+i-m)_{\geq 0} \arrow{r} & K^y (a+1+i)_{\geq 0}
\end{tikzcd}
\]
given by multiplying $x$. The elements $\phi_{r, s}(a_{i, j})$ are the morphisms
\[
\begin{tikzcd}
R(a+1+i-m)_{\geq 0} \arrow{r} & K^j (a+1)_{\geq 0}
\end{tikzcd}
\]
given by multiplying $(x^r y^s)^{m-i}$. The elements $\phi_{r, s}(b_{i, x})$ are the morphisms
\[
\begin{tikzcd}
K^x (a+1+i)_{\geq 0} \arrow{r} & K^x (a+1+i+m-\lfloor \frac{i+m}{n}\rfloor n)_{\geq 0}
\end{tikzcd}
\]
given by multiplying $\frac{x}{y^{\lfloor \frac{i+m}{n}\rfloor}}$. The elements $\phi_{r, s}(b_{i, y})$ are the morphisms
\[
\begin{tikzcd}
K^y (a+1+i)_{\geq 0} \arrow{r} & K^y (a+1+i+n-\lfloor \frac{i+n}{m}\rfloor m)_{\geq 0}
\end{tikzcd}
\]
given by multiplying $\frac{y}{x^{\lfloor \frac{i+n}{m}\rfloor}}$. The elements $\phi_{r, s}(b_{0, j})$ are the morphisms
\[
K^j (a+1)_{\geq 0} \longrightarrow K^j (a+1)_{\geq 0}
\]
given by multiplying $\frac{x^n}{y^m}$. Then the kernel of $\phi_{r, s}$ coincides with $I_{r, s}$.

\begin{remark} \label{rk:well-defined}
The quiver $Q$ is uniquely determined by $R$ up to an isomorphism. But the relation $I_{r, s}$ and the morphism $\phi_{r, s}$ depend on choices of the pair $(r, s)$ in general.
\end{remark}

\subsection{Numerical semigroup algebras} \label{ss:numerical semigroup algebras}

For a cofinite submonoid $S\subseteq \N$, the stable category $\ct=\ul{\CM \!}^\Z k[S]$ contains a standard tilting object $V_S=\bigoplus_{i=1}^{a+1} k[S](i)_{\geq 0}$ given in Theorem~1.2 of \cite{BuchweitzIyamaYamaura20}. The following result is an application of Theorem~\ref{thm:non-negative}.

\begin{corollary} \label{cor:numerical semigroup algebra}
Let $S\subset \N$ be a cofinite proper submonoid which is generated by two elements. Let $k[S]\subset k[x]$ be the numerical semigroup algebra. Then the global dimension of the endomorphism algebra $\Ga_S$ of the standard tilting object $V_S$ in $\ul{\CM \!}^\Z k[S]$ is less than or equal to $2$.
\end{corollary}

\begin{proof}
Assume that $S$ is generated by $m$ and $n$. Since $m$ and $n$ are greater than $1$, the $a$-invariant $a=mn-m-n$ is non-negative. The standard tilting object $V_S$ is isomorphic to the $V$ in Section~\ref{ss:structure theorem for non-negative a-invariants}. Then the assertion follows by applying part~a) of Theorem~\ref{thm:non-negative} to $f=x^n-y^m$.
\end{proof}

We give a family of examples to illustrate quivers with relations $(Q, I_{r, s})$ in this case.

\begin{example} \label{ex:numerical semigroup algebra}
Let $k$ be an arbitrary field. For $R=k[x, y]/(x^{m+1}-y^m)$ with $\deg x=m\geq 2$ and $\deg y=m+1$, we have $a=m^2-m-1$. The quiver $Q$ is given as follows.
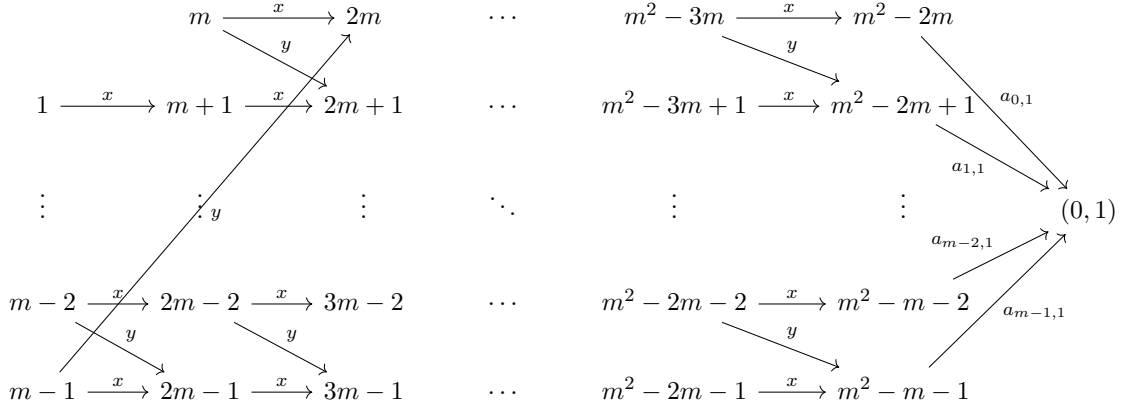
\[
\adjustbox{max width=\textwidth}{
\begin{tikzcd}
 & m \arrow{r}{x} \arrow{dr}{y} & 2m & \cdots & m^2-3m \arrow{r}{x} \arrow{dr}{y} & m^2-2m \arrow{ddr}{a_{0, 1}} & \\
1 \arrow{r}{x} & m+1 \arrow{r}{x} & 2m+1 & \cdots & m^2-3m+1 \arrow{r}{x} & m^2-2m+1 \arrow[swap]{dr}{a_{1, 1}} & \\
\vdots & \vdots & \vdots & \ddots & \vdots & \vdots & (0, 1) \\
m-2 \arrow{r}{x} \arrow{dr}{y} & 2m-2 \arrow{r}{x} \arrow{dr}{y} & 3m-2 & \cdots & m^2-2m-2 \arrow{r}{x} \arrow{dr}{y} & m^2-m-2 \arrow{ur}{a_{m-2, 1}} & \\
m-1 \arrow{r}{x} \arrow[swap]{uuuurr}{y} & 2m-1 \arrow{r}{x} & 3m-1 & \cdots & m^2-2m-1 \arrow{r}{x} & m^2-m-1 \arrow[swap]{uur}{a_{m-1, 1}} &
\end{tikzcd}
}
\]

The ideal $I$ is generated by $xy-yx$. For a choice of the pair $(r, s)$, the module $(F^1_{xy})_{r, s}$ over $k[x^{\pm 1}, y^{\pm 1}]\ten_k k$ is generated by $a_{m-1, 1}\frac{x^{m+1}-y^m}{y^m}$ and $a_{i, 1}-a_{i+1, 1}x^r y^s$. The ideal $\ol{I}^1_{r, s}$ is generated by $a_{i, 1}x-a_{i+1, 1}y$. Therefore, for any choice of $(r, s)$, the relation $I_{r, s}$ is the ideal generated by $xy-yx$ and $a_{i, 1}x-a_{i+1, 1}y$.

The surjective morphism $\phi_{r, s}\colon kQ\to \Ga_S$ of algebras is given as follows. The images of the idempotents corresponding to the vertices $i$ (respectively, $(0, 1)$) under $\phi_{r, s}$ are the idempotents of $\Ga_S$ corresponding to the indecomposable direct summands $V^i$ (respectively, $V^{m^2 - m, f_1}$). The images of the elements $x_i$, $y_i$, and $a_{i, 1}$ under $\phi_{r, s}$ are given as follows. The elements $\phi_{r, s}(x_i)$ are the morphisms
\[
R(i)_{\geq 0} \longrightarrow R(i+m)_{\geq 0}
\]
given by multiplying $x$. The elements $\phi_{r, s}(y_i)$ are the morphisms
\[
R(i)_{\geq 0} \longrightarrow R(i+m+1)_{\geq 0}
\]
given by multiplying $y$. The elements $\phi_{r, s}(a_{i, 1})$ are the morphisms
\[
\begin{tikzcd}
R(m^2-2m+i)_{\geq 0} \arrow{r} & (k[x^{\pm 1}, y^{\pm 1}]/(x^{m+1}-y^m))(m^2-m)_{\geq 0}
\end{tikzcd}
\]
given by multiplying $(x^r y^s)^{m-i}$. Then the kernel of $\phi_{r, s}$ coincides with $I_{r, s}$.
\end{example}

It is natural to ask whether, if a cofinite submonoid $S\subset \N$ is not generated by two elements, the global dimension of the endomorphism algebra $\Ga_S$ of the standard tilting object $V_S$ is also less than or equal to $2$. The following counter example in the case where $S$ is generated by three elements shows it is not the case.

\begin{example} \label{ex:three generators}
Let $k$ be an arbitrary field and $S\subset \N$ the submonoid generated by $15$, $21$, and $35$. 
Let $k[S]\subset k[x]$ be the numerical semigroup algebra. We have $a=139$. Denote the Gabriel quiver of $\Ga_S$ by $Q$. The set $Q_0$ of its vertices is $\{i\mid 1\leq i\leq 140\}$. The set $Q_1$ of its arrows is the union
\begin{align*}
 & \{x=x_i\colon i \to i+15\mid 1\leq i\leq 124\} \\
\cup & \{y=y_i\colon i \to i+21\mid 1\leq i\leq 118\} \\
\cup & \{z=z_i\colon i \to i+35\mid 1\leq i\leq 104\} \\
\cup & \{a=a_i\colon i+125 \to 140\mid 0\leq i<15\} \: .
\end{align*}
The arrows of $Q$ make $Q_0$ into a poset, and the algebra $\Ga_S$ is isomorphic to the incidence algebra of $Q_0$. We have the minimal projective resolution
\[
\begin{tikzcd}[ampersand replacement=\&]
0 \arrow{r} \& P_{140}^{\oplus 2}\arrow{rr}{
\begin{bsmallmatrix}
-ax^4 & ax^4 \\
0 & -2ax \\
ax & ax \\
a & 0 \\
0 & a
\end{bsmallmatrix}
} \& \& P_{72}\oplus P_{121}^{\oplus 2}\oplus P_{127}^{\oplus 2} \arrow{rrr}{
\begin{bsmallmatrix}
z & y^4 & 0 & x^6 & x^6 \\
y & z^2 & z^2 & 0 & 0 \\
x & 0 & 0 & z^2 & -z^2 \\
0 & 0 & x & -y & -y \\
0 & x & x & -y & y
\end{bsmallmatrix}
} \& \& \& P_{37}\oplus P_{51}\oplus P_{57}\oplus P_{106}^{\oplus 2}
\end{tikzcd}
\]
\[
\begin{tikzcd}[ampersand replacement=\&]
\arrow{rrrr}{
\begin{bsmallmatrix}
y & -z & 0 & x^6 & 0 \\
-x & 0 & z & -y^4 & y^4 \\
0 & x & -y & 0 & -z^2
\end{bsmallmatrix}
} \& \& \& \& P_{16}\oplus P_{22}\oplus P_{36} \arrow{r}{
\begin{bsmallmatrix}
x & y & z
\end{bsmallmatrix}
} \& P_1 \arrow{r} \& S_1 \arrow{r} \& 0
\end{tikzcd}
\]
of the simple $\Ga_S \op$-module $S_1$. The global dimension of $\Ga_S$ is $4$.
\end{example}

More generally, fix an integer $n>2$ and consider cofinite submonoids $S\subset \N$ whose minimal number of generators is $n$, and the endomorphism algebra $\Ga_S$ of the standard tilting object $V_S$. We believe that the global dimensions of $\Ga_S$ can be as large as possible.

\subsection{Examples of Theorem~\ref{thm:non-negative}} \label{ss:examples of non-negative}

The following examples serve to illustrate quivers with relations $(Q, I_{r, s})$ in the case of non-standard grading, non-reduced ring $R$, and non-algebraically closed field $k$.

\begin{example} \label{ex:ADE singularities}
The classification of $ADE$ singularities $R=k[x, y]/(f)$ is given in Section~\ref{section:Auslander--Reiten quivers}. It shows that almost all $ADE$ singularities are given by non-standard grading. In the case of type $E_7$, we have $f=y(x^3-y^2)$ with $\deg x=2$ and $\deg y=3$. So we have $a=4$. The quiver $Q$ is given as follows.
\[
\begin{tikzcd}[row sep=1.2em]
1 \arrow{rr}{x} \arrow{ddrrr}{y} & & 3 \arrow{rr}{a_{0, y}} \arrow{drr}{a_{0, 1}} & & (0, y) & \\
 & & & & (0, 1) & \\
 & 2 \arrow{rr}{x} & & 4 \arrow{ur}{a_{1, 1}} \arrow{rr}{a_{1, y}} & & (1, y)
\end{tikzcd}
\]

The ideal $I$ is generated by $a_{1, y}y$. For a choice of the pair $(r, s)$, the module $(F^1_{xy})_{r, s}$ over $k[x^{\pm 1}, y^{\pm 1}]\ten_k k$ is generated by $a_{1, 1}\frac{x^3-y^2}{y^2}$ and $a_{0, 1}-a_{1, 1}x^r y^s$. The ideal $\ol{I}^1_{r, s}$ is generated by $a_{0, 1}x-a_{1, 1}y$. Therefore, for any choice of $(r, s)$, the relation $I_{r, s}$ is the ideal generated by $a_{1, y}y$ and $a_{0, 1}x-a_{1, 1}y$.

The surjective morphism $\phi_{r, s}\colon kQ\to \Ga$ of algebras is given as follows. The images of the idempotents corresponding to the vertices $i$, $(i, y)$, and $(0, 1)$ under $\phi_{r, s}$ are the idempotents of $\Ga$ corresponding to the indecomposable direct summands $V^i$, $V^{i+5, y}$, and $V^{5, f_1}$, respectively. The images of the elements $x_i$, $y_1$, $a_{i, y}$, and $a_{i, 1}$ under $\phi_{r, s}$ are given as follows. The elements $\phi_{r, s}(x_i)$ are the morphisms
\[
R(i)_{\geq 0} \longrightarrow R(i+2)_{\geq 0}
\]
given by multiplying $x$. The elements $\phi_{r, s}(y_1)$ is the morphism
\[
R(1)_{\geq 0} \longrightarrow R(4)_{\geq 0}
\]
given by multiplying $y$. The elements $\phi_{r, s}(a_{i, y})$ are the morphisms
\[
R(i+3)_{\geq 0} \longrightarrow k[x^{\pm 1}](i+5)_{\geq 0}
\]
given by multiplying $x$. The elements $\phi_{r, s}(a_{i, 1})$ are the morphisms
\[
\begin{tikzcd}
R(i+3)_{\geq 0} \arrow{r} & (k[x^{\pm 1}, y^{\pm 1}]/(x^3 -y^2))(5)_{\geq 0}
\end{tikzcd}
\]
given by multiplying $(x^r y^s)^{2-i}$. Then the kernel of $\phi_{r, s}$ coincides with $I_{r, s}$.

For other $ADE$ singularities, the corresponding quivers with relations can be illustrated similarly.
\end{example}

\begin{example} \label{ex:non-reduced}
We give a family of examples in the case where $R$ is not reduced. Let $k$ be an arbitrary field and $n_y \geq 2$ an integer. For $R=k[x, y]/(y^{n_y})$ with $\deg x=2$ and $\deg y=3$, we have $a=3n_y-5$. If $n_y$ is even, then the quiver $Q$ is given as follows.
\[
\adjustbox{max width=\textwidth}{
\begin{tikzcd}
 & 2 \arrow{r}{x} \arrow{dr}{y} & 4 & \cdots & 3n_y-10 \arrow{r}{x} \arrow{dr}{y} & 3n_y-8 \arrow{r}{x} \arrow{dr}{y} & 3n_y-6 \arrow{r}{a_{0, y}} & (0, y) \arrow[shift left=0.5ex]{d}{b_{0, y}} \\
1 \arrow{r}{x} \arrow{urr}{y} & 3 \arrow{r}{x} & 5 & \cdots & 3n_y-9 \arrow{r}{x} \arrow{urr}{y} & 3n_y-7 \arrow{r}{x} & 3n_y-5 \arrow{r}{a_{1, y}} & (1, y) \arrow[shift left=0.5ex]{u}{b_{1, y}}
\end{tikzcd}
}
\]

If $n_y$ is odd, then the quiver $Q$ is given as follows.
\[
\adjustbox{max width=\textwidth}{
\begin{tikzcd}
1 \arrow{r}{x} \arrow{dr}{y} & 3 \arrow{r}{x} \arrow{dr}{y} & 5 & \cdots & 3n_y-10 \arrow{r}{x} \arrow{dr}{y} & 3n_y-8 \arrow{r}{x} \arrow{dr}{y} & 3n_y-6 \arrow{r}{a_{0, y}} & (0, y) \arrow[shift left=0.5ex]{d}{b_{0, y}} \\
2 \arrow{r}{x} \arrow{urr}{y} & 4 \arrow{r}{x} & 6 & \cdots & 3n_y-9 \arrow{r}{x} \arrow{urr}{y} & 3n_y-7 \arrow{r}{x} & 3n_y-5 \arrow{r}{a_{1, y}} & (1, y) \arrow[shift left=0.5ex]{u}{b_{1, y}}
\end{tikzcd}
}
\]

In both cases, the ideal $I$ is generated by $b_y ^{n_y}$, $xy-yx$, $b_{0, y}a_{0, y}x-a_{1, y}y$, and \linebreak $b_{1, y}a_{1, y}x^2 -a_{0, y}y$. Since there is no $(F^j_{xy})_{r, s}$, for any choice of the pair $(r, s)$, the relation $I_{r, s}$ coincides with $I$.

The surjective morphism $\phi_{r, s}\colon kQ\to \Ga$ of algebras is given as follows. The images of the idempotents corresponding to the vertices $i$ (respectively, $(i, y)$) under $\phi_{r, s}$ are the idempotents of $\Ga$ corresponding to the indecomposable direct summands $V^i$ (respectively, $V^{3n_y-4+i, y}$). The images of the elements $x_i$, $y_i$, $a_{i, y}$, and $b_{i, y}$ under $\phi_{r, s}$ are given as follows. The elements $\phi_{r, s}(x_i)$ are the morphisms
\[
R(i)_{\geq 0}\longrightarrow R(i+2)_{\geq 0}
\]
given by multiplying $x$. The elements $\phi_{r, s}(y_i)$ are the morphisms
\[
R(i)_{\geq 0}\longrightarrow R(i+3)_{\geq 0}
\]
given by multiplying $y$. The elements $\phi_{r, s}(a_{i, y})$ are the morphisms
\[
\begin{tikzcd}
R(3n_y-6+i)_{\geq 0} \arrow{r} & (k[x^{\pm 1}, y]/(y^{n_y}))(3n_y-4+i)_{\geq 0}
\end{tikzcd}
\]
given by multiplying $x$. The element $\phi_{r, s}(b_{0, y})$ is the morphism
\[
\begin{tikzcd}
(k[x^{\pm 1}, y]/(y^{n_y}))(3n_y-4)_{\geq 0} \arrow{r} & (k[x^{\pm 1}, y]/(y^{n_y}))(3n_y-3)_{\geq 0}
\end{tikzcd}
\]
given by multiplying $\frac{y}{x}$. The element $\phi_{r, s}(b_{1, y})$ is the morphism
\[
\begin{tikzcd}
(k[x^{\pm 1}, y]/(y^{n_y}))(3n_y-3)_{\geq 0} \arrow{r} & (k[x^{\pm 1}, y]/(y^{n_y}))(3n_y-4)_{\geq 0}
\end{tikzcd}
\]
given by multiplying $\frac{y}{x^2}$. Then the kernel of $\phi_{r, s}$ coincides with $I_{r, s}$.
\end{example}

\begin{example} \label{ex:non-algebraically closed}
To give an example in the case where $k$ is not algebraically closed, we consider $R=\Q[x, y]/(x^6 +y^4)$ with $\deg x=2$ and $\deg y=3$. Then we have $a=7$. The quiver $Q$ is given as follows.
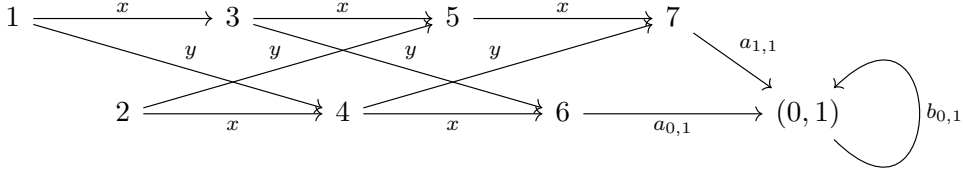
\[
\begin{tikzcd}
1 \arrow{rr}{x}\arrow{drrr}{y} & & 3 \arrow{rr}{x}\arrow{drrr}{y} & & 5 \arrow{rr}{x} & & 7 \arrow{dr}{a_{1, 1}} \\
 & 2 \arrow[swap]{rr}{x}\arrow{urrr}{y} & & 4 \arrow[swap]{rr}{x}\arrow{urrr}{y} & & 6 \arrow[swap]{rr}{a_{0, 1}} & & (0, 1) \arrow[out=315,in=45,loop,"b_{0, 1}",swap]
\end{tikzcd}
\]
The ideal $I$ is generated by $xy-yx$ and $b_{0, 1}^2+e_{0, 1}$. For a choice of the pair $(r, s)$, the module $(F^1_{xy})_{r, s}$ over $\Q[x^{\pm 1}, y^{\pm 1}]\ten_{\Q} \Q[b_{0, 1}]/(b_{0, 1}^2)$ is generated by $b_{0, 1}a_{1, 1}-a_{1, 1}\frac{x^3}{y^2}$ and \linebreak $a_{0, 1}-a_{1, 1}x^r y^s$. The pair $(r, s)$ is in one of the following four cases.

If we have $(r, s)=(12t-1, -8t+1)$ for an integer $t$, then the ideal $\ol{I}^1_{r, s}$ is generated by $b_{0, 1}a_{1, 1}y^2-a_{1, 1}x^3$ and $a_{0, 1}x-a_{1, 1}y$. Therefore, the relation $I_{r, s}$ is the ideal generated by $xy-yx$, $b_{0, 1}^2+e_{0, 1}$, $b_{0, 1}a_{1, 1}y^2-a_{1, 1}x^3$, and $a_{0, 1}x-a_{1, 1}y$.

If we have $(r, s)=(12t+2, -8t-1)$ for an integer $t$, then the ideal $\ol{I}^1_{r, s}$ is generated by $b_{0, 1}a_{1, 1}y^2-a_{1, 1}x^3$ and $a_{0, 1}y-a_{1, 1}x^2$. Therefore, the relation $I_{r, s}$ is the ideal generated by $xy-yx$, $b_{0, 1}^2+e_{0, 1}$, $b_{0, 1}a_{1, 1}y^2-a_{1, 1}x^3$, and $a_{0, 1}y-a_{1, 1}x^2$.

If we have $(r, s)=(12t-4, -8t+3)$ for an integer $t$, then the ideal $\ol{I}^1_{r, s}$ is generated by $b_{0, 1}a_{1, 1}y^2-a_{1, 1}x^3$ and $a_{0, 1}y+a_{1, 1}x^2$. Therefore, the relation $I_{r, s}$ is the ideal generated by $xy-yx$, $b_{0, 1}^2+e_{0, 1}$, $b_{0, 1}a_{1, 1}y^2-a_{1, 1}x^3$, and $a_{0, 1}y+a_{1, 1}x^2$.

If we have $(r, s)=(12t+5, -8t-3)$ for an integer $t$, then the ideal $\ol{I}^1_{r, s}$ is generated by $b_{0, 1}a_{1, 1}y^2-a_{1, 1}x^3$ and $a_{0, 1}x+a_{1, 1}y$. Therefore, the relation $I_{r, s}$ is the ideal generated by $xy-yx$, $b_{0, 1}^2+e_{0, 1}$, $b_{0, 1}a_{1, 1}y^2-a_{1, 1}x^3$, and $a_{0, 1}x+a_{1, 1}y$.

The surjective morphism $\phi_{r, s}\colon kQ\to \Ga$ of $\Q$-algebras is given as follows. The images of the idempotents corresponding to the vertices $i$ (respectively, $(0, 1)$) under $\phi_{r, s}$ are the idempotents of $\Ga$ corresponding to the indecomposable direct summands $V^i$ (respectively, $V^{8, f_1}$). The images of the elements $x_i$, $y_i$, $a_{i, 1}$, and $b_{0, 1}$ under $\phi_{r, s}$ are given as follows. The elements $\phi_{r, s}(x_i)$ are the morphisms
\[
R(i)_{\geq 0} \longrightarrow R(i+2)_{\geq 0}
\]
given by multiplying $x$. The elements $\phi_{r, s}(y_i)$ are the morphisms
\[
R(i)_{\geq 0} \longrightarrow R(i+3)_{\geq 0}
\]
given by multiplying $y$. The element $\phi_{r, s}(a_{0, 1})$ is the morphism
\[
\begin{tikzcd}
R(6)_{\geq 0} \arrow{r} & (\Q[x^{\pm 1}, y^{\pm 1}]/(x^6 +y^4))(8)_{\geq 0}
\end{tikzcd}
\]
given by multiplying $(x^r y^s)^2$. The element $\phi_{r, s}(a_{1, 1})$ is the morphism
\[
\begin{tikzcd}
R(7)_{\geq 0} \arrow{r} & (\Q[x^{\pm 1}, y^{\pm 1}]/(x^6 +y^4))(8)_{\geq 0}
\end{tikzcd}
\]
given by multiplying $x^r y^s$. The element $\phi_{r, s}(b_{0, 1})$ is the morphism
\[
\begin{tikzcd}
(\Q[x^{\pm 1}, y^{\pm 1}]/(x^6 +y^4))(8)_{\geq 0} \arrow{r} & (\Q[x^{\pm 1}, y^{\pm 1}]/(x^6 +y^4))(8)_{\geq 0}
\end{tikzcd}
\]
given by multiplying $\frac{x^3}{y^2}$. Then the kernel of $\phi_{r, s}$ coincides with $I_{r, s}$.
\end{example}

\subsection{The structure theorem for negative $a$-invariants} \label{ss:structure theorem for negative a-invariants}

Recall that we write
\[
f=x^{n_x}y^{n_y} \prod_{j=1}^q f_j ^{n_j} \: ,
\]
where $f_j$ are monic irreducible polynomials. We assume that $m=\deg x$ and $n=\deg y$ are coprime with $m\leq n$.

\begin{lemma} \label{lem:negative}
The $a$-invariant of $R$ is negative if and only if either the $\N$-graded commutative ring $R$ is regular, or $f$ equals $x^{n_x}$ with $n_x>1$ and $(n_x-1)m<n$.
\end{lemma}

\begin{proof}
We first prove the sufficiency. If $R$ is regular, then the polynomial $f$ does not lie in the ideal $(x, y)^2$ of $k[x, y]$. Therefore, the degree of $f$ is $m$ or $n$ and hence the $a$-invariant $a=\deg f-m-n$ is negative. If $f$ equals $x^{n_x}$ with $n_x>1$ and $(n_x-1)m<n$, then $a$ is also negative.

We now prove the necessity. Since the $a$-invariant $a=\deg f-m-n$ of $R$ is negative, we have $\deg f<m+n$. If $n_y$ is positive, then we must have $n_x=n_j=0$ for all $1\leq j\leq q$ and $f=y$. If $n_j$ is positive for some $1\leq j\leq q$, then we must have $n_x=n_y=n_{j'}=0$ for all $j'\neq j$ and $f=\alpha x^n+\beta y$ for $\alpha$, $\beta \neq 0$. If we have $n_y=n_j=0$ for all $1\leq j\leq q$ and $n_x=1$, then we have $f=x$. In these cases, the $\N$-graded commutative ring $R$ is regular. Otherwise, we have $f=x^{n_x}$ with $n_x>1$ and $(n_x-1)m<n$.
\end{proof}

We now suppose that the $a$-invariant of $R$ is negative and the stable category $\ct =\ul{\CM \!}_0^\Z R$ does not vanish. For any positive integer $i$, we define graded $R$-modules
\[
V^i=(k[x, y^{\pm 1}]/(x^{n_x}))(i)_{\geq 0} \: .
\]
By Theorem~1.6 of \cite{BuchweitzIyamaYamaura20}, the category $\ct$ contains a standard silting object $V=\bigoplus_{i=1}^{a+n}V^i$. To describe the dg endomorphism algebra of $V$, we introduce the following dg path algebra $(kQ', d)$.
\begin{definition} \label{def:dg path algebras}
The set $Q_0$ of the vertices of $Q$ is $\Z/n\Z$. For an even number $p>0$ and $i\in \Z/n\Z$, we define $\nu_p (i)=i+\frac{p}{2}n_x m$. For an odd number $p>0$ and $i\in \Z/n\Z$, we define $\nu_p (i)=i+(1+\frac{p-1}{2}n_x)m$. The set $Q_1$ of the arrows of $Q$ is
\[
\{ \beta_{p, i}\colon i\to \nu_p (i)\mid p>0, i\in \Z/n\Z \} \: ,
\]
where $\beta_{p, i}$ is of degree $1-p$. Let $(kQ, d)$ be the dg path algebra $kQ$ with the differential determined by
\begin{align*}
d(\beta_{p, i})= & \sum_{p_1+p_2=p, 2\mid p_1 p_2}(-1)^{p_2}\beta_{p_2, \nu_{p_1}(i)}\beta_{p_1, i} \\
 & +\sum_{p_1+\cdots +p_{n_x}=p+n_x-2, 2\nmid p_1 \ldots p_{n_x}}(-1)^{\frac{n_x(n_x-1)}{2}}\beta_{p_{n_x}, (\nu_{p_{n_x-1}}\circ \cdots \circ \nu_{p_1})(i)}\ldots \beta_{p_1, i} \: .
\end{align*}
Let $Q'$ be the graded quiver obtained from $Q$ by removing the vertices $0$, \ldots, $-a-1$ and the arrows which are adjacent to them. The differential $d$ induced by that of $(kQ, d)$ makes $kQ'$ into a dg algebra.
\end{definition}

\begin{theorem} \label{thm:negative}
Let $R$ and $(kQ', d)$ be as above. Denote the dg endomorphism algebra of the standard silting object $V$ in $\ul{\CM \!}_0^\Z R$ by $\Ga$.
\begin{itemize}
\item[a)] The dg algebra $\Ga$ is quasi-isomorphic to $(kQ', d)$. In particular, the stable category $\ul{\CM \!}_0^\Z R$ is triangle equivalent to $\per (kQ', d)$.
\item[b)] The dg algebra $\Ga$ is proper and Gorenstein.
\end{itemize}
\end{theorem}

\subsection{An example of Theorem~\ref{thm:negative}} \label{ss:an example of negative}

The following example serves to illustrate the construction of dg path algebras $(kQ', d)$. Let $k$ be an arbitrary field. For $R=k[x, y]/(x^3)$ with $\deg x=2$ and $\deg y=9$, we have $a=-5$. The graded quiver $Q$ is given as follows.
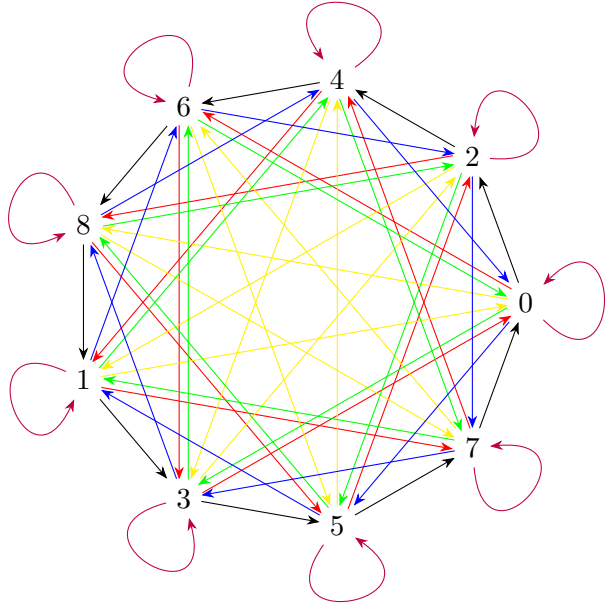
\[
\begin{tikzpicture}[>=Stealth, scale=3]
  \def\labels{{0,2,4,6,8,1,3,5,7}}

  \foreach \i [evaluate=\i as \label using {\labels[\i]}] in {0,...,8} {
    \node (A\i) at ({360/9*(\i)}:1) {\label};
  }

  \foreach \i in {0,...,8} {
    \pgfmathtruncatemacro{\j}{mod(\i+1,9)}
    \draw[->] (A\i) -- (A\j);
  }

  \foreach \i in {0,...,8} {
    \pgfmathtruncatemacro{\j}{mod(\i+3,9)}
    \pgfmathsetmacro{\dx}{cos(360/9*\j) - cos(360/9*\i)}
    \pgfmathsetmacro{\dy}{sin(360/9*\j) - sin(360/9*\i)}
    \pgfmathsetmacro{\len}{veclen(\dx,\dy)}
    \pgfmathsetmacro{\ox}{0.02*\dy/\len}
    \pgfmathsetmacro{\oy}{-0.02*\dx/\len}
    \draw[->, red, shorten >=0.25cm, shorten <=0.25cm] 
      ($ (A\i) + (\ox,\oy) $) -- 
      ($ (A\j) + (\ox,\oy) $);
  }

  \foreach \i in {0,...,8} {
    \pgfmathtruncatemacro{\j}{mod(\i+4,9)}
    \draw[->, yellow] (A\i) -- (A\j);
  }

  \foreach \i in {0,...,8} {
    \pgfmathtruncatemacro{\j}{mod(\i+6,9)}
    \pgfmathsetmacro{\dx}{cos(360/9*\j) - cos(360/9*\i)}
    \pgfmathsetmacro{\dy}{sin(360/9*\j) - sin(360/9*\i)}
    \pgfmathsetmacro{\len}{veclen(\dx,\dy)}
    \pgfmathsetmacro{\ox}{0.02*\dy/\len}
    \pgfmathsetmacro{\oy}{-0.02*\dx/\len}
    \draw[->, green, shorten >=0.25cm, shorten <=0.25cm] 
      ($ (A\i) + (\ox,\oy) $) -- 
      ($ (A\j) + (\ox,\oy) $);
  }

  \foreach \i in {0,...,8} {
    \pgfmathtruncatemacro{\j}{mod(\i+7,9)}
    \draw[->, blue] (A\i) -- (A\j);
  }

  \foreach \i in {0,...,8} {
    \pgfmathtruncatemacro{\outangle}{315 + 40*\i}
    \pgfmathtruncatemacro{\inangle}{45 + 40*\i}
    \draw[->, purple] (A\i) to[out=\outangle, in=\inangle, loop] (A\i);
  }
\end{tikzpicture}
\]
Here each arrow represents a family of arrows parametrized by $\N \setminus \{0\}$. The black arrows represent $\{ \beta_{6p-5, i} \mid p\geq 1, i\in Q_0 \}$, the red arrows represent $\{ \beta_{6p-4, i} \mid p\geq 1, i\in Q_0 \}$, the yellow arrows represent $\{ \beta_{6p-3, i} \mid p\geq 1, i\in Q_0 \}$, the green arrows represent $\{ \beta_{6p-2, i} \mid p\geq 1, i\in Q_0 \}$, the blue arrows represent $\{ \beta_{6p-1, i} \mid p\geq 1, i\in Q_0 \}$, and the purple loops represent $\{ \beta_{6p, i} \mid p\geq 1, i\in Q_0 \}$. The differentials of $\beta_{p, i}$ for small $p$ are given by
\begin{align*}
d(\beta_{1, i})= & 0 \: , \\
d(\beta_{2, i})= & -\beta_{1, i+4}\beta_{1, i+2}\beta_{1, i} \: , \\
d(\beta_{3, i})= & \beta_{2, i+2}\beta_{1, i}-\beta_{1, i+6}\beta_{2, i} \: , \\
d(\beta_{4, i})= & \beta_{2, i+6}\beta_{2, i}-\beta_{3, i+4}\beta_{1, i+2}\beta_{1, i}-\beta_{1, i+10}\beta_{3, i+2}\beta_{1, i}-\beta_{1, i+10}\beta_{1, i+8}\beta_{3, i} \: , \\
d(\beta_{5, i})= & \beta_{4, i+2}\beta_{1, i}-\beta_{3, i+6}\beta_{2, i}+\beta_{2, i+8}\beta_{3, i}-\beta_{1, i+12}\beta_{4, i} \: , \\
d(\beta_{6, i})= & \beta_{4, i+6}\beta_{2, i}+\beta_{2, i+12}\beta_{4, i}-\beta_{5, i+4}\beta_{1, i+2}\beta_{1, i}-\beta_{1, i+16}\beta_{5, i+2}\beta_{1, i}-\beta_{1, i+16}\beta_{1, i+14}\beta_{5, i} \\
 & -\beta_{3, i+10}\beta_{3, i+2}\beta_{1, i}-\beta_{3, i+10}\beta_{1, i+8}\beta_{3, i}-\beta_{1, i+16}\beta_{3, i+8}\beta_{3, i} \: .
\end{align*}
Let $Q'$ be the graded quiver obtained from $Q$ by removing the vertices $0$, $1$, $2$, $3$, and $4$ and the arrows which are adjacent to them. The differential $d$ induced by that of $(kQ, d)$ makes $kQ'$ into a dg algebra. The dg endomorphism algebra $\Ga$ of the standard silting object $V$ in $\ul{\CM \!}_0^\Z R$ is quasi-isomorphic to $(kQ', d)$.

\section{Proof of our results} \label{section:proof of our results}

\subsection{Structures of the morphism spaces} \label{ss:structures of the morphism spaces}

In this section, we suppose that the \linebreak $a$-invariant of $R$ is non-negative. We analyze the structures of the morphism spaces, radical morphism spaces, and the square of radical morphism spaces between direct summands of $V$ in the following lemmas. They will be used to prove Theorem~\ref{thm:non-negative}. We write the morphism spaces and radical morphism spaces in the full subcategory $\add V\subseteq \CM \!_0^\Z R$ by $\Hom_R(?, -)_0$ respectively $\rad_V(?, -)_0$.

\begin{lemma} \label{lem:additive generators}
For large enough integer $N$, the subcategory $\add (\bigoplus_{i=a+1}^N V^i) \subseteq \CM \!_0^\Z R$ coincides with $\add (V^x \oplus V^y \oplus \bigoplus_{j=1}^q V^{f_j})$. In particular, we have $\add (\bigoplus_{i=1}^N V^i)=\add V$.
\end{lemma}

\begin{proof}
By part~(a) of Lemma~4.11 of \cite{BuchweitzIyamaYamaura20} and part~(c) of Theorem~2.1 of \cite{BuchweitzIyamaYamaura20}, the object $V^i$ is isomorphic to $V^{i, x}\oplus V^{i, y}\oplus \bigoplus_{j=1}^q V^{i, f_j}$ for all $i\geq a+1$. Since $N$ is large enough, by part~(b) of Lemma~4.11 of \cite{BuchweitzIyamaYamaura20}, we have
\[
\add (\bigoplus_{i=a+1}^N (V^{i, x}\oplus V^{i, y}\oplus \bigoplus_{j=1}^q V^{i, f_j}))=\add (V^x \oplus V^y \oplus \bigoplus_{j=1}^q V^{f_j}) \: .
\]
Then the assertions follow.
\end{proof}

\begin{lemma} \label{lem:stable endomorphism}
We have the isomorphism $\End_R (V)_0 \iso \End_{\ct}(V)$ of algebras.
\end{lemma}

\begin{proof}
The assertion follows by Lemma~\ref{lem:additive generators} and part~(b) of Lemma~4.13 of \cite{BuchweitzIyamaYamaura20}.
\end{proof}

Consider the full subquivers $Q^x$, $Q^y$, and $Q^j \subseteq Q$ with vertex sets $\{(i, x)\mid 0\leq i<n \}$, $\{(i, y)\mid 0\leq i<m\}$, and $\{(0, j)\}$, respectively.

\begin{lemma} \label{lem:morphism}\mbox{}
\begin{itemize}
\item[a)] For any $1\leq i$, $i'\leq a$, we have a bijection $\Hom_R (V^i, V^{i'})_0\iso R_{i'-i}$.
\item[b)] We have an isomorphism
\[
\End_R(V^x)_0 \xlongrightarrow{_\sim} kQ^x/(b_x^{n_x})
\]
of algebras. For any $a+1\leq i\leq a+n$ and positive integer $t$ satisfying $i-tn\geq 1$, the morphism $\cdot y^t \colon V^{i-tn} \to V^{i, x}$ induces the bijection
\[
\begin{tikzcd}
\Hom_R(V^{i, x}, V^x)_0 \arrow{r}{\sim} & \Hom_R(V^{i-tn}, V^x)_0 \: .
\end{tikzcd}
\]
\item[c)] We have an isomorphism
\[
\End_R(V^y)_0 \xlongrightarrow{_\sim} kQ^y/(b_y^{n_y})
\]
of algebras. For any $a+1\leq i\leq a+m$ and positive integer $t$ satisfying $i-tm\geq 1$, the morphism $\cdot x^t \colon V^{i-tm} \to V^{i, y}$ induces the bijection
\[
\begin{tikzcd}
\Hom_R(V^{i, y}, V^y)_0 \arrow{r}{\sim} & \Hom_R(V^{i-tm}, V^y)_0 \: .
\end{tikzcd}
\]
\item[d)] For any $1\leq j\leq q$, we have an isomorphism
\[
\End_R(V^{f_j})_0 \xlongrightarrow{_\sim} kQ^j/(g_j(b_j)^{n_j})
\]
of algebras. For any $1\leq j\leq q$ and $1\leq t\leq a$, the morphism \mbox{$\cdot (x^r y^s)^t \colon V^{a+1-t} \to V^{f_j}$} induces the bijection
\[
\begin{tikzcd}
\End_R(V^{f_j})_0 \arrow{r}{\sim} & \Hom_R(V^{a+1-t}, V^{f_j})_0 \: .
\end{tikzcd}
\]
\item[e)] We have
\begin{align*}
\Hom_R(V^x, V^{[1, a]}\oplus V^y \oplus \bigoplus_{j=1}^q V^{f_j})_0 & =0\mbox{ and}\\
\Hom_R(V^y, V^{[1, a]}\oplus V^x \oplus \bigoplus_{j=1}^q V^{f_j})_0 & =0 \: .
\end{align*}
For any $1\leq j\leq q$, we have
\[
\Hom_R(V^{f_j}, V^{[1, a]}\oplus V^x \oplus V^y \oplus \bigoplus_{j'\neq j} V^{f_{j'}})_0 =0 \: .
\]
\item[f)] For any $1\leq i\leq a$, all morphisms in $\rad_V (V^i, V)_0$ factor through the morphism $\cdot x \colon V^i \to V^{i+m}$ or the morphism $\cdot y \colon V^i \to V^{i+n}$.
\end{itemize}
\end{lemma}

\begin{proof}
a) It is clear.

b) In the definition
\[
V^x=\bigoplus_{i=a+1}^{a+n}V^{i, x}=\bigoplus_{i=a+1}^{a+n}(k[x, y^{\pm 1}]/(x^{n_x}))(i)_{\geq 0} \: ,
\]
the index $i$ runs from $a+1$ to $a+n$. Since the invertible element $y$ is of degree $n$, the morphism space $\Hom_R(V^{i, x}, V^x)_0$ is spanned by the morphisms given by multiplying non-negative powers of $x$ and the uniquely determined non-negative powers of $y^{-1}$. So the algebra $\End_R(V^x)_0$ is isomorphic to $kQ^x/(b_x^{n_x})$. We now prove the second assertion. By the universal properties of localizations and quotients, we have the bijection
\[
\begin{tikzcd}
\Hom_R (V^{i-tn, x}, V^x)_0 \arrow{r}{\sim} & \Hom_R (V^{i-tn}, V^x)_0 \: .
\end{tikzcd}
\]
Since the morphism
\[
\cdot y^t \colon V^{i-tn, x} \iso V^{i, x}
\]
is an isomorphism, we have the bijection
\[
\begin{tikzcd}
\Hom_R (V^{i, x}, V^x)_0 \arrow{r}{\sim} & \Hom_R (V^{i-tn, x}, V^x)_0 \: .
\end{tikzcd}
\]
Then the assertion follows from the above bijections.

c) Similar to part~b).

d) Since the morphism space $\End_R(V^{f_j})_0$ is spanned by the morphisms given by multiplying non-negative powers of $f_j$ and the uniquely determined positive powers of $y^{-1}$, the algebra $\End_R(V^{f_j})_0$ is isomorphic to $kQ^j/(g_j(b_j)^{n_j})$. We now prove the second assertion. By the universal properties of localizations and quotients, we have the bijection
\[
\begin{tikzcd}
\Hom_R (V^{a+1-t, f_j}, V^{f_j})_0 \arrow{r}{\sim} & \Hom_R (V^{a+1-t}, V^{f_j})_0 \: .
\end{tikzcd}
\]
Since the morphism
\[
\cdot (x^r y^s)^t \colon V^{a+1-t, f_j}\iso V^{a+1, f_j}
\]
is an isomorphism, we have the bijection
\[
\End_R (V^{f_j})_0 \xlongrightarrow{_\sim} \Hom_R (V^{a+1-t, f_j}, V^{f_j})_0 \: .
\]
Then the assertion follows from the above bijections.

e) Since $V^x$ is annihilated by $x^{n_x}$ but $V^y$ and $V^{f_j}$ do not contain a nonzero element annihilated by $x^{n_x}$, we have $\Hom_R (V^x, V^y \oplus \bigoplus_{j=1}^q V^{f_j})_0 =0$. By Lemma~\ref{lem:additive generators} and part~(a) of Lemma~4.13 of \cite{BuchweitzIyamaYamaura20}, we have $\Hom_R (V^x, V^{[1, a]})=0$. So the first assertion follows. The proofs of the other assertions are similar.

f) This is a consequence of parts~a), b), c), and d).
\end{proof}

\begin{lemma} \label{lem:radical}\mbox{}
\begin{itemize}
\item[a)] The radical of $\End_R (V^{[1, a]})_0$ is generated by the morphisms in $\Hom_R (V^i, V^{[1, a]})_0$ given by multiplying $x$ or $y$.
\item[b)] The radical of $\End_R (V^x)_0$ is generated by the morphisms in $\Hom_R (V^{i, x}, V^x)_0$ given by multiplying $x$ and the uniquely determined non-negative powers of $y^{-1}$.
\item[c)] The radical of $\End_R (V^y)_0$ is generated by the morphisms in $\Hom_R (V^{i, y}, V^y)_0$ given by multiplying $y$ and the uniquely determined non-negative powers of $x^{-1}$.
\item[d)] For any $1\leq j\leq q$, the radical of $\End_R (V^{f_j})_0$ is generated by the morphism given by multiplying $f_j$ and the uniquely determined positive power of $y^{-1}$.
\item[e)] The object $V$ is basic. In particular, all morphisms between different indecomposable direct summands of $V$ are radical morphisms.
\end{itemize}
\end{lemma}

\begin{proof}
a) This follows from part~a) of Lemma~\ref{lem:morphism}.

b) The ideal of $\End_R (V^x)_0$, generated by the morphisms in $\Hom_R (V^{i, x}, V^x)_0$ given by multiplying $x$ and the uniquely determined non-negative powers of $y^{-1}$, is nilpotent. The quotient of $\End_R (V^x)_0$ by this ideal is isomorphic to the $n$th products of copies of $k$. Therefore, this ideal equals the radical of $\End_R (V^x)_0$.

c) Similar to part~b).

d) Similar to part~b).

e) The assertion follows from parts~a), b), c), d) and e) of Lemma~\ref{lem:morphism}.
\end{proof}

\begin{lemma} \label{lem:radical square}\mbox{}
\begin{itemize}
\item[a)] The square of the radical of $\End_R (V^{[1, a]})_0$ is generated by the morphisms in \linebreak $\Hom_R (V^i, V^{[1, a]})_0$ given by multiplying $x^2$, $xy$, or $y^2$.
\item[b)] The square of the radical of $\End_R (V^x)_0$ is generated by the morphisms in \linebreak $\Hom_R (V^{i, x}, V^x)_0$ given by multiplying $x^2$ and the uniquely determined non-negative powers of $y^{-1}$.
\item[c)] The square of the radical of $\End_R (V^y)_0$ is generated by the morphisms in \linebreak $\Hom_R (V^{i, y}, V^y)_0$ given by multiplying $y^2$ and the uniquely determined non-negative powers of $x^{-1}$.
\item[d)] For any $1\leq j\leq q$, the square of the radical of $\End_R (V^{f_j})_0$ is generated by the morphism given by multiplying $f_j ^2$ and the uniquely determined positive power of $y^{-1}$.
\item[e)] For any $a+1+\max \{n-a, 0\}\leq i\leq a+n$, the morphism $\cdot y \colon V^{i-n} \to V^{i, x}$ induces the bijection
\[
\begin{tikzcd}
\rad_V (V^{i, x}, V^x)_0 \arrow{r}{\sim} & \rad_V ^2 (V^{i-n}, V^x)_0 \: .
\end{tikzcd}
\]
For any $a+1\leq i\leq a+n$ and integer $t$ satisfying $t\geq 2$ and $i-tn\geq 1$, the morphism $\cdot y^t \colon V^{i-tn} \to V^{i, x}$ induces the bijection
\[
\begin{tikzcd}
\Hom_R (V^{i, x}, V^x)_0 \arrow{r}{\sim} & \rad_V ^2 (V^{i-tn}, V^x)_0 \: .
\end{tikzcd}
\]
\item[f)] For any $a+1+\max \{m-a, 0\}\leq i\leq a+m$, the morphism $\cdot x \colon V^{i-m} \to V^{i, y}$ induces the bijection
\[
\begin{tikzcd}
\rad_V (V^{i, y}, V^y)_0 \arrow{r}{\sim} & \rad_V ^2 (V^{i-m}, V^y)_0 \: .
\end{tikzcd}
\]
For any $a+1\leq i\leq a+m$ and integer $t$ satisfying $t\geq 2$ and $i-tm\geq 1$, the morphism $\cdot x^t \colon V^{i-tm} \to V^{i, y}$ induces the bijection
\[
\begin{tikzcd}
\Hom_R (V^{i, y}, V^y)_0 \arrow{r}{\sim} & \rad_V ^2 (V^{i-tm}, V^y)_0 \: .
\end{tikzcd}
\]
\item[g)] For any $a+1+\max \{m-a, 0\}\leq i\leq a+m$, the morphism
\[
\cdot (x^r y^s)^{a+m+1-i} \colon V^{i-m} \to V^{f_j}
\]
induces the bijection
\[
\begin{tikzcd}
\rad_V (V^{f_j}, V^{f_j})_0 \arrow{r}{\sim} & \rad_V ^2 (V^{i-m}, V^{f_j})_0 \: .
\end{tikzcd}
\]
For any $a+1\leq i\leq a+m$ and integer $t$ satisfying $t\geq 2$ and $i-tm\geq 1$, the morphism
\[
\cdot x^{t-1}(x^r y^s)^{a+m+1-i} \colon V^{i-tm} \to V^{f_j}
\]
induces the bijection
\[
\begin{tikzcd}
\End_R (V^{f_j})_0 \arrow{r}{\sim} & \rad_V ^2 (V^{i-tm}, V^{f_j})_0 \: .
\end{tikzcd}
\]
\end{itemize}
\end{lemma}

\begin{proof}
a) This follows from part~a) of Lemma~\ref{lem:radical}.

b) This follows from part~b) of Lemma~\ref{lem:radical}.

c) This follows from part~c) of Lemma~\ref{lem:radical}.

d) This follows from part~d) of Lemma~\ref{lem:radical}.

e) We first prove the first bijection. Since we have $a-(i-n)<n$, for any $i-n\leq i'\leq a$, the morphism space $\Hom_R (V^{i-n}, V^{i'})_0$ is spanned by the morphisms given by multiplying non-negative powers of $x$. For any $a+1\leq i''\leq a+n$, we have the commutative diagram
\[
\begin{tikzcd}
V^{i-n} \arrow{r} \arrow{d} & V^{i'} \arrow{d} \\
V^{i-n+i''-i', x} \arrow{r} & V^{i'', x} \mathrlap{\: ,}
\end{tikzcd}
\]
where the horizontal morphisms (respectively, the vertical morphisms) are the corresponding morphisms. This implies that
\[
\Hom_R (V^{i'}, V^{i'', x})_0 \cdot \rad_V (V^{i-n}, V^{i'})_0
\]
is contained in
\[
\rad_V (V^{i-n+i''-i', x}, V^{i'', x})_0 \cdot \Hom_R (V^{i-n}, V^{i-n+i''-i', x})_0 \: .
\]
Therefore, by part~e) of Lemma~\ref{lem:morphism}, we have
\begin{align*}
\rad_V ^2(V^{i-n}, V^x)_0 = & \rad_V (V^{[1, a]}, V^x)_0 \cdot \rad_V (V^{i-n}, V^{[1, a]})_0 \\
 & +\rad_V (V^x, V^x)_0 \cdot \rad_V (V^{i-n}, V^x)_0 \\
= & \rad_V (V^x, V^x)_0 \cdot \rad_V (V^{i-n}, V^x)_0 \: .
\end{align*}
Then the assertion follows from part~e) of Lemma~\ref{lem:radical} and part~b) of Lemma~\ref{lem:morphism}.

We now prove the second bijection. For any morphism $\phi \colon V^{i-tn}\to V^x$, by the universal properties of localizations and quotients, it factors uniquely through the canonical morphism $V^{i-tn} \to V^{i-tn, x}$. This unique morphism factors uniquely through the isomorphism
\[
\cdot y^{t-1} \colon V^{i-tn, x} \xlongrightarrow{_\sim} V^{i-n, x} \: .
\]
Then by the commutative diagram
\[
\begin{tikzcd}
V^{i-tn} \arrow{r}{\cdot y^{t-1}} \arrow{d} & V^{i-n} \arrow{d} \\
V^{i-tn, x} \arrow{r}{\cdot y^{t-1}} & V^{i-n, x}
\end{tikzcd}
\]
with the vertical morphisms the canonical morphisms, we deduce that $\phi$ factors through $V^{i-n}$. By part~e) of Lemma~\ref{lem:radical}, this implies that $\phi$ lies in $\rad_V ^2(V^{i-tn}, V^x)_0$. So we have
\[
\rad_V ^2(V^{i-tn}, V^x)_0 =\Hom_R (V^{i-tn}, V^x)_0 \: .
\]
Then the assertion follows from part~b) of Lemma~\ref{lem:morphism}.

f) Similar to part~e).

g) We first prove the first bijection. Since we have $a-(i-m)<m$, for any $i-m<i'\leq a$, the morphism space $\Hom_R (V^{i-m}, V^{i'})_0$ vanishes. Therefore, by part~e) of Lemma~\ref{lem:morphism}, we have
\begin{align*}
\rad_V ^2(V^{i-m}, V^{f_j})_0 = & \rad_V (V^{i-m}, V^{f_j})_0 \cdot \rad_V (V^{i-m}, V^{i-m})_0 \\
 & +\rad_V (V^{f_j}, V^{f_j})_0 \cdot \rad_V (V^{i-m}, V^{f_j})_0 \: .
\end{align*}
Then the assertion follows from parts~a) and e) of Lemma~\ref{lem:radical} and part~d) of Lemma~\ref{lem:morphism}.

We now prove the second bijection. For any morphism $\phi \colon V^{i-tm}\to V^{f_j}$, by the universal properties of localizations and quotients, it factors uniquely through the canonical morphism $V^{i-tm} \to V^{i-tm, f_j}$. This unique morphism factors uniquely through the isomorphism
\[
\cdot x^{t-1}\colon V^{i-tm, f_j} \xlongrightarrow{_\sim} V^{i-m, f_j} \: .
\]
Then by the commutative diagram
\[
\begin{tikzcd}
V^{i-tm} \arrow{r}{\cdot x^{t-1}} \arrow{d} & V^{i-m} \arrow{d} \\
V^{i-tm, f_j} \arrow{r}{\cdot x^{t-1}} & V^{i-m, f_j}
\end{tikzcd}
\]
with the vertical morphisms the canonical morphisms, we deduce that $\phi$ factors through $V^{i-m}$. By part~e) of Lemma~\ref{lem:radical}, this implies that $\phi$ lies in $\rad_V ^2(V^{i-tm}, V^{f_j})_0$. So we have
\[
\rad_V ^2 (V^{i-tm}, V^{f_j})_0 =\Hom_R (V^{i-tm}, V^{f_j})_0 \: .
\]
Then the assertion follows from part~d) of Lemma~\ref{lem:morphism}.

\end{proof}

\subsection{Proof of Theorem~\ref{thm:non-negative}} \label{ss:proof of non-negative}

By Lemma~\ref{lem:stable endomorphism}, we may and will assume that we have $\End_R (V)_0=\Ga$.

a) Since the category $\ul{\CM \!}_0^\Z R$ is $\Hom$-finite, the algebra $\Ga$ is finite-dimensional. To prove it is Iwanaga--Gorenstein, we first claim that the projective dimension of the simple \linebreak $\Ga \op$-module $S_i$ is less than or equal to $2$ for all $1\leq i\leq a$. In fact, we have the short exact sequence
\[
\begin{tikzcd}[ampersand replacement=\&]
0 \arrow{r} \& R(i)_{\geq 0} \arrow{r}{
\begin{bsmallmatrix}
x \\
y
\end{bsmallmatrix}
} \& R(i+m)_{\geq 0}\oplus R(i+n)_{\geq 0} \arrow{r}{
\begin{bsmallmatrix}
y & -x
\end{bsmallmatrix}
} \& R(i+m+n)_{\geq 0} \arrow{r} \& 0
\end{tikzcd}
\]
in $\CM \!_0^\Z R$. If we apply the functor $\Hom_R (?, V)_0$ to it, we obtain the exact sequence
\[
\begin{tikzcd}[ampersand replacement=\&]
0 \arrow{r} \& \Hom_R(V^{i+m+n}, V)_0 \arrow{r}{
\begin{bsmallmatrix}
y \\
-x
\end{bsmallmatrix}
} \& \Hom_R(V^{i+m}\oplus V^{i+n}, V)_0 \arrow{r}{
\begin{bsmallmatrix}
x & y
\end{bsmallmatrix}
} \& \Hom_R(V^i, V)_0
\end{tikzcd}
\]
of $\Ga \op$-modules. By Lemma~\ref{lem:additive generators}, the $\Ga \op$-module $\Hom_R (V^{i+m+n}, V)_0$ is projective. By part~f) of Lemma~\ref{lem:morphism}, the image of the last morphism is $\rad_V (V^i, V)_0$. So the cokernel of the last morphism is the simple $\Ga \op$-module $S_i$. This implies that its projective dimension is less than or equal to $2$.

We now prove that the projective dimension of any indecomposable injective $\Ga \op$-module is less than or equal to $2$. For any $1\leq i\leq a$, by part~e) of Lemma~\ref{lem:morphism}, we have
\[
\Hom_R(V^x, V^i)_0=\Hom_R(V^y, V^i)_0=\Hom_R(V^{f_j}, V^i)_0=0 \: .
\]
So the indecomposable injective $\Ga \op$-module $I_i$ is an extension of the simple $\Ga \op$-modules $S_{i'}$, $1\leq i'\leq a$. So the projective dimension of $I_i$ is less than or equal to $2$. For any $0\leq i<n$, since we have
\[
\Hom_R(V^y, V^{i, x})_0=\Hom_R(V^{f_j}, V^{i, x})_0=\Hom_R(V^{i, x}, V^{[1, a]})_0=0 \: ,
\]
we have the short exact sequence
\[
\begin{tikzcd}
0 \arrow{r} & M \arrow{r} & I_{(i, x)} \arrow{r} & N \arrow{r} & 0
\end{tikzcd}
\]
of $\Ga \op$-modules, where $M$ (respectively, $N$) is an extension of the simple $\Ga \op$-modules \linebreak $S_{(i', x)}$, $0\leq i'<n$ (respectively, $S_{i'}$, $1\leq i'\leq a$). By the previous argument, the projective dimension of $N$ is less than or equal to $2$. The $\Ga \op$-module $M$ is the image of the indecomposable injective $\End_R(V^x)_0 \op$-module $J_{(i, x)}$ under the fully faithful functor
\[
F\colon \mod(\End_R(V^x)_0 \op) \longrightarrow \mod \Ga \op \: .
\]
By part~b) of Lemma~\ref{lem:morphism}, the algebra $\End_R(V^x)_0$ is self-injective. This implies that the $\End_R(V^x)_0 \op$-module $J_{(i, x)}$ is also projective and hence we have $J_{(i, x)}=Q_{(i-(n_x-1)m, x)}$. Since we have
\[
\Hom_R(V^{i', x}, V^{[1, a]})_0=\Hom_R(V^{i', x}, V^y)_0=\Hom_R(V^{i', x}, V^{f_j})_0=0 \: ,
\]
the $\Ga \op$-module $M$, as the image $\Hom_R(V^{i-(n_x-1)m, x}, V)_0$ of $\Hom_R(V^{i-(n_x-1)m, x}, V^x)_0$ under the functor $F$, is also projective. Since the projective dimensions of the $\Ga \op$-modules $M$ and $N$ are less than or equal to $2$, so is the projective dimension of their extension $I_{(i, x)}$. Similarly, the projective dimensions of the indecomposable injective $\Ga \op$-modules $I_{(i, y)}$ and $I_{(0, j)}$ are less than or equal to $2$. Therefore, the injective dimension of $\Ga$ as a left module over itself is less than or equal to $2$. Similarly, the injective dimension of $\Ga$ as a right module over itself is also less than or equal to $2$. We conclude that the algebra $\Ga$ is Iwanaga--Gorenstein of self-injective dimension less than or equal to $2$. Its global dimension is finite if and only if it is less than or equal to $2$ if and only if the projective dimensions of the simple $\Ga \op$-modules $S_{(i, x)}$, $S_{(i, y)}$, and $S_{(0, j)}$ are less than or equal to $2$. Since the algebras $\End_R(V^x)_0$, $\End_R(V^y)_0$, and $\End_R(V^{f_j})_0$ are isomorphic to self-injective Nakayama algebras, the above equivalent conditions are equivalent to that $n_x$, $n_y$, and $n_j$ are less than or equal to $1$. This means that $R$ is reduced.

b) By part~e) of Lemma~\ref{lem:radical}, the algebra $\Ga$ is basic. By Lemmas~\ref{lem:morphism}, \ref{lem:radical} and \ref{lem:radical square}, we know the structures of $\Ga/\rad\, \Ga$ and $\rad\, \Ga/\rad^2 \Ga$. So we have the surjective morphism $\phi_{r, s}\colon kQ\to \Ga$ of algebras defined in Section~\ref{ss:structure theorem for non-negative a-invariants}. By the descriptions of the images of the arrows under $\phi_{r, s}$, we see that the kernel of $\phi_{r, s}$ coincides with $I_{r, s}$.

\subsection{Proof of Theorem~\ref{thm:negative}} \label{ss:proof of negative}

a) The assertion follows from Proposition~\ref{prop:properness and Gorensteinness}.

b) By part~(b) of Theorem~1.6 of \cite{BuchweitzIyamaYamaura20}, the dg endomorphism algebra of the standard silting object $V$ in $\ul{\CM \!}_0^\Z R$ is quasi-isomorphic to the dg quotient of $\End_{\ct}(\bigoplus_{i=1}^n V^i)$ by $\End_{\ct}(\bigoplus_{i=a+n+1}^n V^i)$. By part~b) of Lemma~\ref{lem:morphism} and Lemma~\ref{lem:stable endomorphism}, the endomorphism algebra $\End_{\ct}(\bigoplus_{i=1}^n V^i)$ is isomorphic to the self-injective Nakayama algebra $N_{n, n_x}$ defined in Proposition~\ref{prop:cofibrant dg replacement of Nakayama algebras}. By Proposition~\ref{prop:cofibrant dg replacement of Nakayama algebras}, it is quasi-isomorphic to the cofibrant dg algebra $(kQ, d)$ defined in Definition~\ref{def:dg path algebras}. So the dg quotient is $(kQ', d)$.

\section{Auslander--Reiten quivers for finite and countable Cohen--Macaulay types} \label{section:Auslander--Reiten quivers}

Buchweitz--Greuel--Schreyer gave a classification of commutative complete local Gorenstein rings of finite (respectively, countable) Cohen--Macaulay representation type: Under mild assumptions, they are precisely the $ADE$ singularities
\[
k\llbracket x, y, z_2, \ldots, z_d\rrbracket /(f(x,y)+z_2^2+\cdots+z_d^2)\: ,
\]
where $f(x,y)$ are the polynomials listed below, \cf Theorem~8.8 of \cite{Yoshino90}, Corollary~10.19 and Theorem 14.16 of \cite{LeuschkeWiegand12}.

In this section, we consider the corresponding $\N$-graded hypersurface singularities \linebreak $R=k[x,y]/(f)$, which are of finite (respectively, countable) Cohen--Macaulay representation type in the sense that there are only finitely (respectively, countably) many indecomposable objects in $\CM\!^{\Z}R$ up to degree shifts.

\[
\adjustbox{max width=\textwidth}{
$\begin{array}{|c||c|c|c|c|}\hline
R&A_{n}\ (\mbox{$n$ is odd})&A_{n}\ (\mbox{$n$ is even})&D_{n}\ (\mbox{$n$ is even})&D_{n}\ (\mbox{$n$ is odd})\\ \hline
f&x^{n+1}-y^2&x^{n+1}-y^2&x^{n-1}-xy^2&x^{n-1}-xy^2\\ \hline
(\deg x,\deg y)&(1,\frac{n+1}{2})&(2,n+1)&(1,\frac{n}{2}-1)&(2,n-2)\\ \hline
\begin{array}{l}\mbox{irreducible}\\ \mbox{factors of $f$}\end{array}&\begin{array}{l}f_1=x^{\frac{n+1}{2}}-y\: ,\\ f_2=x^{\frac{n+1}{2}}+y\end{array}&f&\begin{array}{l}x\: ,\\ f_1=x^{\frac{n}{2}-1}-y\: ,\\ f_2=x^{\frac{n}{2}-1}+y\end{array}&\begin{array}{l}x\: ,\\ f_1=x^{n-2}-y\end{array} \\ \hline
\end{array}$}
\]

\vskip-.1em
\[
\adjustbox{max width=\textwidth}{
$\begin{array}{|c||c|c|c||c|c|}\hline
R&E_6&E_7&E_8&A_\infty&D_\infty\\ \hline
f&x^4-y^3&x^3y-y^3&x^5-y^3&y^2&xy^2\\ \hline
(\deg x,\deg y)&(3,4)&(2,3)&(3,5)&(m,l)&(m,l)\\ \hline
\begin{array}{l}\mbox{irreducible}\\ \mbox{factors of $f$}\end{array}&f&\begin{array}{l}y,\\ f_1=x^3-y^2\end{array}&f&y&\begin{array}{l}x,\\ y\end{array}\\ \hline
\end{array}$
}
\]

In the rest, we draw the Auslander--Reiten quiver of the category $\CM\!_0^\Z R$ for $R$ in the above list. Our Auslander--Reiten quivers are displayed in a standard way following the theory of translation quivers \cite{Riedtmann80, AssemSimsonSkowronski06}, \cf Section~3 of \cite{Araya99}, see also Section~3 of \cite{DieterichWiedemann86}, Chapter~9 of \cite{Yoshino90}, Section~4.3 of \cite{HijikataNishida97}, and Section~13.2 of \cite{LeuschkeWiegand12} for the non-graded case.

Moreover, we give an explicit description of each indecomposable object $M\in \CM\!_0 ^{\Z} R$ by giving a homogeneous $k$-basis.
More explicitly, for $f=\prod_{j=1}^pf_j$, we view $R$ as an $\N$-graded subalgebra of $\prod_{j=1}^p R/(f_j)$.
In each picture, a column labelled by $f_{j_i}$ represents a $\Z$-graded $R$-submodule $M^i$ of the graded normalization of $R/(f_{j_i})$ whose homogeneous $k$-basis is given by black circles $\bullet$.
A picture with $r$ columns labelled by $f_{j_1},\ldots,f_{j_r}$ represents a $\Z$-graded $R$-submodule of $\bigoplus_{i=1}^r M^i$ spanned by the black circles $\bullet$, where $\bullet\!\!=\!\!\bullet$ and $\bullet\!\!=\!\!\bullet\!\!=\!\!\bullet$ mean the sum of all $\bullet$.
If we have $p=1$, then we omit the label $f$ for simplicity. 

\subsection{Finite Cohen--Macaulay representation type}

In each Auslander--Reiten quiver, we illustrate the standard tilting object $V$ (\cf Section~\ref{ss:structure theorem for non-negative a-invariants}) 
as the direct sum of the framed objects.

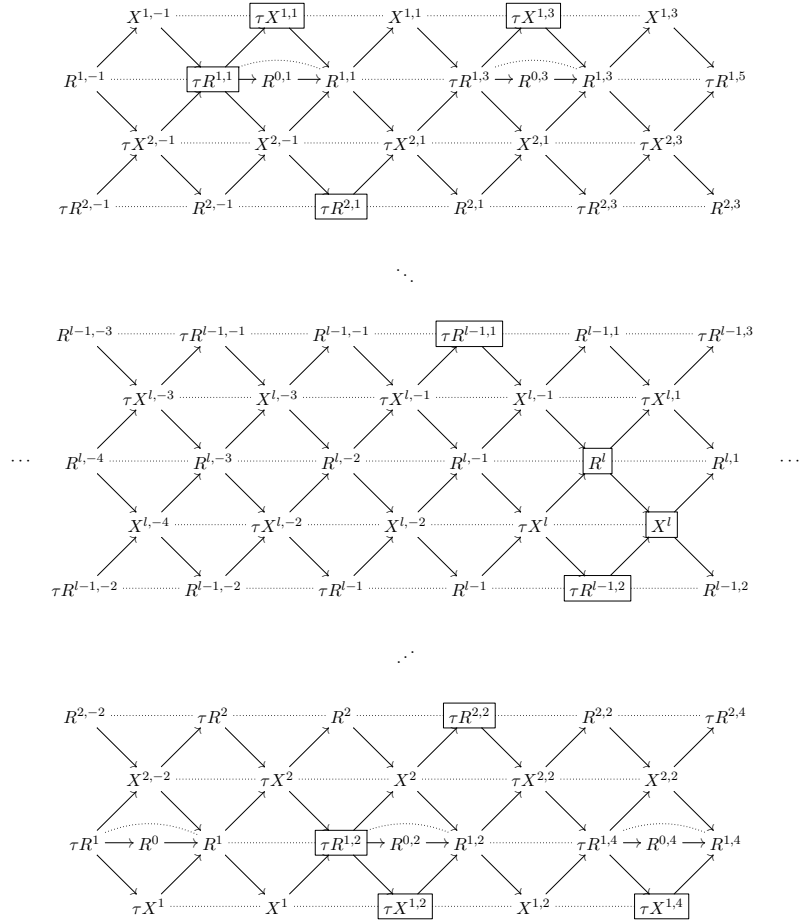
\begin{equation*}
\begin{gathered}
\begin{minipage}{0.67\textwidth}
  \centering
\scalebox{0.6}{
\begin{xy} 0;<20pt,0pt>:<0pt,20pt>::
(6,8) *+{\cdots},
(18,8) *+{\iddots},
(30,8) *+{\cdots},
(8,14) *+{R^{l, -4}} ="70",
(10,16) *+{X^{-4}} ="80",
(10,14) *+{Y^{-4}} ="90",
(8,10) *+{R^{l-2, -3}} ="51",
(10,12) *+{R^{l-1, -3}} ="61",
(12,14) *+{R^{l, -3}} ="71",
(14,16) *+{Y^{-3}} ="81",
(14,14) *+{X^{-3}} ="91",
(8,6) *+{R^{3, -3}} ="42",
(12,10) *+{R^{l-2, -2}} ="52",
(14,12) *+{R^{l-1, -2}} ="62",
(16,14) *+{R^{l, -2}} ="72",
(18,16) *+{X^{-2}} ="82",
(18,14) *+{Y^{-2}} ="92",
(8,2) *+{R^{1, -1}} ="23",
(10,4) *+{R^{2, -1}} ="33",
(12,6) *+{R^{3, -1}} ="43",
(16,10) *+{R^{l-2, -1}} ="53",
(18,12) *+{R^{l-1, -1}} ="63",
(20,14) *+{R^{l, -1}} ="73",
(22,16) *+{Y^{-1}} ="83",
(22,14) *+{X^{-1}} ="93",
(10,0) *+{R^0} ="14",
(12,2) *+[F]{R^1} ="24",
(14,4) *+[F]{R^2} ="34",
(16,6) *+[F]{R^3} ="44",
(20,10) *+[F]{R^{l-2}} ="54",
(22,12) *+[F]{R^{l-1}} ="64",
(24,14) *+[F]{R^{l}} ="74",
(26,16) *+[F]{X} ="84",
(26,14) *+[F]{Y} ="94",
(14,0) *+{R^{0, 1}} ="15",
(16,2) *+{R^{1, 1}} ="25",
(18,4) *+{R^{2, 1}} ="35",
(20,6) *+{R^{3, 1}} ="45",
(24,10) *+{R^{l-2, 1}} ="55",
(26,12) *+{R^{l-1, 1}} ="65",
(28,14) *+{R^{l, 1}} ="75",
(18,0) *+{R^{0, 2}} ="16",
(20,2) *+{R^{1, 2}} ="26",
(22,4) *+{R^{2, 2}} ="36",
(24,6) *+{R^{3, 2}} ="46",
(28,10) *+{R^{l-2, 2}} ="56",
(22,0) *+{R^{0, 3}} ="17",
(24,2) *+{R^{1, 3}} ="27",
(26,4) *+{R^{2, 3}} ="37",
(28,6) *+{R^{3, 3}} ="47",
(26,0) *+{R^{0, 4}} ="18",
(28,2) *+{R^{1, 4}} ="28",
"70", {\ar"80"},
"70", {\ar"90"},
"70", {\ar"61"},
"80", {\ar"71"},
"90", {\ar"71"},
"70", {\ar@/^1pc/@{.} "71"},
"80", {\ar@{.} "81"},
"90", {\ar@/_1pc/@{.} "91"},
"51", {\ar"61"},
"61", {\ar"71"},
"71", {\ar"81"},
"71", {\ar"91"},
"61", {\ar"52"},
"71", {\ar"62"},
"81", {\ar"72"},
"91", {\ar"72"},
"51", {\ar@{.} "52"},
"61", {\ar@{.} "62"},
"71", {\ar@/^1pc/@{.} "72"},
"81", {\ar@{.} "82"},
"91", {\ar@/_1pc/@{.} "92"},
"52", {\ar"62"},
"62", {\ar"72"},
"72", {\ar"82"},
"72", {\ar"92"},
"42", {\ar"33"},
"62", {\ar"53"},
"72", {\ar"63"},
"82", {\ar"73"},
"92", {\ar"73"},
"42", {\ar@{.} "43"},
"52", {\ar@{.} "53"},
"62", {\ar@{.} "63"},
"72", {\ar@/^1pc/@{.} "73"},
"82", {\ar@{.} "83"},
"92", {\ar@/_1pc/@{.} "93"},
"23", {\ar"33"},
"33", {\ar"43"},
"53", {\ar"63"},
"63", {\ar"73"},
"73", {\ar"83"},
"73", {\ar"93"},
"23", {\ar"14"},
"33", {\ar"24"},
"43", {\ar"34"},
"63", {\ar"54"},
"73", {\ar"64"},
"83", {\ar"74"},
"93", {\ar"74"},
"23", {\ar@{.} "24"},
"33", {\ar@{.} "34"},
"43", {\ar@{.} "44"},
"53", {\ar@{.} "54"},
"63", {\ar@{.} "64"},
"73", {\ar@/^1pc/@{.} "74"},
"83", {\ar@{.} "84"},
"93", {\ar@/_1pc/@{.} "94"},
"14", {\ar"24"},
"24", {\ar"34"},
"34", {\ar"44"},
"54", {\ar"64"},
"64", {\ar"74"},
"74", {\ar"84"},
"74", {\ar"94"},
"24", {\ar"15"},
"34", {\ar"25"},
"44", {\ar"35"},
"64", {\ar"55"},
"74", {\ar"65"},
"84", {\ar"75"},
"94", {\ar"75"},
"24", {\ar@{.} "25"},
"34", {\ar@{.} "35"},
"44", {\ar@{.} "45"},
"54", {\ar@{.} "55"},
"64", {\ar@{.} "65"},
"74", {\ar@/^1pc/@{.} "75"},
"15", {\ar"25"},
"25", {\ar"35"},
"35", {\ar"45"},
"55", {\ar"65"},
"65", {\ar"75"},
"25", {\ar"16"},
"35", {\ar"26"},
"45", {\ar"36"},
"65", {\ar"56"},
"25", {\ar@{.} "26"},
"35", {\ar@{.} "36"},
"45", {\ar@{.} "46"},
"55", {\ar@{.} "56"},
"16", {\ar"26"},
"26", {\ar"36"},
"36", {\ar"46"},
"26", {\ar"17"},
"36", {\ar"27"},
"46", {\ar"37"},
"26", {\ar@{.} "27"},
"36", {\ar@{.} "37"},
"46", {\ar@{.} "47"},
"17", {\ar"27"},
"27", {\ar"37"},
"37", {\ar"47"},
"27", {\ar"18"},
"37", {\ar"28"},
"27", {\ar@{.} "28"},
"18", {\ar"28"},
\end{xy}
}\\[10pt]
\mbox{\footnotesize type $A_n$ ($n$ is odd), where $l=\frac{n-1}{2}$ and $*^j=*(j)$ for $*=R^i$, $X$, or $Y$}
\end{minipage}
\hfill
\begin{minipage}{0.33\textwidth}
  \centering
$\begin{smallmatrix}
& \mathclap{\small \mbox{$R^i$ ($i\geq 0$)}} & & \quad & \mathclap{\small \mbox{$X$}} & \quad & \mathclap{\small \mbox{$Y$}} & & \\[6pt]
\mbox{\tiny $f_1$} & & \mbox{\tiny $f_2$} & \quad & \mbox{\tiny $f_1$} & \quad & \mbox{\tiny $f_2$} & & \\[2pt]
\bigbullet & \sequal & \bigbullet & \quad & \bigbullet & \quad & \bigbullet & & \makebox[0pt][l]{\tiny $0$} \\[-5pt]
\vdots & & \vdots & \quad & \vdots & \quad & \vdots & & \\
\bigbullet & \sequal & \bigbullet & \quad & \bigbullet & \quad & \bigbullet & & \makebox[0pt][l]{\tiny $l-i$} \\
\bigbullet & & \bigbullet & \quad & \bigbullet & \quad & \bigbullet & & \makebox[0pt][l]{\tiny $l-i+1$} \\[-5pt]
\vdots & & \vdots & \quad & \vdots & \quad & \vdots & &
\end{smallmatrix}$
\\[10pt]
\mbox{\footnotesize $k$-basis of $\Z$-graded modules}
\end{minipage}
\end{gathered}
\end{equation*}

\begin{equation*}
\begin{gathered}
\begin{minipage}{0.67\textwidth}
  \centering
\scalebox{0.6}{
\begin{xy} 0;<20pt,0pt>:<0pt,20pt>::
(6,11) *+{\cdots},
(16,14) *+{\ddots},
(16,6) *+{\iddots},
(24,11) *+{\cdots},
(8,22) *+{R^{0, -1}} ="e0",
%
(8,18) *+{R^{2, -3}} ="c1",
(10,20) *+[F]{R^{1, -1}} ="d1",
(12,22) *+{R^{0, 1}} ="e1",
%
(10,16) *+{R^{3, -3}} ="b2",
(12,18) *+[F]{R^{2, -1}} ="c2",
(14,20) *+{R^{1, 1}} ="d2",
(16,22) *+{R^{0, 3}} ="e2",
%
(8,10) *+{R^{l, -6}} ="73",
(10,12) *+{R^{l, -5}} ="83",
(14,16) *+[F]{R^{3, -1}} ="b3",
(16,18) *+{R^{2, 1}} ="c3",
(18,20) *+{R^{1, 3}} ="d3",
(20,22) *+{R^{0, 5}} ="e3",
%
(10,8) *+{R^{l-1, -4}} ="64",
(12,10) *+{R^{l, -4}} ="74",
(14,12) *+{R^{l, -3}} ="84",
(18,16) *+{R^{3, 1}} ="b4",
(20,18) *+{R^{2, 3}} ="c4",
(22,20) *+{R^{1, 5}} ="d4",
%
(8,2) *+{R^{1, -2}} ="25",
(10,4) *+{R^{2, -2}} ="35",
(14,8) *+{R^{l-1, -2}} ="65",
(16,10) *+{R^{l, -2}} ="75",
(18,12) *+[F]{R^{l, -1}} ="85",
(22,16) *+{R^{3, 3}} ="b5",
%
(10,0) *+{R^0} ="16",
(12,2) *+[F]{R^1} ="26",
(14,4) *+[F]{R^2} ="36",
(18,8) *+[F]{R^{l-1}} ="66",
(20,10) *+[F]{R^{l}} ="76",
(22,12) *+{R^{l, 1}} ="86",
%
(14,0) *+{R^{0, 2}} ="17",
(16,2) *+{R^{1, 2}} ="27",
(18,4) *+{R^{2, 2}} ="37",
(22,8) *+{R^{l-1, 2}} ="67",
(18,0) *+{R^{0, 4}} ="18",
(20,2) *+{R^{1, 4}} ="28",
(22,4) *+{R^{2, 4}} ="38",
%
(22,0) *+{R^{0, 6}} ="19",
%
%
"e0", {\ar"d1"},
"c1", {\ar"d1"},
"d1", {\ar"e1"},
%
"c1", {\ar"b2"},
"d1", {\ar"c2"},
"e1", {\ar"d2"},
"c1", {\ar@{.} "c2"},
"d1", {\ar@{.} "d2"},
%
"b2", {\ar"c2"},
"c2", {\ar"d2"},
"d2", {\ar"e2"},
%
"c2", {\ar"b3"},
"d2", {\ar"c3"},
"e2", {\ar"d3"},
"b2", {\ar@{.} "b3"},
"c2", {\ar@{.} "c3"},
"d2", {\ar@{.} "d3"},
%
"73", {\ar"83"},
"b3", {\ar"c3"},
"c3", {\ar"d3"},
"d3", {\ar"e3"},
%
"73", {\ar"64"},
"83", {\ar"74"},
"c3", {\ar"b4"},
"d3", {\ar"c4"},
"e3", {\ar"d4"},
"73", {\ar@{.} "74"},
"83", {\ar@{.} "84"},
"b3", {\ar@{.} "b4"},
"c3", {\ar@{.} "c4"},
"d3", {\ar@{.} "d4"},
%
"64", {\ar"74"},
"74", {\ar"84"},
"b4", {\ar"c4"},
"c4", {\ar"d4"},
%
"74", {\ar"65"},
"84", {\ar"75"},
"c4", {\ar"b5"},
"64", {\ar@{.} "65"},
"74", {\ar@{.} "75"},
"84", {\ar@{.} "85"},
"b4", {\ar@{.} "b5"},
%
"25", {\ar"35"},
"65", {\ar"75"},
"75", {\ar"85"},
%
"25", {\ar"16"},
"35", {\ar"26"},
"75", {\ar"66"},
"85", {\ar"76"},
"25", {\ar@{.} "26"},
"35", {\ar@{.} "36"},
"65", {\ar@{.} "66"},
"75", {\ar@{.} "76"},
"85", {\ar@{.} "86"},
%
"16", {\ar"26"},
"26", {\ar"36"},
"66", {\ar"76"},
"76", {\ar"86"},
%
"26", {\ar"17"},
"36", {\ar"27"},
"76", {\ar"67"},
"26", {\ar@{.} "27"},
"36", {\ar@{.} "37"},
"66", {\ar@{.} "67"},
"17", {\ar"27"},
"27", {\ar"37"},
%
"27", {\ar"18"},
"37", {\ar"28"},
"27", {\ar@{.} "28"},
"37", {\ar@{.} "38"},
%
"18", {\ar"28"},
"28", {\ar"38"},
%
"28", {\ar"19"},
%
%
\end{xy}
}\\[10pt]
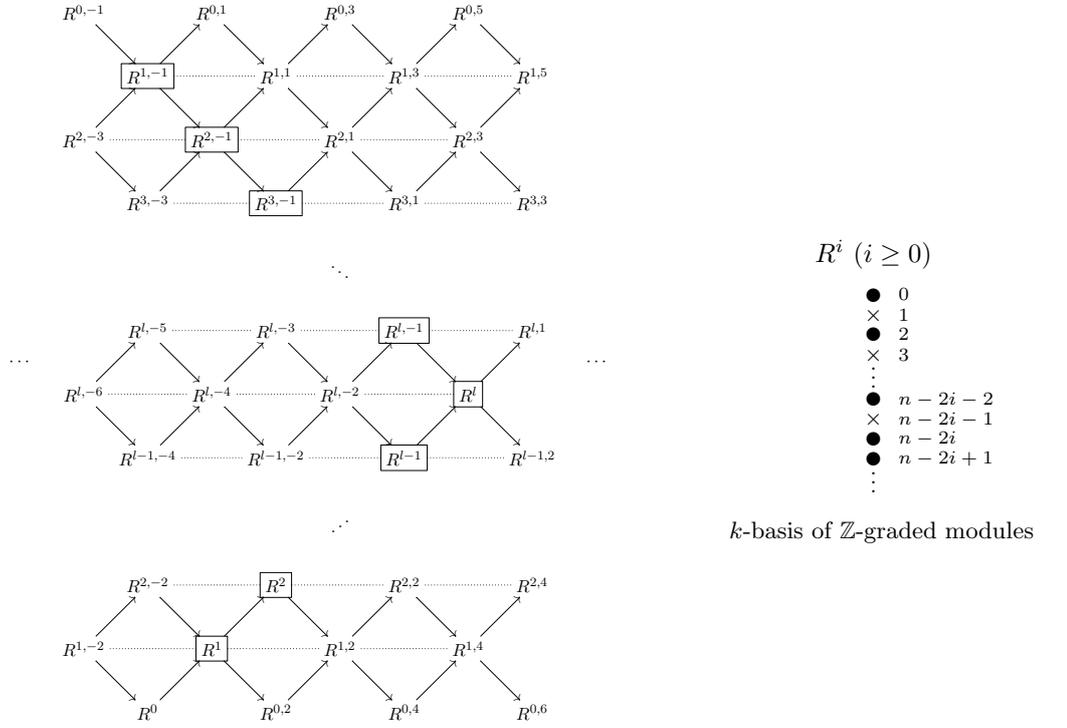
\mbox{\footnotesize type $A_n$ ($n$ is even), where $l=\frac{n}{2}$ and $*^j=*(j)$ for $*=R^i$}
\end{minipage}
\hfill
\begin{minipage}{0.33\textwidth}
  \centering
$\begin{smallmatrix}
\mathclap{\mbox{\small $R^i$ ($i\geq 0$)}} & & \\[6pt]
\bigbullet & & \makebox[0pt][l]{\tiny $0$} \\
\times & & \makebox[0pt][l]{\tiny $1$} \\
\bigbullet & & \makebox[0pt][l]{\tiny $2$} \\
\times & & \makebox[0pt][l]{\tiny $3$} \\[-5pt]
\vdots & \\
\bigbullet & & \makebox[0pt][l]{\tiny $n-2i-2$} \\
\times & & \makebox[0pt][l]{\tiny $n-2i-1$} \\
\bigbullet & & \makebox[0pt][l]{\tiny $n-2i$} \\
\bigbullet & & \makebox[0pt][l]{\tiny $n-2i+1$} \\[-5pt]
\vdots & &
\end{smallmatrix}$
\\[10pt]
\mbox{\footnotesize $k$-basis of $\Z$-graded modules}
\end{minipage}
\end{gathered}
\end{equation*}

\begin{equation*}
\begin{gathered}
\scalebox{0.6}{
\begin{xy} 0;<20pt,0pt>:<0pt,20pt>::
(6,8) *+{\cdots},
(18,8) *+{\iddots},
(34,8) *+{\cdots},
(8,14) *+{R^{l, -2}} ="71",
(10,16) *+{\tau X^{-2}} ="81",
(10,14) *+{\tau Y^{-2}} ="a1",
(8,10) *+{\tau R^{l-1, -1}} ="52",
(10,12) *+{Z^{l, -2}} ="62",
(12,14) *+{\tau R^{l, -1}} ="72",
(14,16) *+{X^{-2}} ="82",
(14,14) *+{Y^{-2}} ="a2",
(8,6) *+{R^{2, -1}} ="43",
(12,10) *+{R^{l-1, -1}} ="53",
(14,12) *+{\tau Z^{l, -1}} ="63",
(16,14) *+{R^{l, -1}} ="73",
(18,16) *+{\tau X^{-1}} ="83",
(18,14) *+{\tau Y^{-1}} ="a3",
(8,2) *+{\tau R^1} ="24",
(10,4) *+{Z^{2, -1}} ="34",
(12,6) *+{\tau R^2} ="44",
(16,10) *+{\tau R^{l-1}} ="54",
(18,12) *+{Z^{l, -1}} ="64",
(20,14) *+{\tau R^{l}} ="74",
(22,16) *+{X^{-1}} ="84",
(10,2) *+{R^0} ="94",
(22,14) *+{Y^{-1}} ="a4",
(10,0) *+{\tau Z^1} ="15",
(12,2) *+{R^1} ="25",
(14,4) *+{\tau Z^2} ="35",
(16,6) *+{R^2} ="45",
(20,10) *+{R^{l-1}} ="55",
(22,12) *+{\tau Z^{l}} ="65",
(24,14) *+{R^{l}} ="75",
(26,16) *+{\tau X} ="85",
(26,14) *+{\tau Y} ="a5",
(14,0) *+{Z^1} ="16",
(16,2) *+[F]{\tau R^{1, 1}} ="26",
(18,4) *+{Z^2} ="36",
(20,6) *+[F]{\tau R^{2, 1}} ="46",
(24,10) *+[F]{\tau R^{l-1, 1}} ="56",
(26,12) *+{Z^{l}} ="66",
(28,14) *+[F]{\tau R^{l, 1}} ="76",
(30,16) *+[F]{X} ="86",
(18,2) *+{R^{0, 1}} ="96",
(30,14) *+[F]{Y} ="a6",
(18,0) *+[F]{\tau Z^{1, 1}} ="17",
(20,2) *+{R^{1, 1}} ="27",
(22,4) *+{\tau Z^{2, 1}} ="37",
(24,6) *+{R^{2, 1}} ="47",
(28,10) *+{R^{l-1, 1}} ="57",
(30,12) *+{\tau Z^{l, 1}} ="67",
(32,14) *+{R^{l, 1}} ="77",
(22,0) *+{Z^{1, 1}} ="18",
(24,2) *+{\tau R^{1, 2}} ="28",
(26,4) *+{Z^{2, 1}} ="38",
(28,6) *+{\tau R^{2, 2}} ="48",
(32,10) *+{\tau R^{l-1, 2}} ="58",
(26,2) *+{R^{0, 2}} ="98",
(26,0) *+[F]{\tau Z^{1, 2}} ="19",
(28,2) *+{R^{1, 2}} ="29",
(30,4) *+{\tau Z^{2, 2}} ="39",
(32,6) *+{R^{2, 2}} ="49",
(30,0) *+{Z^{1, 2}} ="110",
(32,2) *+{\tau R^{1, 3}} ="210",
"71", {\ar"81"},
"71", {\ar"a1"},
"71", {\ar"62"},
"81", {\ar"72"},
"a1", {\ar"72"},
"71", {\ar@/^1pc/@{.} "72"},
"81", {\ar@{.} "82"},
"a1", {\ar@/_1pc/@{.} "a2"},
"52", {\ar"62"},
"62", {\ar"72"},
"72", {\ar"82"},
"72", {\ar"a2"},
"62", {\ar"53"},
"72", {\ar"63"},
"82", {\ar"73"},
"a2", {\ar"73"},
"52", {\ar@{.} "53"},
"62", {\ar@{.} "63"},
"72", {\ar@/^1pc/@{.} "73"},
"82", {\ar@{.} "83"},
"a2", {\ar@/_1pc/@{.} "a3"},
"53", {\ar"63"},
"63", {\ar"73"},
"73", {\ar"83"},
"73", {\ar"a3"},
"43", {\ar"34"},
"43", {\ar@{.} "44"},
"63", {\ar"54"},
"73", {\ar"64"},
"83", {\ar"74"},
"a3", {\ar"74"},
"53", {\ar@{.} "54"},
"63", {\ar@{.} "64"},
"73", {\ar@/^1pc/@{.} "74"},
"83", {\ar@{.} "84"},
"a3", {\ar@/_1pc/@{.} "a4"},
"24", {\ar"34"},
"34", {\ar"44"},
"54", {\ar"64"},
"64", {\ar"74"},
"74", {\ar"84"},
"24", {\ar"94"},
"74", {\ar"a4"},
"24", {\ar"15"},
"34", {\ar"25"},
"44", {\ar"35"},
"64", {\ar"55"},
"74", {\ar"65"},
"84", {\ar"75"},
"94", {\ar"25"},
"a4", {\ar"75"},
"24", {\ar@/^1pc/@{.} "25"},
"34", {\ar@{.} "35"},
"44", {\ar@{.} "45"},
"54", {\ar@{.} "55"},
"64", {\ar@{.} "65"},
"74", {\ar@/^1pc/@{.} "75"},
"84", {\ar@{.} "85"},
"a4", {\ar@/_1pc/@{.} "a5"},
"15", {\ar"25"},
"25", {\ar"35"},
"35", {\ar"45"},
"55", {\ar"65"},
"65", {\ar"75"},
"75", {\ar"85"},
"75", {\ar"a5"},
"25", {\ar"16"},
"35", {\ar"26"},
"45", {\ar"36"},
"65", {\ar"56"},
"75", {\ar"66"},
"85", {\ar"76"},
"a5", {\ar"76"},
"15", {\ar@{.} "16"},
"25", {\ar@{.} "26"},
"35", {\ar@{.} "36"},
"45", {\ar@{.} "46"},
"55", {\ar@{.} "56"},
"65", {\ar@{.} "66"},
"35", {\ar@{.} "36"},
"45", {\ar@{.} "46"},
"75", {\ar@/^1pc/@{.} "76"},
"85", {\ar@{.} "86"},
"a5", {\ar@/_1pc/@{.} "a6"},
"16", {\ar"26"},
"26", {\ar"36"},
"36", {\ar"46"},
"56", {\ar"66"},
"66", {\ar"76"},
"76", {\ar"86"},
"76", {\ar"a6"},
"26", {\ar"96"},
"26", {\ar"17"},
"36", {\ar"27"},
"46", {\ar"37"},
"66", {\ar"57"},
"76", {\ar"67"},
"86", {\ar"77"},
"96", {\ar"27"},
"a6", {\ar"77"},
"16", {\ar@{.} "17"},
"26", {\ar@/^1pc/@{.} "27"},
"36", {\ar@{.} "37"},
"46", {\ar@{.} "47"},
"56", {\ar@{.} "57"},
"66", {\ar@{.} "67"},
"76", {\ar@/^1pc/@{.} "77"},
"17", {\ar"27"},
"27", {\ar"37"},
"37", {\ar"47"},
"57", {\ar"67"},
"67", {\ar"77"},
"27", {\ar"18"},
"37", {\ar"28"},
"47", {\ar"38"},
"67", {\ar"58"},
"17", {\ar@{.} "18"},
"27", {\ar@{.} "28"},
"37", {\ar@{.} "38"},
"47", {\ar@{.} "48"},
"57", {\ar@{.} "58"},
"18", {\ar"28"},
"28", {\ar"38"},
"38", {\ar"48"},
"28", {\ar"98"},
"28", {\ar"19"},
"38", {\ar"29"},
"48", {\ar"39"},
"98", {\ar"29"},
"18", {\ar@{.} "19"},
"28", {\ar@/^1pc/@{.} "29"},
"38", {\ar@{.} "39"},
"48", {\ar@{.} "49"},
"19", {\ar"29"},
"29", {\ar"39"},
"39", {\ar"49"},
"29", {\ar"110"},
"39", {\ar"210"},
"19", {\ar@{.} "110"},
"29", {\ar@{.} "210"},
"110", {\ar"210"},
\end{xy}
}\\[10pt]
\mbox{\footnotesize type $D_n$ ($n$ is even), where $l=\frac{n}{2}-1$ and $*^j=*(j)$ for $*=R^i$, $X$, $Y$, $Z^i$, $\tau R^i$, $\tau X$, $\tau Y$, or $\tau Z^i$}
\end{gathered}
\end{equation*}

In the above Auslander--Reiten quiver, the last framed object in the bottom row is $\tau Z^{1, l}$.

\[
\begin{gathered}
\strut\kern-3em
\begin{smallmatrix}
 & & \mathclap{\mbox{\small $R^{0}$}} & & & \quad & & & \mathclap{\mbox{\small $R^{i}$ ($i\geq 1$)}} & & & \quad & \mathclap{\mbox{\small $X$}} & \quad & \mathclap{\mbox{\small $Y$}} & \quad & & \mathclap{\mbox{\small $Z^i$ ($i\geq 1$)}} & & \quad \; \; & & & \mathclap{\mbox{\small $\tau R^{i}$ ($i\geq 1$)}} & & & \quad & & \mathclap{\mbox{\small $\tau X$}} & & \quad & & \mathclap{\mbox{\small $\tau Y$}} & & \quad & \mathclap{\mbox{\small $\tau Z^1$}} & \quad & & & & \mathclap{\mbox{\small $\tau Z^i$ ($i\geq 2$)}} & & & & & \\[6pt]
\mbox{\tiny $x$} & & \mbox{\tiny $f_1$} & & \mbox{\tiny $f_2$} & \quad & \mbox{\tiny $x$} & & \mbox{\tiny $f_1$} & & \mbox{\tiny $f_2$} & \quad & \mbox{\tiny $f_1$} & \quad & \mbox{\tiny $f_2$} & \quad & \mbox{\tiny $f_1$} & & \mbox{\tiny $f_2$} & \quad \; \; & \mbox{\tiny $x$} & & \mbox{\tiny $f_1$} & & \mbox{\tiny $f_2$} & \quad & \mbox{\tiny $x$} & & \mbox{\tiny $f_2$} & \quad & \mbox{\tiny $x$} & & \mbox{\tiny $f_1$} & \quad & \mbox{\tiny $x$} & \quad & \mbox{\tiny $x$} & & \mbox{\tiny $f_1$} & & \mbox{\tiny $f_2$} & & \mbox{\tiny $x$} & & \\[2pt]
\bigbullet & \sequal & \bigbullet & \sequal & \bigbullet & \quad & \bigbullet & \sequal & \bigbullet & \sequal & \bigbullet & \quad & \bigbullet & \quad & \bigbullet & \quad & \bigbullet & \sequal & \bigbullet & \quad \; \; & & & \times & & \times & \quad & \bigbullet & \sequal & \bigbullet & \quad & \bigbullet & \sequal & \bigbullet & \quad & \times & \quad & \bigbullet & \sequal & \bigbullet & \sequal & \bigbullet & & & & \makebox[0pt][l]{\tiny $0$} \\
 & & \bigbullet & \sequal & \bigbullet & \quad & & & \bigbullet & \sequal & \bigbullet & \quad & \bigbullet & \quad & \bigbullet & \quad & \bigbullet & \sequal & \bigbullet & \quad \; \; & & & \bigbullet & \sequal & \bigbullet & \quad & & & \bigbullet & \quad & & & \bigbullet & \quad & & \quad & & & \bigbullet & \sequal & \bigbullet & & & & \makebox[0pt][l]{\tiny $1$} \\[-5pt]
 & & \vdots & & \vdots & \quad & & & \vdots & & \vdots & \quad & \vdots & \quad & \vdots & \quad & \vdots & & \vdots & \quad \; \; & & & \vdots & & \vdots & \quad & & & \vdots & \quad & & & \vdots & \quad & & \quad & & & \vdots & & \vdots & & & & \\
 & & \bigbullet & \sequal & \bigbullet & \quad & & & \bigbullet & \sequal & \bigbullet & \quad & \bigbullet & \quad & \bigbullet & \quad & \bigbullet & \sequal & \bigbullet & \quad \; \; & & & \bigbullet & \sequal & \bigbullet & \quad & & & \bigbullet & \quad & & & \bigbullet & \quad & & \quad & & & \bigbullet & \sequal & \bigbullet & & & & \makebox[0pt][l]{\tiny $l-i$} \\
 & & \bigbullet & \sequal & \bigbullet & \quad & & & \bigbullet & & \bigbullet & \quad & \bigbullet & \quad & \bigbullet & \quad & \bigbullet & & \bigbullet & \quad \; \; & \bigbullet & \sminus & \bigbullet & \sminus & \bigbullet & \quad & & & \bigbullet & \quad & & & \bigbullet & \quad & & \quad & & & \bigbullet & \sminus & \bigbullet & \sminus & \bigbullet & & \makebox[0pt][l]{\tiny $l-i+1$} \\
 & & \bigbullet & \sequal & \bigbullet & \quad & & & \bigbullet & & \bigbullet & \quad & \bigbullet & \quad & \bigbullet & \quad & \bigbullet & & \bigbullet & \quad \; \; & & & \bigbullet & & \bigbullet & \quad & & & \bigbullet & \quad & & & \bigbullet & \quad & & \quad & & & \bigbullet & & \bigbullet & & & & \makebox[0pt][l]{\tiny $l-i+2$} \\[-5pt]
 & & \vdots & & \vdots & \quad & & & \vdots & & \vdots & \quad & \vdots & \quad & \vdots & \quad & \vdots & & \vdots & \quad \; \; & & & \vdots & & \vdots & \quad & & & \vdots & \quad & & & \vdots & \quad & & \quad & & & \vdots & & \vdots & & & & \\
 & & \bigbullet & \sequal & \bigbullet & \quad & & & \bigbullet & & \bigbullet & \quad & \bigbullet & \quad & \bigbullet & \quad & \bigbullet & & \bigbullet & \quad \; \; & & & \bigbullet & & \bigbullet & \quad & & & \bigbullet & \quad & & & \bigbullet & \quad & & \quad & & & \bigbullet & & \bigbullet & & & & \makebox[0pt][l]{\tiny $l-1$} \\
\bigbullet & \sminus & \bigbullet & \sminus & \bigbullet & \quad & \bigbullet & & \bigbullet & & \bigbullet & \quad & \bigbullet & \quad & \bigbullet & \quad & \bigbullet & & \bigbullet & \quad \; \; & & & \bigbullet & & \bigbullet & \quad & \bigbullet & & \bigbullet & \quad & \bigbullet & & \bigbullet & \quad & \bigbullet & \quad & \bigbullet & & \bigbullet & & \bigbullet & & & & \makebox[0pt][l]{\tiny $l$} \\
 & & \bigbullet & & \bigbullet & \quad & & & \bigbullet & & \bigbullet & \quad & \bigbullet & \quad & \bigbullet & \quad & \bigbullet & & \bigbullet & \quad \; \; & & & \bigbullet & & \bigbullet & \quad & & & \bigbullet & \quad & & & \bigbullet & \quad & & \quad & & & \bigbullet & & \bigbullet & & & & \makebox[0pt][l]{\tiny $l+1$} \\[-5pt]
 & & \vdots & & \vdots & \quad & & & \vdots & & \vdots & \quad & \vdots & \quad & \vdots & \quad & \vdots & & \vdots & \quad \; \; & & & \vdots & & \vdots & \quad & & & \vdots & \quad & & & \vdots & \quad & & \quad & & & \vdots & & \vdots & & & & \\
 & & \bigbullet & & \bigbullet & \quad & & & \bigbullet & & \bigbullet & \quad & \bigbullet & \quad & \bigbullet & \quad & \bigbullet & & \bigbullet & \quad \; \; & & & \bigbullet & & \bigbullet & \quad & & & \bigbullet & \quad & & & \bigbullet & \quad & & \quad & & & \bigbullet & & \bigbullet & & & & \makebox[0pt][l]{\tiny $n-i-2$} \\
 & & \bigbullet & & \bigbullet & \quad & & & \bigbullet & & \bigbullet & \quad & \bigbullet & \quad & \bigbullet & \quad & \bigbullet & & \bigbullet & \quad \; \; & \bigbullet & & \bigbullet & & \bigbullet & \quad & & & \bigbullet & \quad & & & \bigbullet & \quad & & \quad & & & \bigbullet & & \bigbullet & & \bigbullet & & \makebox[0pt][l]{\tiny $n-i-1$} \\
 & & \bigbullet & & \bigbullet & \quad & & & \bigbullet & & \bigbullet & \quad & \bigbullet & \quad & \bigbullet & \quad & \bigbullet & & \bigbullet & \quad \; \; & & & \bigbullet & & \bigbullet & \quad & & & \bigbullet & \quad & & & \bigbullet & \quad & & \quad & & & \bigbullet & & \bigbullet & & & & \makebox[0pt][l]{\tiny $n-i$} \\[-5pt]
 & & \vdots & & \vdots & \quad & & & \vdots & & \vdots & \quad & \vdots & \quad & \vdots & \quad & \vdots & & \vdots & \quad \; \; & & & \vdots & & \vdots & \quad & & & \vdots & \quad & & & \vdots & \quad & & \quad & & & \vdots & & \vdots & & & & \\
 & & \bigbullet & & \bigbullet & \quad & & & \bigbullet & & \bigbullet & \quad & \bigbullet & \quad & \bigbullet & \quad & \bigbullet & & \bigbullet & \quad \; \; & & & \bigbullet & & \bigbullet & \quad & & & \bigbullet & \quad & & & \bigbullet & \quad & & \quad & & & \bigbullet & & \bigbullet & & & & \makebox[0pt][l]{\tiny $n-3$} \\
\bigbullet & & \bigbullet & & \bigbullet & \quad & \bigbullet & & \bigbullet & & \bigbullet & \quad & \bigbullet & \quad & \bigbullet & \quad & \bigbullet & & \bigbullet & \quad \; \; & & & \bigbullet & & \bigbullet & \quad & \bigbullet & & \bigbullet & \quad & \bigbullet & & \bigbullet & \quad & \bigbullet & \quad & \bigbullet & & \bigbullet & & \bigbullet & & & & \makebox[0pt][l]{\tiny $n-2$} \\
 & & \bigbullet & & \bigbullet & \quad & & & \bigbullet & & \bigbullet & \quad & \bigbullet & \quad & \bigbullet & \quad & \bigbullet & & \bigbullet & \quad \; \; & & & \bigbullet & & \bigbullet & \quad & & & \bigbullet & \quad & & & \bigbullet & \quad & & \quad & & & \bigbullet & & \bigbullet & & & & \makebox[0pt][l]{\tiny $n-1$} \\[-5pt]
\vdots & & \vdots & & \vdots & \quad & \vdots & & \vdots & & \vdots & \quad & \vdots & \quad & \vdots & \quad & \vdots & & \vdots & \quad \; \; & \vdots & & \vdots & & \vdots & \quad & \vdots & & \vdots & \quad & \vdots & & \vdots & \quad & \vdots & \quad & \vdots & & \vdots & & \vdots & & \vdots & &
\end{smallmatrix} \\[10pt]
\mbox{\footnotesize $k$-basis of $\Z$-graded modules for type $D_n$ ($n$ is even)}
\end{gathered}
\]

The symbol $\bullet\!\!-\!\!\bullet\!\!-\!\!\bullet$ means the subspace $\{(a,b,c)\in k^3\mid 2a-b+c=0\}$, where the natural basis of $k^3$ is given by the black circles in the columns labelled by $x$, $f_1$, $f_2$, respectively.

\begin{equation*}
\begin{gathered}
\scalebox{0.6}{
\begin{xy} 0;<20pt,0pt>:<0pt,20pt>::
(6,14) *+{\cdots},
(18,20) *+{\ddots},
(18,8) *+{\iddots},
(30,14) *+{\cdots},
(8,26) *+{R^{1, -1}} ="c0",
(10,28) *+{X^{1, -1}} ="d0",
(8,18) *+{R^{l-1, -3}} ="91",
(8,22) *+{\tau R^{2, -1}} ="a1",
(10,24) *+{\tau X^{2, -1}} ="b1",
(12,26) *+[F]{\tau R^{1, 1}} ="c1",
(14,28) *+[F]{\tau X^{1, 1}} ="d1",
(14,26) *+{R^{0, 1}} ="f1",
(8,14) *+{R^{l, -4}} ="72",
(10,16) *+{\tau X^{l, -3}} ="82",
(12,18) *+{\tau R^{l-1, -1}} ="92",
(12,22) *+{R^{2, -1}} ="a2",
(14,24) *+{X^{2, -1}} ="b2",
(16,26) *+{R^{1, 1}} ="c2",
(18,28) *+{X^{1, 1}} ="d2",
(8,10) *+{\tau R^{l-1, -2}} ="53",
(10,12) *+{X^{l, -4}} ="63",
(12,14) *+{R^{l, -3}} ="73",
(14,16) *+{X^{l, -3}} ="83",
(16,18) *+{R^{l-1, -1}} ="93",
(16,22) *+[F]{\tau R^{2, 1}} ="a3",
(18,24) *+{\tau X^{2, 1}} ="b3",
(20,26) *+{\tau R^{1, 3}} ="c3",
(22,28) *+[F]{\tau X^{1, 3}} ="d3",
(22,26) *+{R^{0, 3}} ="f3",
(8,6) *+{R^{2, -2}} ="44",
(12,10) *+{R^{l-1, -2}} ="54",
(14,12) *+{\tau X^{l, -2}} ="64",
(16,14) *+{R^{l, -2}} ="74",
(18,16) *+{\tau X^{l, -1}} ="84",
(20,18) *+[F]{\tau R^{l-1, 1}} ="94",
(20,22) *+{R^{2, 1}} ="a4",
(22,24) *+{X^{2, 1}} ="b4",
(24,26) *+{R^{1, 3}} ="c4",
(26,28) *+{X^{1, 3}} ="d4",
(8,2) *+{\tau R^1} ="25",
(10,4) *+{X^{2, -2}} ="35",
(12,6) *+{\tau R^2} ="45",
(16,10) *+{\tau R^{l-1}} ="55",
(18,12) *+{X^{l, -2}} ="65",
(20,14) *+{R^{l, -1}} ="75",
(22,16) *+{X^{l, -1}} ="85",
(24,18) *+{R^{l-1, 1}} ="95",
(24,22) *+{\tau R^{2, 3}} ="a5",
(26,24) *+{\tau X^{2, 3}} ="b5",
(28,26) *+{\tau R^{1, 5}} ="c5",
(10,2) *+{R^0} ="e5",
(10,0) *+{\tau X^1} ="16",
(12,2) *+{R^1} ="26",
(14,4) *+{\tau X^2} ="36",
(16,6) *+{R^2} ="46",
(20,10) *+{R^{l-1}} ="56",
(22,12) *+{\tau X^{l}} ="66",
(24,14) *+[F]{R^{l}} ="76",
(26,16) *+{\tau X^{l, 1}} ="86",
(28,18) *+{\tau R^{l-1, 3}} ="96",
(28,22) *+{R^{2, 3}} ="a6",
(14,0) *+{X^1} ="17",
(16,2) *+[F]{\tau R^{1, 2}} ="27",
(18,4) *+{X^2} ="37",
(20,6) *+[F]{\tau R^{2, 2}} ="47",
(24,10) *+[F]{\tau R^{l-1, 2}} ="57",
(26,12) *+[F]{X^{l}} ="67",
(28,14) *+{R^{l, 1}} ="77",
(18,2) *+{R^{0, 2}} ="e7",
(18,0) *+[F]{\tau X^{1, 2}} ="18",
(20,2) *+{R^{1, 2}} ="28",
(22,4) *+{\tau X^{2, 2}} ="38",
(24,6) *+{R^{2, 2}} ="48",
(28,10) *+{R^{l-1, 2}} ="58",
(22,0) *+{X^{1, 2}} ="19",
(24,2) *+{\tau R^{1, 4}} ="29",
(26,4) *+{X^{2, 2}} ="39",
(28,6) *+{\tau R^{2, 4}} ="49",
(26,2) *+{R^{0, 4}} ="e9",
(26,0) *+[F]{\tau X^{1, 4}} ="110",
(28,2) *+{R^{1, 4}} ="210",
"c0", {\ar"d0"},
"c0", {\ar"b1"},
"d0", {\ar"c1"},
"c0", {\ar@{.} "c1"},
"d0", {\ar@{.} "d1"},
"a1", {\ar"b1"},
"b1", {\ar"c1"},
"c1", {\ar"d1"},
"c1", {\ar"f1"},
"91", {\ar"82"},
"b1", {\ar"a2"},
"c1", {\ar"b2"},
"d1", {\ar"c2"},
"f1", {\ar"c2"},
"91", {\ar@{.} "92"},
"a1", {\ar@{.} "a2"},
"b1", {\ar@{.} "b2"},
"c1", {\ar@/^1pc/@{.} "c2"},
"d1", {\ar@{.} "d2"},
"72", {\ar"82"},
"82", {\ar"92"},
"a2", {\ar"b2"},
"b2", {\ar"c2"},
"c2", {\ar"d2"},
"72", {\ar"63"},
"82", {\ar"73"},
"92", {\ar"83"},
"b2", {\ar"a3"},
"c2", {\ar"b3"},
"d2", {\ar"c3"},
"72", {\ar@{.} "73"},
"82", {\ar@{.} "83"},
"92", {\ar@{.} "93"},
"a2", {\ar@{.} "a3"},
"b2", {\ar@{.} "b3"},
"c2", {\ar@{.} "c3"},
"d2", {\ar@{.} "d3"},
"53", {\ar"63"},
"63", {\ar"73"},
"73", {\ar"83"},
"83", {\ar"93"},
"a3", {\ar"b3"},
"b3", {\ar"c3"},
"c3", {\ar"d3"},
"c3", {\ar"f3"},
"63", {\ar"54"},
"73", {\ar"64"},
"83", {\ar"74"},
"93", {\ar"84"},
"b3", {\ar"a4"},
"c3", {\ar"b4"},
"d3", {\ar"c4"},
"f3", {\ar"c4"},
"53", {\ar@{.} "54"},
"63", {\ar@{.} "64"},
"73", {\ar@{.} "74"},
"83", {\ar@{.} "84"},
"93", {\ar@{.} "94"},
"a3", {\ar@{.} "a4"},
"b3", {\ar@{.} "b4"},
"c3", {\ar@/^1pc/@{.} "c4"},
"d3", {\ar@{.} "d4"},
"54", {\ar"64"},
"64", {\ar"74"},
"74", {\ar"84"},
"84", {\ar"94"},
"a4", {\ar"b4"},
"b4", {\ar"c4"},
"c4", {\ar"d4"},
"44", {\ar"35"},
"64", {\ar"55"},
"74", {\ar"65"},
"84", {\ar"75"},
"94", {\ar"85"},
"b4", {\ar"a5"},
"c4", {\ar"b5"},
"d4", {\ar"c5"},
"44", {\ar@{.} "45"},
"54", {\ar@{.} "55"},
"64", {\ar@{.} "65"},
"74", {\ar@{.} "75"},
"84", {\ar@{.} "85"},
"94", {\ar@{.} "95"},
"a4", {\ar@{.} "a5"},
"b4", {\ar@{.} "b5"},
"c4", {\ar@{.} "c5"},
"25", {\ar"35"},
"35", {\ar"45"},
"55", {\ar"65"},
"65", {\ar"75"},
"75", {\ar"85"},
"85", {\ar"95"},
"a5", {\ar"b5"},
"b5", {\ar"c5"},
"25", {\ar"e5"},
"25", {\ar"16"},
"35", {\ar"26"},
"45", {\ar"36"},
"65", {\ar"56"},
"75", {\ar"66"},
"85", {\ar"76"},
"95", {\ar"86"},
"b5", {\ar"a6"},
"e5", {\ar"26"},
"25", {\ar@/^1pc/@{.} "26"},
"35", {\ar@{.} "36"},
"45", {\ar@{.} "46"},
"55", {\ar@{.} "56"},
"65", {\ar@{.} "66"},
"75", {\ar@{.} "76"},
"85", {\ar@{.} "86"},
"95", {\ar@{.} "96"},
"a5", {\ar@{.} "a6"},
"16", {\ar"26"},
"26", {\ar"36"},
"36", {\ar"46"},
"56", {\ar"66"},
"66", {\ar"76"},
"76", {\ar"86"},
"86", {\ar"96"},
"26", {\ar"17"},
"36", {\ar"27"},
"46", {\ar"37"},
"66", {\ar"57"},
"76", {\ar"67"},
"86", {\ar"77"},
"16", {\ar@{.} "17"},
"26", {\ar@{.} "27"},
"36", {\ar@{.} "37"},
"46", {\ar@{.} "47"},
"56", {\ar@{.} "57"},
"66", {\ar@{.} "67"},
"76", {\ar@{.} "77"},
"17", {\ar"27"},
"27", {\ar"37"},
"37", {\ar"47"},
"57", {\ar"67"},
"67", {\ar"77"},
"27", {\ar"e7"},
"27", {\ar"18"},
"37", {\ar"28"},
"47", {\ar"38"},
"67", {\ar"58"},
"e7", {\ar"28"},
"17", {\ar@{.} "18"},
"27", {\ar@/^1pc/@{.} "28"},
"37", {\ar@{.} "38"},
"47", {\ar@{.} "48"},
"57", {\ar@{.} "58"},
"18", {\ar"28"},
"28", {\ar"38"},
"38", {\ar"48"},
"28", {\ar"19"},
"38", {\ar"29"},
"48", {\ar"39"},
"18", {\ar@{.} "19"},
"28", {\ar@{.} "29"},
"38", {\ar@{.} "39"},
"48", {\ar@{.} "49"},
"19", {\ar"29"},
"29", {\ar"39"},
"39", {\ar"49"},
"29", {\ar"e9"},
"29", {\ar"110"},
"39", {\ar"210"},
"e9", {\ar"210"},
"19", {\ar@{.} "110"},
"29", {\ar@/^1pc/@{.} "210"},
"110", {\ar"210"},
\end{xy}
}\\[10pt]
\mbox{\footnotesize type $D_n$ ($n$ is odd), where $l=\frac{n-1}{2}$ and $*^j=*(j)$ for $*=R^i$, $X^i$, $\tau R^i$, or $\tau X^i$}
\end{gathered}
\end{equation*}

In the above Auslander--Reiten quiver, the last framed object in the bottom (respectively, top) row is $\tau X^{1, n-3}$ (respectively, $\tau X^{1, n-2}$).

\[
\begin{gathered}
\begin{smallmatrix}
 & \mathclap{\mbox{\small $R^0$}} & & \qquad & & \mathclap{\mbox{\small $R^i$ ($i\geq 1$)}} & & \qquad \quad & \mathclap{\mbox{\small $X^i$ ($i\geq 1$)}} & \qquad \quad & & \mathclap{\mbox{\small $\tau R^i$ ($i\geq 1$)}} & & \qquad & \mathclap{\mbox{\small $\tau X^1$}} & \qquad & & & \mathclap{\mbox{\small $\tau X^i$ ($i\geq 2$)}} & & & & \\[6pt]
\mbox{\tiny $x$} & & \mbox{\tiny $f_1$} & \qquad & \mbox{\tiny $x$} & & \mbox{\tiny $f_1$} & \qquad \quad & f_1 & \qquad \quad & \mbox{\tiny $x$} & & \mbox{\tiny $f_1$} & \qquad & \mbox{\tiny $x$} & \qquad & \mbox{\tiny $x$} & & \mbox{\tiny $f_1$} & & \mbox{\tiny $x$} & & \\[2pt]
\bigbullet & \sequal & \bigbullet & \qquad & \bigbullet & \sequal & \bigbullet & \qquad \quad & \bigbullet & \qquad \quad & & & \times & \qquad & \times & \qquad & \bigbullet & \sequal & \bigbullet & & & & \makebox[0pt][l]{\tiny $0$} \\
 & & \times & \qquad & & & \times & \qquad \quad & \times & \qquad \quad & & & \times & \qquad & & \qquad & & & \times & & & & \makebox[0pt][l]{\tiny $1$} \\
 & & \bigbullet & \qquad & & & \bigbullet & \qquad \quad & \bigbullet & \qquad \quad & & & \bigbullet & \qquad & & \qquad & & & \bigbullet & & & & \makebox[0pt][l]{\tiny $2$} \\[-5pt]
 & & \vdots & \qquad & & & \vdots & \qquad \quad & \vdots & \qquad \quad & & & \vdots & \qquad & & \qquad & & & \vdots & & & & \\
 & & \bigbullet & \qquad & & & \bigbullet & \qquad \quad & \bigbullet & \qquad \quad & & & \bigbullet & \qquad & & \qquad & & & \bigbullet & & & & \makebox[0pt][l]{\tiny $n-2i-3$} \\
 & & \times & \qquad & & & \times & \qquad \quad & \times & \qquad \quad & & & \times & \qquad & & \qquad & & & \times & & & & \makebox[0pt][l]{\tiny $n-2i-2$} \\
 & & \bigbullet & \qquad & & & \bigbullet & \qquad \quad & \bigbullet & \qquad \quad & & & \bigbullet & \qquad & & \qquad & & & \bigbullet & & & & \makebox[0pt][l]{\tiny $n-2i-1$} \\
 & & \times & \qquad & & & \bigbullet & \qquad \quad & \bigbullet & \qquad \quad & \bigbullet & \sequal & \bigbullet & \qquad & & \qquad & & & \bigbullet & \sequal & \bigbullet & & \makebox[0pt][l]{\tiny $n-2i$} \\
 & & \bigbullet & \qquad & & & \bigbullet & \qquad \quad & \bigbullet & \qquad \quad & & & \bigbullet & \qquad & & \qquad & & & \bigbullet & & & & \makebox[0pt][l]{\tiny $n-2i+1$} \\[-5pt]
 & & \vdots & \qquad & & & \vdots & \qquad \quad & \vdots & \qquad \quad & & & \vdots & \qquad & & \qquad & & & \vdots & & & & \\
 & & \times & \qquad & & & \bigbullet & \qquad \quad & \bigbullet & \qquad \quad & & & \bigbullet & \qquad & & \qquad & & & \bigbullet & & & & \makebox[0pt][l]{\tiny $n-4$} \\
 & & \bigbullet & \qquad & & & \bigbullet & \qquad \quad & \bigbullet & \qquad \quad & & & \bigbullet & \qquad & & \qquad & & & \bigbullet & & & & \makebox[0pt][l]{\tiny $n-3$} \\
\bigbullet & \sequal & \bigbullet & \qquad & \bigbullet & & \bigbullet & \qquad \quad & \bigbullet & \qquad \quad & & & \bigbullet & \qquad & \bigbullet & \qquad & \bigbullet & & \bigbullet & & & & \makebox[0pt][l]{\tiny $n-2$} \\
 & & \bigbullet & \qquad & & & \bigbullet & \qquad \quad & \bigbullet & \qquad \quad & & & \bigbullet & \qquad & & \qquad & & & \bigbullet & & & & \makebox[0pt][l]{\tiny $n-1$} \\[-5pt]
 & & \vdots & \qquad & & & \vdots & \qquad \quad & \vdots & \qquad \quad & & & \vdots & \qquad & & \qquad & & & \vdots & & & & \\
 & & \bigbullet & \qquad & & & \bigbullet & \qquad \quad & \bigbullet & \qquad \quad & & & \bigbullet & \qquad & & \qquad & & & \bigbullet & & & & \makebox[0pt][l]{\tiny $2n-2i-3$} \\
 & & \bigbullet & \qquad & & & \bigbullet & \qquad \quad & \bigbullet & \qquad \quad & \bigbullet & & \bigbullet & \qquad & & \qquad & & & \bigbullet & & \bigbullet & & \makebox[0pt][l]{\tiny $2n-2i-2$} \\
 & & \bigbullet & \qquad & & & \bigbullet & \qquad \quad & \bigbullet & \qquad \quad & & & \bigbullet & \qquad & & \qquad & & & \bigbullet & & & & \makebox[0pt][l]{\tiny $2n-2i-1$} \\[-5pt]
 & & \vdots & \qquad & & & \vdots & \qquad \quad & \vdots & \qquad \quad & & & \vdots & \qquad & & \qquad & & & \vdots & & & & \\
 & & \bigbullet & \qquad & & & \bigbullet & \qquad \quad & \bigbullet & \qquad \quad & & & \bigbullet & \qquad & & \qquad & & & \bigbullet & & & & \makebox[0pt][l]{\tiny $2n-5$} \\
\bigbullet & & \bigbullet & \qquad & \bigbullet & & \bigbullet & \qquad \quad & \bigbullet & \qquad \quad & & & \bigbullet & \qquad & \bigbullet & \qquad & \bigbullet & & \bigbullet & & & & \makebox[0pt][l]{\tiny $2n-4$} \\
 & & \bigbullet & \qquad & & & \bigbullet & \qquad \quad & \bigbullet & \qquad \quad & & & \bigbullet & \qquad & & \qquad & & & \bigbullet & & & & \makebox[0pt][l]{\tiny $2n-3$} \\[-5pt]
\vdots & & \vdots & \qquad & \vdots & & \vdots & \qquad \quad & \vdots & \qquad \quad & \vdots & & \vdots & \qquad & \vdots & \qquad & \vdots & & \vdots & & \vdots & &
\end{smallmatrix} \\[10pt]
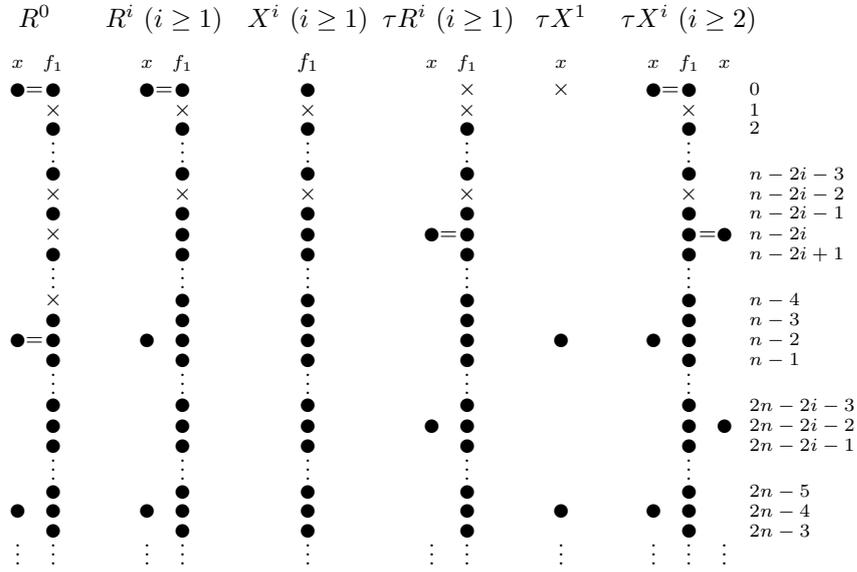
\mbox{\footnotesize $k$-basis of $\Z$-graded modules for type $D_n$ ($n$ is odd)}
\end{gathered}
\]

From now on, the top row of each picture represents the component of degree $0$.

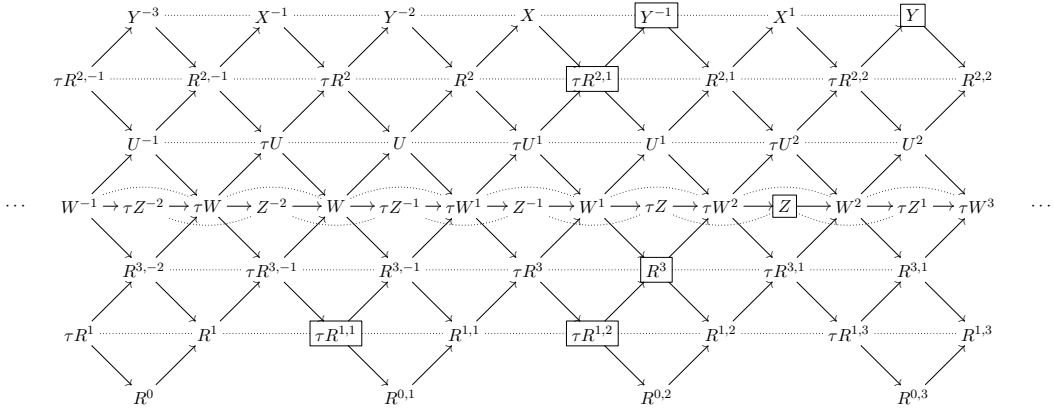
\begin{equation*}
\begin{gathered}
\scalebox{0.6}{
\begin{xy} 0;<20pt,0pt>:<0pt,20pt>::
(6,6) *+{\cdots},
(26,6) *+{\cdots},
(8,10) *+{R^{1, -1}} ="60",
(8,6) *+{X^{-2}} ="41",
(10,8) *+{\tau R^{3, -2}} ="51",
(12,10) *+[F]{\tau R^{1, 1}} ="61",
(14,12) *+{R^{0, 1}} ="71",
(10,6) *+{R^{2, -2}} ="81",
(8,2) *+{\tau R^1} ="22",
(10,4) *+{R^{3, -3}} ="32",
(12,6) *+{X^{-1}} ="42",
(14,8) *+{R^{3, -2}} ="52",
(16,10) *+{R^{1, 1}} ="62",
(14,6) *+{R^{2, -1}} ="82",
(10,0) *+{R^0} ="13",
(12,2) *+{R^1} ="23",
(14,4) *+{\tau R^{3, -1}} ="33",
(16,6) *+{X} ="43",
(18,8) *+{\tau R^3} ="53",
(20,10) *+[F]{\tau R^{1, 3}} ="63",
(22,12) *+{R^{0, 3}} ="73",
(18,6) *+[F]{R^2} ="83",
(16,2) *+[F]{\tau R^{1, 2}} ="24",
(18,4) *+[F]{R^{3, -1}} ="34",
(20,6) *+{X^1} ="44",
(22,8) *+[F]{R^3} ="54",
(24,10) *+{R^{1, 3}} ="64",
(22,6) *+{R^{2, 1}} ="84",
(18,0) *+{R^{0, 2}} ="15",
(20,2) *+{R^{1, 2}} ="25",
(22,4) *+{\tau R^{3, 1}} ="35",
(24,6) *+{X^2} ="45",
(24,2) *+{\tau R^{1, 4}} ="26",
"60", {\ar"51"},
"60", {\ar@{.} "61"},
"41", {\ar"51"},
"51", {\ar"61"},
"61", {\ar"71"},
"41", {\ar"81"},
"41", {\ar"32"},
"51", {\ar"42"},
"61", {\ar"52"},
"71", {\ar"62"},
"81", {\ar"42"},
"41", {\ar@/^1pc/@{.} "42"},
"51", {\ar@{.} "52"},
"61", {\ar@{.} "62"},
"81", {\ar@/_1pc/@{.} "82"},
"22", {\ar"32"},
"32", {\ar"42"},
"42", {\ar"52"},
"52", {\ar"62"},
"42", {\ar"82"},
"22", {\ar"13"},
"32", {\ar"23"},
"42", {\ar"33"},
"52", {\ar"43"},
"62", {\ar"53"},
"82", {\ar"43"},
"22", {\ar@{.} "23"},
"32", {\ar@{.} "33"},
"42", {\ar@/^1pc/@{.} "43"},
"52", {\ar@{.} "53"},
"62", {\ar@{.} "63"},
"82", {\ar@/_1pc/@{.} "83"},
"13", {\ar"23"},
"23", {\ar"33"},
"33", {\ar"43"},
"43", {\ar"53"},
"53", {\ar"63"},
"63", {\ar"73"},
"43", {\ar"83"},
"33", {\ar"24"},
"43", {\ar"34"},
"53", {\ar"44"},
"63", {\ar"54"},
"73", {\ar"64"},
"83", {\ar"44"},
"23", {\ar@{.} "24"},
"33", {\ar@{.} "34"},
"43", {\ar@/^1pc/@{.} "44"},
"53", {\ar@{.} "54"},
"63", {\ar@{.} "64"},
"83", {\ar@/_1pc/@{.} "84"},
"24", {\ar"34"},
"34", {\ar"44"},
"44", {\ar"54"},
"54", {\ar"64"},
"44", {\ar"84"},
"24", {\ar"15"},
"34", {\ar"25"},
"44", {\ar"35"},
"54", {\ar"45"},
"84", {\ar"45"},
"24", {\ar@{.} "25"},
"34", {\ar@{.} "35"},
"44", {\ar@/^1pc/@{.} "45"},
"15", {\ar"25"},
"25", {\ar"35"},
"35", {\ar"45"},
"35", {\ar"26"},
"25", {\ar@{.} "26"},
\end{xy}
}\\[10pt]
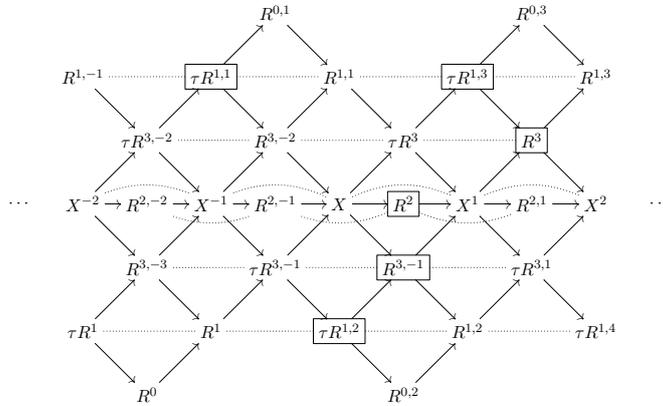
\mbox{\footnotesize type $E_6$, where $*^j=*(j)$ for $*=R^0$, $R^1$, $R^2$, $R^3$, $\tau R^1$, $\tau R^3$, or $X$}
\end{gathered}
\end{equation*}

\[
\begin{gathered}
\begin{smallmatrix}
\mathclap{\mbox{\small $R^0$}} & \quad & \mathclap{\mbox{\small $R^1$}} & \quad & \mathclap{\mbox{\small $R^2$}} & \quad & \mathclap{\mbox{\small $R^3$}} & \quad & \mathclap{\mbox{\small $\tau R^1$}} & & \quad & \mathclap{\mbox{\small $\tau R^{3, -1}$}} & & \quad & & \mathclap{\mbox{\small $X$}} & \\[6pt]
\bigbullet & \quad & \bigbullet & \quad & \bigbullet & \quad & \bigbullet & \quad & \times & \quad & \bigbullet & & \times & \quad & \bigbullet & & \times \\
\times & \quad & \times & \quad & \times & \quad & \bigbullet & \quad & \times & \quad & \times & & \bigbullet & \quad & \times & & \bigbullet \\
\times & \quad & \times & \quad & \bigbullet & \quad & \bigbullet & \quad & \times & \quad & \bigbullet & \sequal & \bigbullet & \quad & \bigbullet & \sequal & \bigbullet \\
\bigbullet & \quad & \bigbullet & \quad & \bigbullet & \quad & \bigbullet & \quad & \bigbullet & \quad & \bigbullet & & \times & \quad & \bigbullet & & \bigbullet \\
\bigbullet & \quad & \bigbullet & \quad & \bigbullet & \quad & \bigbullet & \quad & \bigbullet & \quad & \bigbullet & & \bigbullet & \quad & \bigbullet & & \bigbullet \\
\times & \quad & \bigbullet & \quad & \bigbullet & \quad & \bigbullet & \quad & \times & \quad & \bigbullet & & \bigbullet & \quad & \bigbullet & & \bigbullet \\
\bigbullet & \quad & \bigbullet & \quad & \bigbullet & \quad & \bigbullet & \quad & \bigbullet & \quad & \bigbullet & & \bigbullet & \quad & \bigbullet & & \bigbullet \\[-5pt]
\vdots & \quad & \vdots & \quad & \vdots & \quad & \vdots & \quad & \vdots & \quad & \vdots & & \vdots & \quad & \vdots & & \vdots
\end{smallmatrix} \\[10pt]
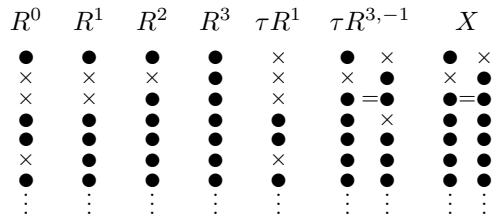
\mbox{\footnotesize $k$-basis of $\Z$-graded modules for type $E_6$}
\end{gathered}
\]

\begin{equation*}
\begin{gathered}
\scalebox{0.6}{
\begin{xy} 0;<20pt,0pt>:<0pt,20pt>::
(6,6) *+{\cdots},
(38,6) *+{\cdots},
(8,10) *+{\tau R^{2, -1}} ="60",
(10,12) *+{Y^{-3}} ="70",
(8,6) *+{W^{-1}} ="41",
(10,8) *+{U^{-1}} ="51",
(12,10) *+{R^{2, -1}} ="61",
(14,12) *+{X^{-1}} ="71",
(10,6) *+{\tau Z^{-2}} ="81",
(8,2) *+{\tau R^1} ="22",
(10,4) *+{R^{3, -2}} ="32",
(12,6) *+{\tau W} ="42",
(14,8) *+{\tau U} ="52",
(16,10) *+{\tau R^2} ="62",
(18,12) *+{Y^{-2}} ="72",
(14,6) *+{Z^{-2}} ="82",
(10,0) *+{R^0} ="13",
(12,2) *+{R^1} ="23",
(14,4) *+{\tau R^{3, -1}} ="33",
(16,6) *+{W} ="43",
(18,8) *+{U} ="53",
(20,10) *+{R^2} ="63",
(22,12) *+{X} ="73",
(18,6) *+{\tau Z^{-1}} ="83",
(16,2) *+[F]{\tau R^{1, 1}} ="24",
(18,4) *+{R^{3, -1}} ="34",
(20,6) *+{\tau W^1} ="44",
(22,8) *+{\tau U^1} ="54",
(24,10) *+[F]{\tau R^{2, 1}} ="64",
(26,12) *+[F]{Y^{-1}} ="74",
(22,6) *+{Z^{-1}} ="84",
(18,0) *+{R^{0, 1}} ="15",
(20,2) *+{R^{1, 1}} ="25",
(22,4) *+{\tau R^3} ="35",
(24,6) *+{W^1} ="45",
(26,8) *+{U^1} ="55",
(28,10) *+{R^{2, 1}} ="65",
(30,12) *+{X^1} ="75",
(26,6) *+{\tau Z} ="85",
(24,2) *+[F]{\tau R^{1, 2}} ="26",
(26,4) *+[F]{R^3} ="36",
(28,6) *+{\tau W^2} ="46",
(30,8) *+{\tau U^2} ="56",
(32,10) *+{\tau R^{2, 2}} ="66",
(34,12) *+[F]{Y} ="76",
(30,6) *+[F]{Z} ="86",
(26,0) *+{R^{0, 2}} ="17",
(28,2) *+{R^{1, 2}} ="27",
(30,4) *+{\tau R^{3, 1}} ="37",
(32,6) *+{W^2} ="47",
(34,8) *+{U^2} ="57",
(36,10) *+{R^{2, 2}} ="67",
(34,6) *+{\tau Z^1} ="87",
(32,2) *+{\tau R^{1, 3}} ="28",
(34,4) *+{R^{3, 1}} ="38",
(36,6) *+{\tau W^3} ="48",
(34,0) *+{R^{0, 3}} ="19",
(36,2) *+{R^{1, 3}} ="29",
"60", {\ar"70"},
"60", {\ar"51"},
"70", {\ar"61"},
"60", {\ar@{.} "61"},
"70", {\ar@{.} "71"},
"41", {\ar"51"},
"51", {\ar"61"},
"61", {\ar"71"},
"41", {\ar"81"},
"41", {\ar"32"},
"51", {\ar"42"},
"61", {\ar"52"},
"71", {\ar"62"},
"81", {\ar"42"},
"41", {\ar@/^1pc/@{.} "42"},
"51", {\ar@{.} "52"},
"61", {\ar@{.} "62"},
"71", {\ar@{.} "72"},
"81", {\ar@/_1pc/@{.} "82"},
"22", {\ar"32"},
"32", {\ar"42"},
"42", {\ar"52"},
"52", {\ar"62"},
"62", {\ar"72"},
"42", {\ar"82"},
"22", {\ar"13"},
"32", {\ar"23"},
"42", {\ar"33"},
"52", {\ar"43"},
"62", {\ar"53"},
"72", {\ar"63"},
"82", {\ar"43"},
"22", {\ar@{.} "23"},
"32", {\ar@{.} "33"},
"42", {\ar@/^1pc/@{.} "43"},
"52", {\ar@{.} "53"},
"62", {\ar@{.} "63"},
"72", {\ar@{.} "73"},
"82", {\ar@/_1pc/@{.} "83"},
"13", {\ar"23"},
"23", {\ar"33"},
"33", {\ar"43"},
"43", {\ar"53"},
"53", {\ar"63"},
"63", {\ar"73"},
"43", {\ar"83"},
"33", {\ar"24"},
"43", {\ar"34"},
"53", {\ar"44"},
"63", {\ar"54"},
"73", {\ar"64"},
"83", {\ar"44"},
"23", {\ar@{.} "24"},
"33", {\ar@{.} "34"},
"43", {\ar@/^1pc/@{.} "44"},
"53", {\ar@{.} "54"},
"63", {\ar@{.} "64"},
"73", {\ar@{.} "74"},
"83", {\ar@/_1pc/@{.} "84"},
"24", {\ar"34"},
"34", {\ar"44"},
"44", {\ar"54"},
"54", {\ar"64"},
"64", {\ar"74"},
"44", {\ar"84"},
"24", {\ar"15"},
"34", {\ar"25"},
"44", {\ar"35"},
"54", {\ar"45"},
"64", {\ar"55"},
"74", {\ar"65"},
"84", {\ar"45"},
"24", {\ar@{.} "25"},
"34", {\ar@{.} "35"},
"44", {\ar@/^1pc/@{.} "45"},
"54", {\ar@{.} "55"},
"64", {\ar@{.} "65"},
"74", {\ar@{.} "75"},
"84", {\ar@/_1pc/@{.} "85"},
"15", {\ar"25"},
"25", {\ar"35"},
"35", {\ar"45"},
"45", {\ar"55"},
"55", {\ar"65"},
"65", {\ar"75"},
"45", {\ar"85"},
"35", {\ar"26"},
"45", {\ar"36"},
"55", {\ar"46"},
"65", {\ar"56"},
"75", {\ar"66"},
"85", {\ar"46"},
"25", {\ar@{.} "26"},
"35", {\ar@{.} "36"},
"45", {\ar@/^1pc/@{.} "46"},
"55", {\ar@{.} "56"},
"65", {\ar@{.} "66"},
"75", {\ar@{.} "76"},
"85", {\ar@/_1pc/@{.} "86"},
"26", {\ar"36"},
"36", {\ar"46"},
"46", {\ar"56"},
"56", {\ar"66"},
"66", {\ar"76"},
"46", {\ar"86"},
"26", {\ar"17"},
"36", {\ar"27"},
"46", {\ar"37"},
"56", {\ar"47"},
"66", {\ar"57"},
"76", {\ar"67"},
"86", {\ar"47"},
"26", {\ar@{.} "27"},
"36", {\ar@{.} "37"},
"46", {\ar@/^1pc/@{.} "47"},
"56", {\ar@{.} "57"},
"66", {\ar@{.} "67"},
"86", {\ar@/_1pc/@{.} "87"},
"17", {\ar"27"},
"27", {\ar"37"},
"37", {\ar"47"},
"47", {\ar"57"},
"57", {\ar"67"},
"47", {\ar"87"},
"37", {\ar"28"},
"47", {\ar"38"},
"57", {\ar"48"},
"87", {\ar"48"},
"27", {\ar@{.} "28"},
"37", {\ar@{.} "38"},
"47", {\ar@/^1pc/@{.} "48"},
"28", {\ar"38"},
"38", {\ar"48"},
"28", {\ar"19"},
"38", {\ar"29"},
"28", {\ar@{.} "29"},
"19", {\ar"29"},
\end{xy}
}\\[10pt]
\mbox{\footnotesize type $E_7$, where $*^j=*(j)$ for $*=R^0$, $R^1$, $R^2$, $R^3$, $X$, $Y$, $Z$, $W$, $U$, $\tau R^1$, $\tau R^2$, $\tau R^3$, $\tau Z$, $\tau W$, or $\tau U$}
\end{gathered}
\end{equation*}

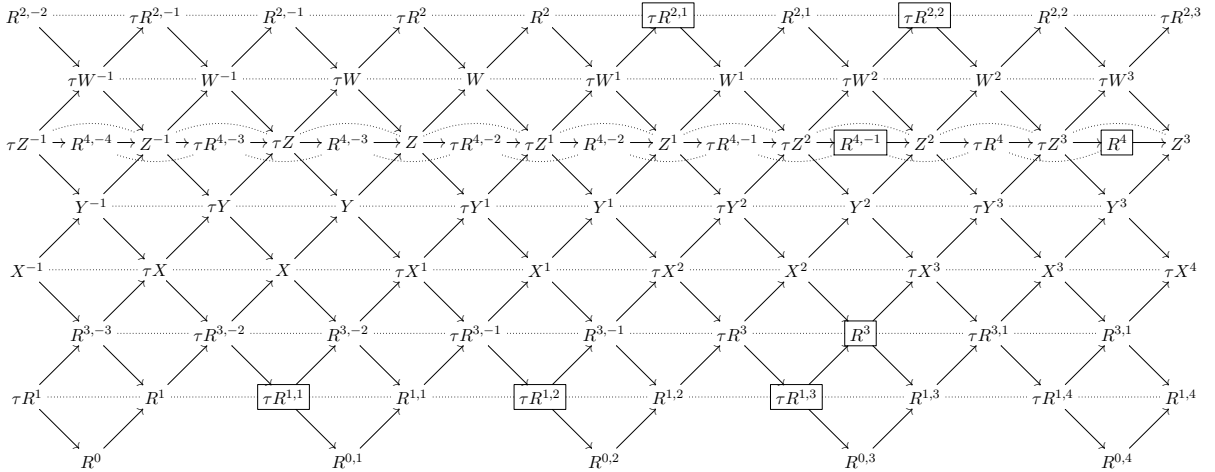
\[
\begin{gathered}
\begin{smallmatrix}
 & \mathclap{\mbox{\small $R^0$}} & & \quad & & \mathclap{\mbox{\small $R^1$}} & & \quad & & \mathclap{\mbox{\small $R^2$}} & & \quad & & \mathclap{\mbox{\small $R^3$}} & & \quad & \mathclap{\mbox{\small $X$}} & \quad & \mathclap{\mbox{\small $Y$}} & \quad & \mathclap{\mbox{\small $Z$}} & \quad & & & & \mathclap{\mbox{\small $W$}} & & & & \quad & & & \mathclap{\mbox{\small $U$}} & & \\[6pt]
\mbox{\tiny $y$} & & \mbox{\tiny $f_1$} & \quad & \mbox{\tiny $y$} & & \mbox{\tiny $f_1$} & \quad & \mbox{\tiny $y$} & & \mbox{\tiny $f_1$} & \quad & \mbox{\tiny $y$} & & \mbox{\tiny $f_1$} & \quad & \mbox{\tiny $f_1$} & \quad & \mbox{\tiny $y$} & \quad & \mbox{\tiny $f_1$} & \quad &  \mbox{\tiny $y$} & & \mbox{\tiny $f_1$} & & \mbox{\tiny $f_1$} & & \mbox{\tiny $y$} & \quad & \mbox{\tiny $y$} & & \mbox{\tiny $f_1$} & & \mbox{\tiny $f_1$} \\[2pt]
\bigbullet & \sequal & \bigbullet & \quad & \bigbullet & \sequal & \bigbullet & \quad & \bigbullet & \sequal & \bigbullet & \quad & \bigbullet & \sequal & \bigbullet & \quad & \bigbullet & \quad & \bigbullet & \quad & \bigbullet & \quad & \bigbullet & \sequal &\bigbullet & & \times & & & \quad & \bigbullet & \sequal & \bigbullet & \sequal & \bigbullet \\
 & & \times & \quad & & & \times & \quad & & & \times & \quad & & & \bigbullet & \quad & \times & \quad & & \quad & \bigbullet & \quad & & & \bigbullet & \sequal & \bigbullet & \sequal & \bigbullet & \quad & & & \bigbullet & & \times \\
\bigbullet & \sequal & \bigbullet & \quad & \bigbullet & \sequal & \bigbullet & \quad & \bigbullet & & \bigbullet & \quad & \bigbullet & & \bigbullet & \quad & \bigbullet & \quad & \bigbullet & \quad & \bigbullet & \quad & \bigbullet & \sequal & \bigbullet & & \bigbullet & & & \quad & \bigbullet & \sequal & \bigbullet & & \bigbullet \\
 & & \bigbullet & \quad &  & & \bigbullet & \quad & & & \bigbullet & \quad & & & \bigbullet & \quad & \bigbullet & \quad & & \quad & \bigbullet & \quad & & & \bigbullet & & \bigbullet & & \bigbullet & \quad & & & \bigbullet & & \bigbullet \\
\bigbullet & \sequal & \bigbullet & \quad & \bigbullet & & \bigbullet & \quad & \bigbullet & & \bigbullet & \quad & \bigbullet & & \bigbullet & \quad & \bigbullet & \quad & \bigbullet & \quad & \bigbullet & \quad & \bigbullet & & \bigbullet & & \bigbullet & & & \quad & \bigbullet & & \bigbullet & & \bigbullet \\
 & & \bigbullet & \quad &  & & \bigbullet & \quad & & & \bigbullet & \quad & & & \bigbullet & \quad & \bigbullet & \quad & & \quad & \bigbullet & \quad &  & & \bigbullet & & \bigbullet & & \bigbullet & \quad & & & \bigbullet & & \bigbullet \\
\bigbullet & & \bigbullet & \quad & \bigbullet & & \bigbullet & \quad & \bigbullet & & \bigbullet & \quad & \bigbullet & & \bigbullet & \quad & \bigbullet & \quad & \bigbullet & \quad & \bigbullet & \quad & \bigbullet & & \bigbullet & & \bigbullet & & & \quad & \bigbullet & & \bigbullet & & \bigbullet \\[-5pt]
\vdots & & \vdots & \quad & \vdots & & \vdots & \quad & \vdots & & \vdots & \quad & \vdots & & \vdots & \quad & \vdots & \quad & \vdots & \quad & \vdots & \quad & \vdots & & \vdots & & \vdots & & \vdots & \quad & \vdots & & \vdots & & \vdots
\end{smallmatrix} \\[10pt]
\begin{smallmatrix}
 & \mathclap{\mbox{\small $\tau R^1$}} & & \quad & & \mathclap{\mbox{\small $\tau R^2$}} & & \quad & & & & \mathclap{\mbox{\small $\tau R^{3, -1}$}} & & & & \quad & & & \mathclap{\mbox{\small $\tau Z^{-1}$}} & & & \quad & & & & \mathclap{\mbox{\small $\tau W$}} & & & & \quad & & & \mathclap{\mbox{\small $\tau U$}} & & \\[6pt]
\mbox{\tiny $y$} & & \mbox{\tiny $f_1$} & \quad & \mbox{\tiny $y$} & & \mbox{\tiny $f_1$} & \quad & \mbox{\tiny $y$} & & \mbox{\tiny $f_1$} & & \mbox{\tiny $f_1$} & & \mbox{\tiny $y$} & \quad & \mbox{\tiny $y$} & & \mbox{\tiny $f_1$} & & \mbox{\tiny $y$} & \quad & \mbox{\tiny $y$} & & \mbox{\tiny $f_1$} & & \mbox{\tiny $f_1$} & & \mbox{\tiny $y$} & \quad & \mbox{\tiny $y$} & & \mbox{\tiny $f_1$} & & \mbox{\tiny $y$} \\[2pt]
\times & & \times & \quad & \times & & \times & \quad & \bigbullet & \sequal & \bigbullet & & \times & & & \quad & \bigbullet & \sequal & \bigbullet & & & \quad & \times & & \times & & \times & & & \quad & & & \times & & \times \\
 & & \times & \quad & & & \bigbullet & \quad & & & \bigbullet & \sequal & \bigbullet & \sequal & \bigbullet & \quad & & & \bigbullet & \sequal & \bigbullet & \quad & & & \bigbullet & \sequal & \bigbullet & \sequal & \bigbullet & \quad & \bigbullet & \sequal & \bigbullet & & \\
\bigbullet & \sequal & \bigbullet & \quad & \bigbullet & \sequal & \bigbullet & \quad & \bigbullet & \sequal & \bigbullet & & \bigbullet & & & \quad & \bigbullet & \sequal & \bigbullet & & & \quad & \bigbullet & \sequal & \bigbullet & & \bigbullet & & & \quad & & & \bigbullet & \sequal & \bigbullet \\
 & & \bigbullet & \quad & & & \bigbullet & \quad & & & \bigbullet & & \bigbullet & \sequal & \bigbullet & \quad & & & \bigbullet & & \bigbullet & \quad & & & \bigbullet & & \bigbullet & \sequal & \bigbullet & \quad & \bigbullet & & \bigbullet & & \\
\bigbullet & \sequal & \bigbullet & \quad & \bigbullet & & \bigbullet & \quad & \bigbullet & & \bigbullet & & \bigbullet & & & \quad & \bigbullet & & \bigbullet & & & \quad & \bigbullet & & \bigbullet & & \bigbullet & & & \quad & & & \bigbullet & & \bigbullet \\
 & & \bigbullet & \quad & & & \bigbullet & \quad & & & \bigbullet & & \bigbullet & & \bigbullet & \quad & & & \bigbullet & & \bigbullet & \quad & & & \bigbullet & & \bigbullet & & \bigbullet & \quad & \bigbullet & & \bigbullet & & \\
\bigbullet & & \bigbullet & \quad & \bigbullet & & \bigbullet & \quad & \bigbullet & & \bigbullet & & \bigbullet & & & \quad & \bigbullet & & \bigbullet & & & \quad & \bigbullet & & \bigbullet & & \bigbullet & & & \quad & & & \bigbullet & & \bigbullet \\[-5pt]
\vdots & & \vdots & \quad & \vdots & & \vdots & \quad & \vdots & & \vdots & & \vdots & & \vdots & \quad & \vdots & & \vdots & & \vdots & \quad & \vdots & & \vdots & & \vdots & & \vdots & \quad & \vdots & & \vdots & & \vdots
\end{smallmatrix} \\[10pt]
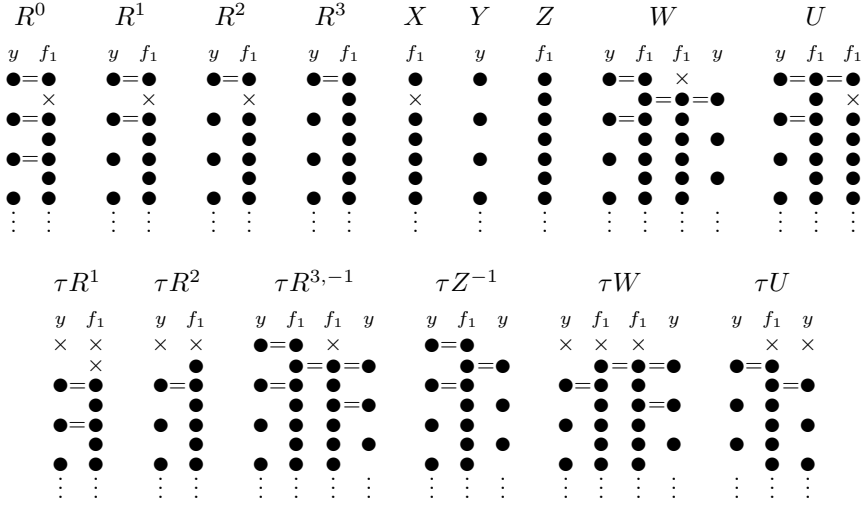
\mbox{\footnotesize $k$-basis of $\Z$-graded modules for type $E_7$}
\end{gathered}
\]

\begin{equation*}
\begin{gathered}
\strut\kern-3em\scalebox{0.60}{
\begin{xy} 0;<20pt,0pt>:<0pt,20pt>::
(6,7) *+{\cdots},
(46,7) *+{\cdots},
(8,14) *+{R^{2, -2}} ="80",
(8,10) *+{\tau Z^{-1}} ="61",
(10,12) *+{\tau W^{-1}} ="71",
(12,14) *+{\tau R^{2, -1}} ="81",
(10,10) *+{R^{4, -4}} ="91",
(8,6) *+{X^{-1}} ="42",
(10,8) *+{Y^{-1}} ="52",
(12,10) *+{Z^{-1}} ="62",
(14,12) *+{W^{-1}} ="72",
(16,14) *+{R^{2, -1}} ="82",
(14,10) *+{\tau R^{4, -3}} ="92",
(8,2) *+{\tau R^1} ="23",
(10,4) *+{R^{3, -3}} ="33",
(12,6) *+{\tau X} ="43",
(14,8) *+{\tau Y} ="53",
(16,10) *+{\tau Z} ="63",
(18,12) *+{\tau W} ="73",
(20,14) *+{\tau R^2} ="83",
(18,10) *+{R^{4, -3}} ="93",
(10,0) *+{R^0} ="14",
(12,2) *+{R^1} ="24",
(14,4) *+{\tau R^{3, -2}} ="34",
(16,6) *+{X} ="44",
(18,8) *+{Y} ="54",
(20,10) *+{Z} ="64",
(22,12) *+{W} ="74",
(24,14) *+{R^2} ="84",
(22,10) *+{\tau R^{4, -2}} ="94",
(16,2) *+[F]{\tau R^{1, 1}} ="25",
(18,4) *+{R^{3, -2}} ="35",
(20,6) *+{\tau X^1} ="45",
(22,8) *+{\tau Y^1} ="55",
(24,10) *+{\tau Z^1} ="65",
(26,12) *+{\tau W^1} ="75",
(28,14) *+[F]{\tau R^{2, 1}} ="85",
(26,10) *+{R^{4, -2}} ="95",
(18,0) *+{R^{0, 1}} ="16",
(20,2) *+{R^{1, 1}} ="26",
(22,4) *+{\tau R^{3, -1}} ="36",
(24,6) *+{X^1} ="46",
(26,8) *+{Y^1} ="56",
(28,10) *+{Z^1} ="66",
(30,12) *+{W^1} ="76",
(32,14) *+{R^{2, 1}} ="86",
(30,10) *+{\tau R^{4, -1}} ="96",
(24,2) *+[F]{\tau R^{1, 2}} ="27",
(26,4) *+{R^{3, -1}} ="37",
(28,6) *+{\tau X^2} ="47",
(30,8) *+{\tau Y^2} ="57",
(32,10) *+{\tau Z^2} ="67",
(34,12) *+{\tau W^2} ="77",
(36,14) *+[F]{\tau R^{2, 2}} ="87",
(34,10) *+[F]{R^{4, -1}} ="97",
(26,0) *+{R^{0, 2}} ="18",
(28,2) *+{R^{1, 2}} ="28",
(30,4) *+{\tau R^3} ="38",
(32,6) *+{X^2} ="48",
(34,8) *+{Y^2} ="58",
(36,10) *+{Z^2} ="68",
(38,12) *+{W^2} ="78",
(40,14) *+{R^{2, 2}} ="88",
(38,10) *+{\tau R^4} ="98",
(32,2) *+[F]{\tau R^{1, 3}} ="29",
(34,4) *+[F]{R^3} ="39",
(36,6) *+{\tau X^3} ="49",
(38,8) *+{\tau Y^3} ="59",
(40,10) *+{\tau Z^3} ="69",
(42,12) *+{\tau W^3} ="79",
(44,14) *+{\tau R^{2, 3}} ="89",
(42,10) *+[F]{R^4} ="99",
(34,0) *+{R^{0, 3}} ="110",
(36,2) *+{R^{1, 3}} ="210",
(38,4) *+{\tau R^{3, 1}} ="310",
(40,6) *+{X^3} ="410",
(42,8) *+{Y^3} ="510",
(44,10) *+{Z^3} ="610",
(40,2) *+{\tau R^{1, 4}} ="211",
(42,4) *+{R^{3, 1}} ="311",
(44,6) *+{\tau X^4} ="411",
(42,0) *+{R^{0, 4}} ="112",
(44,2) *+{R^{1, 4}} ="212",
"80", {\ar"71"},
"80", {\ar@{.} "81"},
"61", {\ar"71"},
"71", {\ar"81"},
"61", {\ar"91"},
"61", {\ar"52"},
"71", {\ar"62"},
"81", {\ar"72"},
"91", {\ar"62"},
"61", {\ar@/^1pc/@{.} "62"},
"71", {\ar@{.} "72"},
"81", {\ar@{.} "82"},
"91", {\ar@/_1pc/@{.} "92"},
"42", {\ar"52"},
"52", {\ar"62"},
"62", {\ar"72"},
"72", {\ar"82"},
"62", {\ar"92"},
"42", {\ar"33"},
"52", {\ar"43"},
"62", {\ar"53"},
"72", {\ar"63"},
"82", {\ar"73"},
"92", {\ar"63"},
"42", {\ar@{.} "43"},
"52", {\ar@{.} "53"},
"62", {\ar@/^1pc/@{.} "63"},
"72", {\ar@{.} "73"},
"82", {\ar@{.} "83"},
"92", {\ar@/_1pc/@{.} "93"},
"23", {\ar"33"},
"33", {\ar"43"},
"43", {\ar"53"},
"53", {\ar"63"},
"63", {\ar"73"},
"73", {\ar"83"},
"63", {\ar"93"},
"23", {\ar"14"},
"33", {\ar"24"},
"43", {\ar"34"},
"53", {\ar"44"},
"63", {\ar"54"},
"73", {\ar"64"},
"83", {\ar"74"},
"93", {\ar"64"},
"23", {\ar@{.} "24"},
"33", {\ar@{.} "34"},
"43", {\ar@{.} "44"},
"53", {\ar@{.} "54"},
"63", {\ar@/^1pc/@{.} "64"},
"73", {\ar@{.} "74"},
"83", {\ar@{.} "84"},
"93", {\ar@/_1pc/@{.} "94"},
"14", {\ar"24"},
"24", {\ar"34"},
"34", {\ar"44"},
"44", {\ar"54"},
"54", {\ar"64"},
"64", {\ar"74"},
"74", {\ar"84"},
"64", {\ar"94"},
"34", {\ar"25"},
"44", {\ar"35"},
"54", {\ar"45"},
"64", {\ar"55"},
"74", {\ar"65"},
"84", {\ar"75"},
"94", {\ar"65"},
"24", {\ar@{.} "25"},
"34", {\ar@{.} "35"},
"44", {\ar@{.} "45"},
"54", {\ar@{.} "55"},
"64", {\ar@/^1pc/@{.} "65"},
"74", {\ar@{.} "75"},
"84", {\ar@{.} "85"},
"94", {\ar@/_1pc/@{.} "95"},
"25", {\ar"35"},
"35", {\ar"45"},
"45", {\ar"55"},
"55", {\ar"65"},
"65", {\ar"75"},
"75", {\ar"85"},
"65", {\ar"95"},
"25", {\ar"16"},
"35", {\ar"26"},
"45", {\ar"36"},
"55", {\ar"46"},
"65", {\ar"56"},
"75", {\ar"66"},
"85", {\ar"76"},
"95", {\ar"66"},
"25", {\ar@{.} "26"},
"35", {\ar@{.} "36"},
"45", {\ar@{.} "46"},
"55", {\ar@{.} "56"},
"65", {\ar@/^1pc/@{.} "66"},
"75", {\ar@{.} "76"},
"85", {\ar@{.} "86"},
"95", {\ar@/_1pc/@{.} "96"},
"16", {\ar"26"},
"26", {\ar"36"},
"36", {\ar"46"},
"46", {\ar"56"},
"56", {\ar"66"},
"66", {\ar"76"},
"76", {\ar"86"},
"66", {\ar"96"},
"36", {\ar"27"},
"46", {\ar"37"},
"56", {\ar"47"},
"66", {\ar"57"},
"76", {\ar"67"},
"86", {\ar"77"},
"96", {\ar"67"},
"26", {\ar@{.} "27"},
"36", {\ar@{.} "37"},
"46", {\ar@{.} "47"},
"56", {\ar@{.} "57"},
"66", {\ar@/^1pc/@{.} "67"},
"76", {\ar@{.} "77"},
"86", {\ar@{.} "87"},
"96", {\ar@/_1pc/@{.} "97"},
"27", {\ar"37"},
"37", {\ar"47"},
"47", {\ar"57"},
"57", {\ar"67"},
"67", {\ar"77"},
"77", {\ar"87"},
"67", {\ar"97"},
"27", {\ar"18"},
"37", {\ar"28"},
"47", {\ar"38"},
"57", {\ar"48"},
"67", {\ar"58"},
"77", {\ar"68"},
"87", {\ar"78"},
"97", {\ar"68"},
"27", {\ar@{.} "28"},
"37", {\ar@{.} "38"},
"47", {\ar@{.} "48"},
"57", {\ar@{.} "58"},
"67", {\ar@/^1pc/@{.} "68"},
"77", {\ar@{.} "78"},
"87", {\ar@{.} "88"},
"97", {\ar@/_1pc/@{.} "98"},
"18", {\ar"28"},
"28", {\ar"38"},
"38", {\ar"48"},
"48", {\ar"58"},
"58", {\ar"68"},
"68", {\ar"78"},
"78", {\ar"88"},
"68", {\ar"98"},
"38", {\ar"29"},
"48", {\ar"39"},
"58", {\ar"49"},
"68", {\ar"59"},
"78", {\ar"69"},
"88", {\ar"79"},
"98", {\ar"69"},
"28", {\ar@{.} "29"},
"38", {\ar@{.} "39"},
"48", {\ar@{.} "49"},
"58", {\ar@{.} "59"},
"68", {\ar@/^1pc/@{.} "69"},
"78", {\ar@{.} "79"},
"88", {\ar@{.} "89"},
"98", {\ar@/_1pc/@{.} "99"},
"29", {\ar"39"},
"39", {\ar"49"},
"49", {\ar"59"},
"59", {\ar"69"},
"69", {\ar"79"},
"79", {\ar"89"},
"69", {\ar"99"},
"29", {\ar"110"},
"39", {\ar"210"},
"49", {\ar"310"},
"59", {\ar"410"},
"69", {\ar"510"},
"79", {\ar"610"},
"99", {\ar"610"},
"29", {\ar@{.} "210"},
"39", {\ar@{.} "310"},
"49", {\ar@{.} "410"},
"59", {\ar@{.} "510"},
"69", {\ar@/^1pc/@{.} "610"},
"110", {\ar"210"},
"210", {\ar"310"},
"310", {\ar"410"},
"410", {\ar"510"},
"510", {\ar"610"},
"310", {\ar"211"},
"410", {\ar"311"},
"510", {\ar"411"},
"210", {\ar@{.} "211"},
"310", {\ar@{.} "311"},
"410", {\ar@{.} "411"},
"211", {\ar"311"},
"311", {\ar"411"},
"211", {\ar"112"},
"311", {\ar"212"},
"211", {\ar@{.} "212"},
"112", {\ar"212"},
\end{xy}
}\\[10pt]
\strut\kern-3em
\mbox{\footnotesize type $E_8$, where $*^j=*(j)$ for $*=R^0$, $R^1$, $R^2$, $R^3$, $R^4$, $X$, $Y$, $Z$, $W$, $\tau R^1$, $\tau R^2$, $\tau R^3$, $\tau R^4$, $\tau X$, $\tau Y$, $\tau Z$, or $\tau W$}
\end{gathered}
\end{equation*}

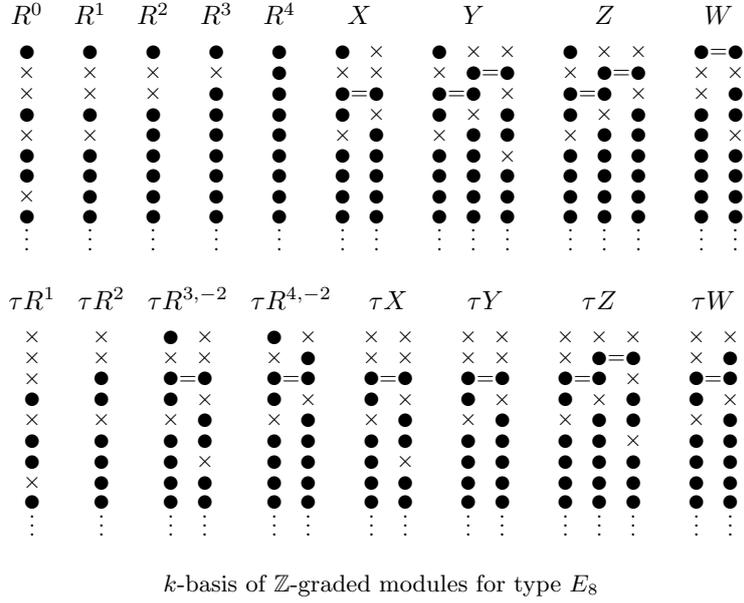
\[
\begin{gathered}
\begin{smallmatrix}
\mathclap{\mbox{\small $R^0$}} & \quad & \mathclap{\mbox{\small $R^1$}} & \quad & \mathclap{\mbox{\small $R^2$}} & \quad & \mathclap{\mbox{\small $R^3$}} & \quad & \mathclap{\mbox{\small $R^4$}} & \quad & & \mathclap{\mbox{\small $X$}} & & \quad & & & \mathclap{\mbox{\small $Y$}} & & & \quad & & & \mathclap{\mbox{\small $Z$}} & & & \quad & & \mathclap{\mbox{\small $W$}} & \\[6pt]
\bigbullet & \quad & \bigbullet & \quad & \bigbullet & \quad & \bigbullet & \quad & \bigbullet & \quad & \bigbullet & & \times & \quad & \bigbullet & & \times & & \times & \quad & \bigbullet & & \times & & \times & \quad & \bigbullet & \sequal & \bigbullet \\
\times & \quad & \times & \quad & \times & \quad & \times & \quad & \bigbullet & \quad & \times & & \times & \quad & \times & & \bigbullet & \sequal & \bigbullet & \quad & \times & & \bigbullet & \sequal & \bigbullet & \quad & \times & & \times \\
\times & \quad & \times & \quad & \times & \quad & \bigbullet & \quad & \bigbullet & \quad & \bigbullet & \sequal & \bigbullet & \quad & \bigbullet & \sequal & \bigbullet & & \times & \quad & \bigbullet & \sequal & \bigbullet & & \times & \quad & \times & & \bigbullet \\
\bigbullet & \quad & \bigbullet & \quad & \bigbullet & \quad & \bigbullet & \quad & \bigbullet & \quad & \bigbullet & & \times & \quad & \bigbullet & & \times & & \bigbullet & \quad & \bigbullet & & \times & & \bigbullet & \quad & \bigbullet & & \bigbullet \\
\times & \quad & \times & \quad & \bigbullet & \quad & \bigbullet & \quad & \bigbullet & \quad & \times & & \bigbullet & \quad & \times & & \bigbullet & & \bigbullet & \quad & \times & & \bigbullet & & \bigbullet & \quad & \bigbullet & & \times \\
\bigbullet & \quad & \bigbullet & \quad & \bigbullet & \quad & \bigbullet & \quad & \bigbullet & \quad & \bigbullet & & \bigbullet & \quad & \bigbullet & & \bigbullet & & \times & \quad & \bigbullet & & \bigbullet & & \bigbullet & \quad & \bigbullet & & \bigbullet \\
\bigbullet & \quad & \bigbullet & \quad & \bigbullet & \quad & \bigbullet & \quad & \bigbullet & \quad & \bigbullet & & \bigbullet & \quad & \bigbullet & & \bigbullet & & \bigbullet & \quad & \bigbullet & & \bigbullet & & \bigbullet & \quad & \bigbullet & & \bigbullet \\
\times & \quad & \bigbullet & \quad & \bigbullet & \quad & \bigbullet & \quad & \bigbullet & \quad & \bigbullet & & \bigbullet & \quad & \bigbullet & & \bigbullet & & \bigbullet & \quad & \bigbullet & & \bigbullet & & \bigbullet & \quad & \bigbullet & & \bigbullet \\
\bigbullet & \quad & \bigbullet & \quad & \bigbullet & \quad & \bigbullet & \quad & \bigbullet & \quad & \bigbullet & & \bigbullet & \quad & \bigbullet & & \bigbullet & & \bigbullet & \quad & \bigbullet & & \bigbullet & & \bigbullet & \quad & \bigbullet & & \bigbullet \\[-5pt]
\vdots & \quad & \vdots & \quad & \vdots & \quad & \vdots & \quad & \vdots & \quad & \vdots & & \vdots & \quad & \vdots & & \vdots & & \vdots & \quad & \vdots & & \vdots & & \vdots & \quad & \vdots & & \vdots
\end{smallmatrix} \\[10pt]
\begin{smallmatrix}
\mathclap{\mbox{\small $\tau R^1$}} & \quad \; & \mathclap{\mbox{\small $\tau R^2$}} & \quad \; & & \mathclap{\mbox{\small $\tau R^{3, -2}$}} & & \quad \; & & \mathclap{\mbox{\small $\tau R^{4, -2}$}} & & \quad & & \mathclap{\mbox{\small $\tau X$}} & & \quad & & \mathclap{\mbox{\small $\tau Y$}} & & \quad & & & \mathclap{\mbox{\small $\tau Z$}} & & & \quad & & \mathclap{\mbox{\small $\tau W$}} & \\[6pt]
\times & \quad \; & \times & \quad \; & \bigbullet & & \times & \quad \; & \bigbullet & & \times & \quad & \times & & \times & \quad & \times & & \times & \quad & \times & & \times & & \times & \quad & \times & & \times \\
\times & \quad \; & \times & \quad \; & \times & & \times & \quad \; & \times & & \bigbullet & \quad & \times & & \times & \quad & \times & & \times & \quad & \times & & \bigbullet & \sequal & \bigbullet & \quad & \times & & \bigbullet \\
\times & \quad \; & \bigbullet & \quad \; & \bigbullet & \sequal & \bigbullet & \quad \; & \bigbullet & \sequal & \bigbullet & \quad & \bigbullet & \sequal & \bigbullet & \quad & \bigbullet & \sequal & \bigbullet & \quad & \bigbullet & \sequal & \bigbullet & & \times & \quad & \bigbullet & \sequal & \bigbullet \\
\bigbullet & \quad \; & \bigbullet & \quad \; & \bigbullet & & \times & \quad \; & \bigbullet & & \times & \quad & \bigbullet & & \times & \quad & \bigbullet & & \times & \quad & \bigbullet & & \times & & \bigbullet & \quad & \bigbullet & & \times \\
\times & \quad \; & \times & \quad \; & \times & & \bigbullet & \quad \; & \times & & \bigbullet & \quad & \times & & \bigbullet & \quad & \times & & \bigbullet & \quad & \times & & \bigbullet & & \bigbullet & \quad & \times & & \bigbullet \\
\bigbullet & \quad \; & \bigbullet & \quad \; & \bigbullet & & \bigbullet & \quad \; & \bigbullet & & \bigbullet & \quad & \bigbullet & & \bigbullet & \quad & \bigbullet & & \bigbullet & \quad & \bigbullet & & \bigbullet & & \times & \quad & \bigbullet & & \bigbullet \\
\bigbullet & \quad \; & \bigbullet & \quad \; & \bigbullet & & \times & \quad \; & \bigbullet & & \bigbullet & \quad & \bigbullet & & \times & \quad & \bigbullet & & \bigbullet & \quad & \bigbullet & & \bigbullet & & \bigbullet & \quad & \bigbullet & & \bigbullet \\
\times & \quad \; & \bigbullet & \quad \; & \bigbullet & & \bigbullet & \quad \; & \bigbullet & & \bigbullet & \quad & \bigbullet & & \bigbullet & \quad & \bigbullet & & \bigbullet & \quad & \bigbullet & & \bigbullet & & \bigbullet & \quad & \bigbullet & & \bigbullet \\
\bigbullet & \quad \; & \bigbullet & \quad \; & \bigbullet & & \bigbullet & \quad \; & \bigbullet & & \bigbullet & \quad & \bigbullet & & \bigbullet & \quad & \bigbullet & & \bigbullet & \quad & \bigbullet & & \bigbullet & & \bigbullet & \quad & \bigbullet & & \bigbullet \\[-5pt]
\vdots & \quad \; & \vdots & \quad \; & \vdots & & \vdots & \quad \; & \vdots & & \vdots & \quad & \vdots & & \vdots & \quad & \vdots & & \vdots & \quad & \vdots & & \vdots & & \vdots & \quad & \vdots & & \vdots
\end{smallmatrix} \\[10pt]
\mbox{\footnotesize $k$-basis of $\Z$-graded modules for type $E_8$}
\end{gathered}
\]

\subsection{Countable Cohen--Macaulay representation type}

For type $A_\infty$ and $D_\infty$, the Auslander--Reiten quiver of the category $\CM\!_0 ^{\Z} R$ has precisely $m$ connected components $\mathrm{C}^j$, $j\in \Z/m\Z$. The component $\mathrm{C}^j$ is the image of $\mathrm{C}^0$ via the $j$th shift functor $(j)$.

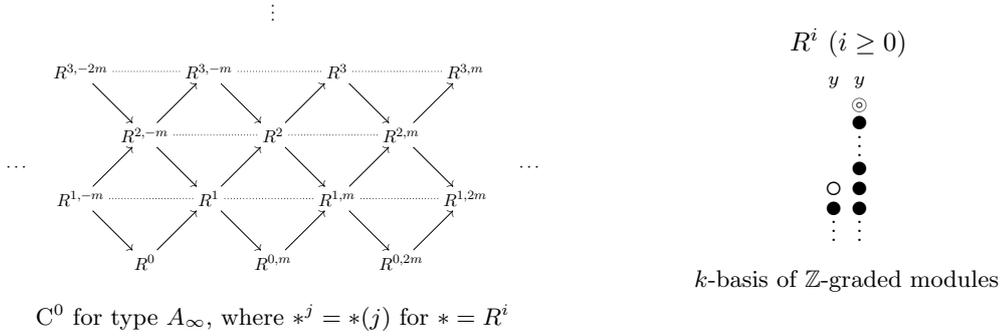
\begin{equation*}
\begin{gathered}
\begin{minipage}{0.67\textwidth}
  \centering
  \scalebox{0.6}{
  \begin{xy} 0;<20pt,0pt>:<0pt,20pt>::
  (6,3) *+{\cdots},
  (14,8) *+{\vdots},
  (22,3) *+{\cdots},
  (8,6) *+{R^{3, -2m}} ="40",
  (8,2) *+{R^{1, -m}} ="21",
  (10,4) *+{R^{2, -m}} ="31",
  (12,6) *+{R^{3, -m}} ="41",
  (10,0) *+{R^0} ="12",
  (12,2) *+{R^1} ="22",
  (14,4) *+{R^2} ="32",
  (16,6) *+{R^3} ="42",
  (14,0) *+{R^{0, m}} ="13",
  (16,2) *+{R^{1, m}} ="23",
  (18,4) *+{R^{2, m}} ="33",
  (20,6) *+{R^{3, m}} ="43",
  (18,0) *+{R^{0, 2m}} ="14",
  (20,2) *+{R^{1, 2m}} ="24",
  "40", {\ar"31"},
  "40", {\ar@{.} "41"},
  "21", {\ar"31"},
  "31", {\ar"41"},
  "21", {\ar"12"},
  "31", {\ar"22"},
  "41", {\ar"32"},
  "21", {\ar@{.} "22"},
  "31", {\ar@{.} "32"},
  "41", {\ar@{.} "42"},
  "12", {\ar"22"},
  "22", {\ar"32"},
  "32", {\ar"42"},
  "22", {\ar"13"},
  "32", {\ar"23"},
  "42", {\ar"33"},
  "22", {\ar@{.} "23"},
  "32", {\ar@{.} "33"},
  "42", {\ar@{.} "43"},
  "13", {\ar"23"},
  "23", {\ar"33"},
  "33", {\ar"43"},
  "23", {\ar"14"},
  "33", {\ar"24"},
  "23", {\ar@{.} "24"},
  "14", {\ar"24"},
  \end{xy}
  }
\\[10pt]
\mbox{\footnotesize $\mathrm{C}^0$ for type $A_\infty$, where $*^j=*(j)$ for $*=R^i$}
\end{minipage}
\hfill
\begin{minipage}{0.33\textwidth}
  \centering
$\begin{smallmatrix}
\mathclap{\mbox{\quad \small $R^i$ ($i\geq 0$)}} & \\[6pt]
\mbox{\tiny $y$} & \mbox{\tiny $y$} \\[2pt]
& \circledcirc \\
& \bigbullet \\[-5pt]
& \vdots \\
& \bigbullet \\
\bigcirc & \bigbullet \\
\bigbullet & \bigbullet \\[-5pt]
\vdots & \vdots
\end{smallmatrix}$\\[10pt]
\mbox{\footnotesize $k$-basis of $\Z$-graded modules}
\end{minipage}
\end{gathered}
\end{equation*}

\begin{equation*}
\begin{gathered}
\begin{minipage}{0.7\textwidth}
  \centering

\end{minipage}
\hfill
\begin{minipage}{0.2\textwidth}
  \centering

\end{minipage}
\end{gathered}
\end{equation*}

The symbols $\bigcirc$ and $\circledcirc$ mean $1$ respectively $x^{-i}y$ in the bases. Each black circle $\bullet$ means the symbol in the previous row multiplied by $x$. The standard silting object is
\[
V=(\bigoplus_{i=1}^{\lceil \frac{l}{m}\rceil-1}\bigoplus_{j=1}^m R^{i, j-m})\oplus(\bigoplus_{j=1}^{l+m-\lceil \frac{l}{m}\rceil m}R^{\lceil \frac{l}{m}\rceil, j-m})\: .
\]

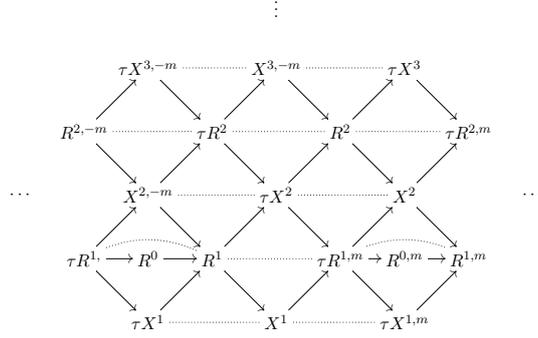
\begin{equation*}
\begin{gathered}
\scalebox{0.6}{
\begin{xy} 0;<20pt,0pt>:<0pt,20pt>::
(6,4) *+{\cdots},
(14,10) *+{\vdots},
(22,4) *+{\cdots},
(8,6) *+{R^{2, -m}} ="40",
(10,8) *+{\tau X^{3, -m}} ="50",
(8,2) *+{\tau R^{1, }} ="21",
(10,4) *+{X^{2, -m}} ="31",
(12,6) *+{\tau R^2} ="41",
(14,8) *+{X^{3, -m}} ="51",
(10,2) *+{R^0} ="61",
(10,0) *+{\tau X^1} ="12",
(12,2) *+{R^1} ="22",
(14,4) *+{\tau X^2} ="32",
(16,6) *+{R^2} ="42",
(18,8) *+{\tau X^3} ="52",
(14,0) *+{X^1} ="13",
(16,2) *+{\tau R^{1, m}} ="23",
(18,4) *+{X^2} ="33",
(20,6) *+{\tau R^{2, m}} ="43",
(18,2) *+{R^{0, m}} ="63",
(18,0) *+{\tau X^{1, m}} ="14",
(20,2) *+{R^{1, m}} ="24",
"40", {\ar"50"},
"40", {\ar"31"},
"50", {\ar"41"},
"40", {\ar@{.} "41"},
"50", {\ar@{.} "51"},
"21", {\ar"31"},
"31", {\ar"41"},
"41", {\ar"51"},
"21", {\ar"61"},
"21", {\ar"12"},
"31", {\ar"22"},
"41", {\ar"32"},
"51", {\ar"42"},
"61", {\ar"22"},
"21", {\ar@/^1pc/@{.} "22"},
"31", {\ar@{.} "32"},
"41", {\ar@{.} "42"},
"51", {\ar@{.} "52"},
"12", {\ar"22"},
"22", {\ar"32"},
"32", {\ar"42"},
"42", {\ar"52"},
"22", {\ar"13"},
"32", {\ar"23"},
"42", {\ar"33"},
"52", {\ar"43"},
"12", {\ar@{.} "13"},
"22", {\ar@{.} "23"},
"32", {\ar@{.} "33"},
"42", {\ar@{.} "43"},
"13", {\ar"23"},
"23", {\ar"33"},
"33", {\ar"43"},
"23", {\ar"63"},
"23", {\ar"14"},
"33", {\ar"24"},
"63", {\ar"24"},
"13", {\ar@{.} "14"},
"23", {\ar@/^1pc/@{.} "24"},
"14", {\ar"24"},
\end{xy}
}\\[10pt]
\mbox{\footnotesize $\mathrm{C}^0$ for type $D_\infty$, where $*^j=*(j)$ for $*=R^i$, $X^i$, $\tau R^i$, or $\tau X^i$}
\end{gathered}
\end{equation*}
\vskip1em
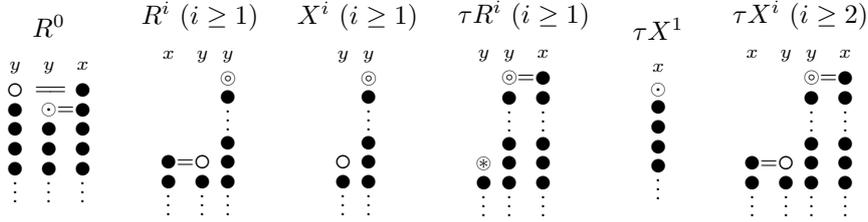
\[
\begin{gathered}
\begin{smallmatrix}
 & & \mathclap{\mbox{\small $R^0$}} & & \\[6pt]
\mbox{\tiny $y$} & & \mbox{\tiny $y$} & & \mbox{\tiny $x$} \\[2pt]
\bigcirc & & \mathllap{=\!}\mathrlap{=} & & \bigbullet \\
\bigbullet & & \smallodot & \sequal & \bigbullet \\
\bigbullet & & \bigbullet & & \bigbullet \\
\bigbullet & & \bigbullet & & \bigbullet \\
\bigbullet & & \bigbullet & & \bigbullet \\[-5pt]
\vdots & & \vdots & & \vdots
\end{smallmatrix}
\qquad
\begin{smallmatrix}
 & \mathclap{\mbox{\quad \small $R^i$ ($i\geq 1$)}} & & \\[6pt]
\mbox{\tiny $x$} & & \mbox{\tiny $y$} & \mbox{\tiny $y$} \\[2pt]
 & & & \circledcirc \\
 & & & \bigbullet \\[-5pt]
 & & & \vdots \\
 & & & \bigbullet \\
\bigbullet & \sequal & \bigcirc & \bigbullet \\
\bigbullet & & \bigbullet & \bigbullet \\[-5pt]
\vdots & & \vdots & \vdots
\end{smallmatrix}
\qquad \quad
\begin{smallmatrix}
\mathclap{\mbox{\quad \small $X^i$ ($i\geq 1$)}} & \\[6pt]
\mbox{\tiny $y$} & \mbox{\tiny $y$} \\[2pt]
 & \circledcirc \\
 & \bigbullet \\[-5pt]
 & \vdots \\
 & \bigbullet \\
\bigcirc & \bigbullet \\
\bigbullet & \bigbullet \\[-5pt]
\vdots & \vdots
\end{smallmatrix}
\qquad \quad
\begin{smallmatrix}
 & \mathclap{\mbox{\quad \small $\tau R^i$ ($i\geq 1$)}} & & \\[6pt]
\mbox{\tiny $y$} & \mbox{\tiny $y$} & & \mbox{\tiny $x$} \\[2pt]
 & \circledcirc & \sequal & \bigbullet \\
 & \bigbullet & & \bigbullet \\[-5pt]
 & \vdots & & \vdots \\
 & \bigbullet & & \bigbullet \\
\smallcircledast & \bigbullet & & \bigbullet \\
\bigbullet & \bigbullet & & \bigbullet \\[-5pt]
\vdots & \vdots & & \vdots
\end{smallmatrix}
\qquad \quad
\begin{smallmatrix}
\mathclap{\mbox{\small $\tau X^1$}} & \\[6pt]
\mbox{\tiny $x$} & \\[2pt]
\smallodot \\
\bigbullet \\
\bigbullet \\
\bigbullet \\
\bigbullet \\[-5pt]
\vdots
\end{smallmatrix}
\qquad
\begin{smallmatrix}
 & & \mathclap{\mbox{\quad \small $\tau X^i$ ($i\geq 2$)}} & & & \\[6pt]
\mbox{\tiny $x$} & & \mbox{\tiny $y$} & \mbox{\tiny $y$} & & \mbox{\tiny $x$} \\[2pt]
 & & & \circledcirc & \sequal & \bigbullet \\
 & & & \bigbullet & & \bigbullet \\[-5pt]
 & & & \vdots & & \vdots \\
 & & & \bigbullet & & \bigbullet \\
\bigbullet & \sequal & \bigcirc & \bigbullet & & \bigbullet \\
\bigbullet & & \bigbullet & \bigbullet & & \bigbullet \\[-5pt]
\vdots & & \vdots & \vdots & & \vdots
\end{smallmatrix} \\[10pt]
\mbox{\footnotesize $k$-basis of $\Z$-graded modules for type $D_\infty$}
\end{gathered}
\]

The symbols $\circ$, $\circledast$, $\odot$, and $\circledcirc$ mean $1$, $x$, $y$, and $x^{1-i}y$ in the bases, respectively. In a column labelled by $x$ (respectively, $y$), each black circle $\bullet$ means the symbol in the previous row multiplied by $y$ (respectively, $x$). The standard silting object is

\begin{align*}
V=& (\bigoplus_{j=1}^m \tau R^{1, j})\oplus(\bigoplus_{i=2}^{\lceil \frac{l}{m}\rceil-1}\bigoplus_{j=1}^m \tau R^{i, j})\oplus(\bigoplus_{j=1}^{l+m-\lceil \frac{l}{m}\rceil m}\tau R^{\lceil \frac{l}{m}\rceil, j}) \\
 & \oplus(\bigoplus_{j=1}^l \tau X^{1, j})\oplus(\bigoplus_{j=1}^m X^{\lceil \frac{l+j}{m}\rceil +1, l-m+j-\lceil \frac{l+j}{m}\rceil m}) \: .
\end{align*}

\appendix

\section{Dg resolutions of self-injective Nakayama algebras} \label{appendix:dg resolutions of self-injective Nakayama algebras}

\subsection{$A_\infty$-algebras} \label{ss:Ainfty-algebras}

Following Section~3.1 of \cite{Keller01}, an {\it $A_\infty$-algebra} is a $\Z$-graded $k$-vector space $A$ endowed with $k$-linear maps $m_n\colon A^{\ten_k n}\to A$ of degree $2-n$, $n\geq 1$, satisfying
\[
\sum_{r+s+t=n}(-1)^{r+st}m_{r+1+t}\circ (\id^{\ten r}\ten m_s \ten \id^{\ten t})=0
\]
for all positive integers $n$, where $r$ and $t$ run through the non-negative integers and $s$ through the positive integers. If $m_n$ vanishes for all $n>2$, then $d=m_1$ and $m_2$ make
$A$ into a dg algebra.

\subsection{Cofibrant dg replacements of self-injective Nakayama algebras} \label{ss:cofibrant dg replacements of self-injective Nakayama algebras}

\begin{proposition} \label{prop:cofibrant dg replacement of Nakayama algebras}
Let $n\geq 1$ and $n_x>1$ be integers. Let $C_n$ be an $n$-cycle whose arrows are labelled by $x$ and $N_{n, n_x}=kQ/(x^{n_x})$ the associated self-injective Nakayama algebra. Then $N_{n, n_x}$ is quasi-isomorphic to the cofibrant dg algebra $(kQ, d)$ defined in Definition~\ref{def:dg path algebras}.
\end{proposition}

\begin{proof}
Assume that the set of the vertices of $C_n$ is $\Z/n\Z$ and arrows are from $i$ to $i+m$ for an integer $m\leq n$ such that $m$ and $n$ are coprime. Denote the direct sum $\bigoplus_{i\in \Z/n\Z} S_i$ of simple $N_{n, n_x}\op$-modules by $S$. By Theorem~4.8 of \cite{Tamaroff21}, the self-injective Nakayama algebra $N_{n, n_x}$ is quasi-isomorphic to the dg tensor algebra
\[
A=\mathrm{T}(\bigoplus_{p=1}^{\infty}D\Ext_{N_{n, n_x}}^p(S, S)(-1)) \: .
\]
Define the permutations $\nu_p$ on $\Z/n\Z$ as in Definition~\ref{def:dg path algebras} for all non-negative integers $p$. By Theorem~2.2 of \cite{DotsenkoVallette13}, \cf~also Theorem~6.5 of \cite{HeLu05}, the $A_{\infty}$-algebra $\Ext_{N_{n, n_x}}^*(S, S)$ has a $k$-basis
\[
(\alpha_{p, i}\in \Ext_{N_{n, n_x}}^p (S_{\nu_p (i)}, S_i))
\]
such that the $A_{\infty}$-structure is determined by
\begin{align*}
m_2(\alpha_{p_1, i}, \alpha_{p_2, \nu_{p_1}(i)}) & =\alpha_{p_1 +p_2 , i}\ko \mbox{if either }p_1\mbox{ or }p_2\mbox{ is even}, \\
m_{n_x}(\alpha_{p_1, i}, \ldots, \alpha_{p_{n_x}, (\nu_{p_{n_x-1}} \circ \cdots \circ \nu_{p_1})(i)}) & =\alpha_{p_1+\cdots +p_{n_x}+2-n_x, i}\ko \mbox{if all }p_j\mbox{ are odd}, \\
\mbox{other }A_{\infty} & \mbox{-multiplications vanish}.
\end{align*}
The differential of the dg tensor algebra $A$ is determined by its restriction to
\[
\bigoplus_{p=1}^{\infty}D\Ext_{N_{n, n_x}}^p(S, S)(-1) \: .
\]
It is given by
\[
(\si^{-1}\ten \si^{-1})\circ Dm_2 \circ \si +(\si^{-1}\ten \cdots \ten \si^{-1})\circ Dm_{n_x} \circ \si \: ,
\]
where $\si \colon V \to V(1)$ denotes the canonical map of degree $-1$. If we choose
\[
(\beta_{p, i}\in D\Ext_{N_{n, n_x}}^p (S_{\nu_p (i)}, S_i)(-1))
\]
as the $k$-basis $(-1)$-shifted dual to $(\alpha_{p, i}\in \Ext_{N_{n, n_x}}^p (S_{\nu_p (i)}, S_i))$, then the differentials of $\beta_{p, i}$ coincide with those in Definition~\ref{def:dg path algebras}. This concludes the proof.
\end{proof}



\def\cprime{$'$} \def\cprime{$'$}
\providecommand{\bysame}{\leavevmode\hbox to3em{\hrulefill}\thinspace}
\providecommand{\MR}{\relax\ifhmode\unskip\space\fi MR }
\providecommand{\MRhref}[2]{%
  \href{http://www.ams.org/mathscinet-getitem?mr=#1}{#2}
}
\providecommand{\href}[2]{#2}

\end{document}